\documentclass[onefignum,onetabnum]{siamart190516}

\usepackage{lipsum}
\usepackage{amsfonts}
\usepackage{amsmath}
\usepackage{amssymb}
\usepackage[mathscr]{euscript}
\usepackage{latexsym}

\usepackage{algorithmic}
\usepackage{enumitem}

\usepackage{graphicx}
\usepackage{epstopdf}
\usepackage{tikz}
\usepackage{onimage}
\usepackage[caption=false]{subfig}
\ifpdf
  \DeclareGraphicsExtensions{.eps,.pdf,.png,.jpg}
\else
  \DeclareGraphicsExtensions{.eps}
\fi

\newsiamremark{remark}{Remark}
\newsiamremark{example}{Example}
\newsiamremark{hypothesis}{Hypothesis}
\crefname{hypothesis}{Hypothesis}{Hypotheses}
\newsiamthm{claim}{Claim}
\newsiamthm{assumption}{Assumption}

\headers{Optimizing Reduced-Order Models}{S. E. Otto, A. Padovan, and C. W. Rowley}

\title{Optimizing Oblique Projections for Nonlinear Systems using Trajectories \thanks{Compiled \today.
\funding{This research was supported by the Army Research Office under grant
  number W911NF-17-1-0512 and the Air Force Office of Scientific Research under
  grant number FA9550-19-1-0005.  S.E.O. was supported by the National Science Foundation Graduate Research Fellowship Program under Grant No. DGE-2039656.}}}

\author{
Samuel E. Otto\thanks{Princeton University, Princeton, NJ, Dept. of Mechanical and Aerospace Engineering (S.E.O.: \email{sotto@princeton.edu}, A.P.: \email{apadovan@princeton.edu}, C.W.R.: \email{cwrowley@princeton.edu})}
\and Alberto Padovan\footnotemark[2]
\and Clarence W. Rowley\footnotemark[2]
}

\usepackage{amsopn}

\DeclareMathOperator{\avg}{avg}
\DeclareMathOperator{\sgn}{sgn}

\DeclareMathOperator{\Tr}{Tr}
\DeclareMathOperator{\rank}{rank}

\DeclareMathOperator{\Range}{Range}

\DeclareMathOperator{\grad}{grad}
\DeclareMathOperator{\td}{\mathrm{d}}
\DeclareMathOperator{\ddt}{\frac{\td}{\td t}}
\DeclareMathOperator{\ddtsq}{\frac{\td^2}{\td t^2}}
\DeclareMathOperator{\D}{\mathrm{D}}

\makeatletter
\renewcommand*\env@matrix[1][*\c@MaxMatrixCols c]{%
  \hskip -\arraycolsep
  \let\@ifnextchar\new@ifnextchar
  \array{#1}}
\makeatother

\ifpdf
\hypersetup{
  pdftitle={Optimizing Oblique Projections for Nonlinear Systems using Trajectories},
  pdfauthor={S. E. Otto, A. Padovan, and C. W. Rowley}
}
\fi

\usepackage{tikz,pgfplots}

\begin{document}

\maketitle

\begin{abstract}
Reduced-order modeling techniques, including balanced truncation and
$\mathcal{H}_2$-optimal model reduction, exploit the structure of linear dynamical systems to produce models that accurately capture the dynamics.
For nonlinear systems operating far away from equilibria, on the other hand, current approaches seek low-dimensional representations of the state that often neglect low-energy features that have high dynamical significance.
For instance, low-energy features are known to play an important role in fluid
dynamics where they can be a driving mechanism for shear-layer instabilities.
Neglecting these features leads to models with poor predictive accuracy despite being able to accurately encode and decode states.
In order to improve predictive accuracy, we propose to optimize the reduced-order model to fit a collection of coarsely sampled trajectories from the original system.
In particular, we optimize over the product of two Grassmann manifolds defining Petrov-Galerkin projections of the full-order governing equations.
We compare our approach with existing methods including proper orthogonal
decomposition, balanced truncation-based Petrov-Galerkin projection, quadratic-bilinear balanced truncation, and the quadratic-bilinear iterative rational Krylov algorithm. 
Our approach demonstrates significantly improved accuracy both on a nonlinear toy model and on
an incompressible (nonlinear) axisymmetric jet flow with $10^5$ states.
\end{abstract}



\section{Introduction}
Accurate low-dimensional models of physical processes enable a variety of important scientific and engineering tasks to be carried out.
However, many real-world systems like complex fluid flows in the atmosphere as well as around and inside aircraft are governed by extremely high-dimensional nonlinear systems --- properties that make tasks like real-time forecasting, state estimation, and control computationally prohibitive using the original governing equations.
Fortunately, the behavior of these systems is frequently dominated by coherent structures and patterns \cite{Brown1974density} that may be modeled with equations whose dimension is much smaller \cite{Sirovich1987turbulence, Holmes2012turbulence}.
The goal of ``reduced-order modeling'' is to obtain simplified models that are suitable for forecasting, estimation, and control from the vastly more complicated governing equations provided by physics.
For reviews of modern techniques, see \cite{Baur2014model}, \cite{Benner2015survey} and \cite{Rowley2017model}.
For a striking display of coherent structures in turbulence, see the shadowgraphs in G. L. Brown and A. Roshko \cite{Brown1974density}.

When the system of interest is operating close to an equilibrium point, the
governing equations are accurately approximated by their linearization about the equilibrium.
In this case, a variety of sophisticated and effective reduced-order modeling techniques can be applied with guarantees on the accuracy of the resulting low-dimensional model \cite{Antoulas2005approximation, Benner2015survey}.
Put simply, linearity provides an elegant and complete characterization of the system's trajectories in response to inputs, disturbances, and initial conditions that can be exploited to build simplified models whose trajectories closely approximate the ones from the original system.
For instance, the balanced truncation method introduced by B. Moore \cite{Moore1981principal} yields a low-dimensional projection of the original system that simultaneously retains the most observable and controllable states of the system and provides bounds on various measures of reduced-order model error \cite{Antoulas2005approximation}.
A computationally efficient approximation called Balanced Proper Orthogonal Decomposition (BPOD) \cite{Rowley2005model} is suitable for high-dimensional fluid flow applications.
Another approach is to find a reduced-order model (ROM) that is as close as possible to a stable full-order model (FOM) with respect to the $\mathcal{H}_2$ norm.
Algorithms like the Iterative Rational Krylov Algorithm (IRKA) \cite{Gugercin2008H2} are based on satisfying necessary conditions for $\mathcal{H}_2$-optimality.

Various generalizations of linear model reduction techniques have also been developed for bilinear \cite{Baur2014model, Benner2012interpolation, Flagg2015multipoint}, quadratic bilinear \cite{Benner2015two, Benner2018H2, Benner2017balanced}, and lifted nonlinear systems \cite{Kramer2019balanced} based on truncated Volterra series expansion of the output.
These methods extend the region of validity for reduced-order models about stable equilibria, yet still suffer as high-order nonlinearities become dominant far away from an equilibrium.
These techniques also require solutions of large-scale Sylvester or Lyapunov equations, making them difficult to apply to fluid flows whose state dimensions can easily exceed $10^5$.



One commonality among the above model reduction approaches is that they utilize oblique projections to retain coordinates or ``features'' with high variance or ``energy'' as well as any coordinates with low variance that significantly influence the dynamics at future times \cite{Benner2015survey,Rowley2005model}.
These small, but dynamically significant features are known to play an important role in driving the growth of instabilities
in ``shear flows'' such as mixing layers and jets.
Linearizations of these shear flows often result in non-normal systems, which
can exhibit large transient growth in response to low-energy perturbations~\cite{Trefethen-93,Schmid2001stability}.
Some successful approaches \cite{Barbagallo2009closed, Ahuja2010feedback,
  Ilak2010model, Illingworth2011feedback} have involved oblique projections of
the nonlinear dynamics onto subspaces identified from the dynamics linearized about an equilibrium.
However, this approach is often not satisfactory since the linearized dynamics
become inaccurate as the state moves away from the equilibrium and nonlinear effects become significant.
In this paper we illustrate how such nonlinear effects can cause reduced-order models obtained using the various approaches described above to perform poorly,
for instance on a simple three-dimensional system as well as on a high-dimensional axisymmetric jet flow.


When dealing with nonlinear systems operating far away from equilibria, nonlinear model reduction approaches tend to follow a two-step process: first identify a set, typically a smooth manifold or a subspace, near which the state of the system is known to lie, then model the dynamics in this set either by a projection of the governing equations or by a black-box data-driven approach.
The most common approach to identify a candidate subspace is Proper Orthogonal Decomposition (POD), whose application to the study of complex fluid flows was pioneered by J. L. Lumley \cite{Lumley1967structure}.
The dynamics may also be projected onto nonlinear manifolds using ``nonlinear Galerkin'' methods \cite{Marion1989nonlinear, Rega2005dimension}.
Recently, more sophisticated manifold learning techniques like deep convolutional autoencoders have also been used \cite{Lee2020model}.
The main obstacle encountered by the manifold-learning and POD-based approaches is that they neglect coordinates with low variance even if they are important for correctly forecasting the system's dynamics.
For instance, in our jet flow example, we find that a model with
$50$ POD modes capturing $99.6\%$ of the state's variance still yields poor
predictions that diverge from the full-order model.

In order to identify and retain the dynamically important coordinates while remaining tractable
for very large-scale systems like fluid flows, we shall optimize an oblique
projection operator to minimize the prediction error of the corresponding reduced-order model on a collection of sampled trajectories.
In this framework, oblique projection operators of a fixed dimension are identified with pairs of subspaces in Grassmann manifolds that meet a transversality condition.
We show that the pairs of subspaces defining oblique projection operators are open, dense, and connected in the product of Grassmann manifolds, and we prove that solutions of our optimization problem exist when it is appropriately regularized. 
Optimization is carried out using the Riemannian conjugate gradient algorithm introduced by H. Sato \cite{Sato2016dai} with formulas for the exponential map and parallel translation along geodesics given by A. Edelman et al. \cite{Edelman1998geometry}.
We provide mild conditions under which the algorithm is guaranteed to converge to a locally optimal oblique projection operator.

Related techniques based on optimizing projection subspaces have been used to produce $\mathcal{H}_2$-optimal reduced-order models for linear and bilinear systems.
Most approaches focus on optimizing orthogonal projection operators over a single Grassmann manifold \cite{Xu2013fast, Sato2015riemannian, Jiang2020riemannian} or an orthogonal Stiefel manifold \cite{Yan1999approximate, Sato2015riemannian,  Wang2018H2optimal, Yang2019trust, xu2019unconstrained}.
On the other hand, an alternating minimization technique over the two Grassmann manifolds defining an oblique projection is proposed by T. Zeng and C. Lu \cite{Zeng2015twosided} for $\mathcal{H}_2$-optimal reduction of linear systems.
For systems with quadratic nonlinearities, Y.-L. Jiang and K.-L. Xu \cite{Jiang2020riemannian} present an approach to optimize orthogonal projection operators based on the same truncated generalization of the $\mathcal{H}_2$ norm used by P. Benner et al.~\cite{Benner2018H2}.
Our approach differs from the ones mentioned above in that it may be used to find optimal reduced-order models based on oblique projections for general nonlinear high-dimensional systems based on sampled trajectories.


\section{Projection-Based Reduced-Order Models}
\label{sec:projection_based_ROM}
Consider a physical process, modeled by an input-output dynamical system
\begin{equation}
\begin{split}
    \ddt x &= f(x, u), \qquad x(t_0) = x_0, \\
    y &= g(x),
\end{split}
\label{eqn:full_order_model}
\end{equation}
with state $x\in \mathbb{R}^n$, input $u\in\mathbb{R}^d$, and output $y$ in $\mathbb{R}^m$, each space being equipped with the Euclidean inner product.
We shall often refer to \cref{eqn:full_order_model} as the full-order model (FOM).
Our goal is to use one or more discrete-time histories of observations $y_l = y(t_l)$ at sample times $t_0<\cdots<t_{L-1}$ in order to learn the key dynamical features of \cref{eqn:full_order_model} and produce a reduced-order model (ROM) that captures these effects.
Throughout the paper we assume that
\begin{assumption}
\label{asmpn:FOM_is_C2}
The functions $(x,t)\mapsto f(x,u(t))$ and $x \mapsto g(x)$ in \cref{eqn:full_order_model} have continuous partial derivatives with respect to $x$ up to second-order.
\end{assumption}

We shall use our observation data to learn an $r$-dimensional subspace $V$ of $\mathbb{R}^n$ in which to represent the state of the system \cref{eqn:full_order_model}.
Since $f(x, u)$ might not lie in $V$ when $x\in V$, we shall also find another $r$-dimensional subspace $W$ of $\mathbb{R}^n$ 
with $V\oplus W^{\perp} = \mathbb{R}^n$ in order to construct an oblique projection operator $P_{V,W}:\mathbb{R}^n\to \mathbb{R}^n$ satisfying
\begin{equation}
    P_{V,W} x \in V \quad \mbox{and} \quad
    x - P_{V,W} x \in W^{\perp} \qquad \forall x\in\mathbb{R}^n.
    \label{eqn:projection_defn}
\end{equation}
Every rank-$r$ oblique projection operator can be identified with a pair of subspaces $(V,W)$ satisfying $V\oplus W^{\perp} = \mathbb{R}^n$ (see Section~5.9 of C. D. Meyer \cite{Meyer2000matrix}), and we denote the set of such subspaces by $\mathcal{P}$.
Moreover, if $\Phi, \Psi \in \mathbb{R}^{n\times r}$ are matrices with $V=\Range\Phi$ and $W=\Range\Psi$, then it is easily shown that $(V,W)\in\mathcal{P}$ if and only if $\det(\Psi^T\Phi) \neq 0$, and the corresponding projection operator is given explicitly by 
\begin{equation}
        P_{V,W} = \Phi (\Psi^T \Phi)^{-1}\Psi^T.
        \label{eqn:oblique_projection_operator}
\end{equation}

Applying the projection defined by $(V,W)\in\mathcal{P}$ to the full-order model \cref{eqn:full_order_model}, we obtain a Petrov-Galerkin reduced-order model whose state $\hat{x}\in V$ evolves according to
\begin{equation}
    \ddt  \hat{x} = P_{V,W} f(\hat{x}, u), \qquad \hat{x}(0) = P_{V,W} x_0,
    \label{eqn:reduced_order_model}
\end{equation}
with observations given by $\hat{y} = g(\hat{x})$.
The two subspaces $V, W$ uniquely define the projection $P_{V,W}$ and the reduced-order model \cref{eqn:reduced_order_model}.
With the initial condition $x_0$ and input signal $u$ fixed, the output of the reduced-order model at each sample time $\hat{y}_l(V,W) = \hat{y}(t_l; (V,W))$ is a function of the chosen subspaces $V,W$.

Let $L_y:\mathbb{R}^m\to [0,+\infty)$ be a smooth penalty function for the difference between each observation $y_l$ and the model's prediction $\hat{y}_l(V,W)$.
Let us also introduce a smooth nonnegative-valued function $\rho (V,W)$, to be defined precisely in \cref{sec:optimization_domain}, that will serve as regularization by
preventing minimizing sequences of subspaces $(V,W)$ from approaching points outside the set $\mathcal{P}$ in which valid Petrov-Galerkin projections can be defined.
Using this regularization with a weight $\gamma > 0$ allows us to seek a minimum of the cost defined by
\begin{equation}
    J(V, W) = \frac{1}{L}\sum_{l=0}^{L-1} L_y\left( \hat{y}_l(V,W) - y_l \right) + \gamma \rho(V,W)
    \label{eqn:ROM_cost}
\end{equation}
over all pairs of $r$-dimensional subspaces $(V, W)$, subject to the reduced-order dynamics \cref{eqn:reduced_order_model}.
Here we shall consider the case when there is a single trajectory generated from
a known initial condition since it will be easy to handle multiple trajectories
from multiple known initial conditions once we understand the single trajectory
case.
The cost function \cref{eqn:ROM_cost} defines an optimization problem, and in
the following section we define a suitable regularization function~$\rho$ and develop a technique for iteratively solving this problem.
We refer to this approach for constructing reduced-order models as Trajectory-based Optimization for Oblique Projections (TrOOP).

\begin{remark}[Integrated objectives and $\mathcal{H}_2$-optimal model reduction]
We may also optimize a cost function where the sum in \cref{eqn:ROM_cost} is replaced by an integral approximated using numerical quadrature;
the details are given in \cref{subapp:integrated_model_error}.
When the full-order model \cref{eqn:full_order_model} is a stable linear-time-invariant system and the trajectories $y(t)$ are generated by unit impulse responses from each input channel, then the $\mathcal{H}_2$ norm \cite{Antoulas2005approximation} of the difference between the reduced-order model and the full-order model can be written as a sum of integrated objectives $\int_0^{\infty} \Vert \hat{y}(t;(V,W)) - y(t) \Vert_2^2 \td t$.
After approximating these integrals by integrals over finite time-horizons, we may employ the technique described in \cref{subapp:integrated_model_error} for $\mathcal{H}_2$-optimal model reduction.
\end{remark}

\section{Optimization Domain, Representatives, and Regularization}
\label{sec:optimization_domain}
The set containing all $r$-dimensional subspaces of $\mathbb{R}^n$ can be endowed with the structure of a compact Riemannian manifold called the Grassmann manifold, which has dimension $nr - r^2$ and is denoted $\mathcal{G}_{n,r}$.
Therefore, our optimization problem entails minimizing the cost given by \cref{eqn:ROM_cost} over the subset $\mathcal{P}$ of the product manifold $\mathcal{M} = \mathcal{G}_{n,r}\times \mathcal{G}_{n,r}$ on which oblique projection operators are defined.
The goal of this section will be to characterize the topology of the set $\mathcal{P}$ and to introduce an appropriate regularization function $\rho$ so that we may instead consider the unconstrained minimization of \cref{eqn:ROM_cost} over $\mathcal{M}$.
We also describe how to work with matrix representatives of the relevant subspaces that can be stored in a computer.

\subsection{Grassmann Manifold and Representatives of Subspaces}
First we describe some basic properties of the Grassmann manifold that can be found in \cite{Absil2004Riemannian, Absil2009optimization, Bendokat2020grassmann}.
If $\mathbb{R}_*^{n,r}$ denotes the smooth manifold of $n\times r$ matrices with linearly independent columns, then $\mathcal{G}_{n,r}$ can be identified with the quotient manifold of $\mathbb{R}_*^{n,r}$ under the action of the general linear group $GL_r$ defining changes of basis $GL_r\times\mathbb{R}_*^{n,r}\to\mathbb{R}_*^{n,r}:(M,X) \mapsto XM$.
Since this group action is free and proper, the quotient manifold theorem (Theorem 21.10 in \cite{Lee2013introduction}) ensures that $\mathbb{R}_*^{n,r}/GL_r$ is a smooth manifold and the quotient map sending $X\in \mathbb{R}_*^{n,r}$ to its equivalence class in $\mathbb{R}_*^{n,r}/GL_r$,
\begin{equation}
    [X] = \left\{ Y\in \mathbb{R}_*^{n,r}\ : \ \mbox{$Y=XM$ for some $M\in GL_r$ } \right\},
\end{equation}
is a smooth submersion.
Each subspace $\Range{X}\in\mathcal{G}_{n,r}$ is identified with the equivalence class $[X]\in\mathbb{R}_*^{n,r}/GL_r$.

In order to optimize the pairs of abstract subspaces $(V,W) \in \mathcal{M} = \mathcal{G}_{n,r}\times \mathcal{G}_{n,r}$ defining oblique projections, we work with matrix representative of these subspaces in the so-called ``structure space'' $\bar{\mathcal{M}} = \mathbb{R}_*^{n,r}\times \mathbb{R}_*^{n,r}$.
A pair of matrices $(\Phi, \Psi) \in \bar{\mathcal{M}}$ are representatives of $(V,W)\in \mathcal{M}$ if $V=\Range{\Phi}$ and $W=\Range{\Psi}$.
The ``canonical projection'' map $\pi : \bar{\mathcal{M}}\to\mathcal{M}$ is defined by
\begin{equation}
    \pi: (\Phi,\Psi) \mapsto (\Range{\Phi}, \Range{\Psi}),
\end{equation}
and it is clear that the set of all representatives of $(V,W)\in \mathcal{M}$ is given by the pre-image set $\pi^{-1}(V,W)$.
The canonical projection map is a surjective submersion since its component maps $\Phi \mapsto \Range{\Phi}$ and $\Psi \mapsto \Range{\Psi}$ are surjective submersions.
Consequently, Theorem~4.29 in J. M. Lee \cite{Lee2013introduction} provides the useful result that a function $F:\mathcal{M} \to \mathcal{N}$, with $\mathcal{N}$ another smooth manifold, is smooth if and only if $F\circ\pi$ is smooth.

Suppose that $(V,W)\in\mathcal{P}$ are a pair of subspaces that define an oblique projection and $(\Phi, \Psi) \in \pi^{-1}(V,W)$ are a choice of representatives.
We observe that the oblique projection operator given explicitly by \cref{eqn:oblique_projection_operator} is independent of the choice of representatives --- as it should be, given that $P_{V,W}$ is uniquely defined in terms of abstract subspaces alone.
Using the representatives and an $r$-dimensional state $z$ defined by $\hat{x} = \Phi z \in V$, we obtain a representative of the reduced-order model \cref{eqn:reduced_order_model} given by
\begin{equation}
\boxed{
\begin{aligned}
    \ddt  z &= (\Psi^T \Phi)^{-1}\Psi^T  f(\Phi z, u) =: \tilde{f}\left(z, u;\ (\Phi, \Psi)\right), \qquad z(t_0) = (\Psi^T \Phi)^{-1}\Psi^T  x_0\\
    \hat{y} &= g(\Phi z) =: \tilde{g}\left(z;\ (\Phi, \Psi)\right),
\end{aligned}
}
\label{eqn:reduced_order_model_representative}
\end{equation}
that can be simulated on a computer.
While the state $z(t)$ of \cref{eqn:reduced_order_model_representative} depends on the choice of $(\Phi,\Psi) \in \pi^{-1}(V,W)$, the output $\hat{y}(t)$ depends only on the subspaces
$(V,W)$, and not on our choice of matrix representatives.

Consequently, any function of $(\Phi, \Psi)$ that depends only on the output $\hat{y}(t)$ of \cref{eqn:reduced_order_model_representative} can be viewed as a function on $\mathcal{M}$ composed with the canonical projection $\pi$.
Hence, we can evaluate our cost function \cref{eqn:ROM_cost} for a subspace pair $(V,W)$ by computing 
\begin{equation}
    \bar{J}(\Phi, \Psi) = J(\pi(\Phi, \Psi))
    \label{eqn:lifted_ROM_cost}
\end{equation}
for any choice of represenatives $(\Phi, \Psi)\in\pi^{-1}(V,W)$, that is, by evaluating the sum in \cref{eqn:ROM_cost} using the output $\hat{y}(t)$ generated by \cref{eqn:reduced_order_model_representative}.
Moreover, Theorem~4.29 in J. M. Lee \cite{Lee2013introduction} tells us that $J$ is smooth if and only if $\bar{J}$ is smooth.

\subsection{Topology of the Optimization Problem Domain}
The main result of this section is the following:

\begin{theorem}[Topology of Subspaces that Define Oblique Projections]
\label{thm:topology_of_P}
Let~$\mathcal{P}$ denote the pairs of subspaces $(V,W)\in \mathcal{G}_{n,r} \times \mathcal{G}_{n,r}$ such that $V\oplus W^{\perp} = \mathbb{R}^n$.
Then $\mathcal{P}$ is open, dense, and connected in $\mathcal{G}_{n,r} \times \mathcal{G}_{n,r}$.
Moreover, $\mathcal{P}$ is diffeomorphic to the set of rank-$r$ projection operators
\begin{equation}
    \mathbb{P} = \left\{ P \in \mathbb{R}^{n\times n} \ : \ P^2 = P \quad \mbox{and} \quad \rank(P) = r \right\}.
\end{equation}
\end{theorem}
\begin{proof}
See \cref{app:topology_of_P}.
\end{proof}
The openness of $\mathcal{P}$ in $\mathcal{G}_{n,r} \times \mathcal{G}_{n,r}$ means that it is a submanifold of $\mathcal{G}_{n,r} \times \mathcal{G}_{n,r}$ with the same dimension $\dim\mathcal{P} = 2nr-2r^2$.
The connectedness result is especially important since it means than an optimization routine can access any point in the set $\mathcal{P}$ by a smooth path from any initial guess without ever encountering the ``bad set'' $\mathcal{G}_{n,r}\times\mathcal{G}_{n,r}\setminus\mathcal{P}$.
In other words the bad set doesn't cut off access to any region of $\mathcal{P}$ by an optimizer that progresses along a smooth path, e.g., a gradient flow.

The reduced-order model \cref{eqn:reduced_order_model} may not have a solution over the desired time interval $[t_0, t_{L-1}]$ for every projection operator defined by $(V,W)\in\mathcal{P}$.
The following result characterizes the appropriate domain $\mathcal{D}\subset\mathcal{P}$ over which the ROM has a unique solution as well as the key properties of solutions when they exist.
\begin{proposition}[Properties of ROM Solutions]
\label{prop:domain_of_existence}
When the reduced-order model \cref{eqn:reduced_order_model} has a solution over the time interval $[t_0, t_{L-1}]$, it is unique.
Let $\mathcal{D}\subset \mathcal{P}$ denote the set of subspace pairs $(V,W)$ for which the resulting reduced-order model \cref{eqn:reduced_order_model} has a unique solution over the time interval $[t_0, t_{L-1}]$, and let $\hat{x}(t;(V,W))$ denote the state of \cref{eqn:reduced_order_model} with $(t, (V,W)) \in [t_0, t_{L-1}]\times \mathcal{D}$.
Then
\begin{enumerate}[leftmargin=1cm]
    \item $\mathcal{D}$ is open in $\mathcal{P}$, and hence $\mathcal{D}$ is also open in $\mathcal{G}_{n,r}\times\mathcal{G}_{n,r}$.
    \item When $\frac{\partial}{\partial x}f(x, u(t))$ is bounded then $\mathcal{D} = \mathcal{P}$.
    \item If $(x,t)\mapsto f(x,u(t))$ has continuous partial derivatives with respect to $x$ up to order $k\geq 1$, then $(t,(V,W)) \mapsto \hat{x}(t;(V,W))$ is continuously differentiable with respect to $(V,W)$ up to order $k$ on $[t_0, t_{L-1}]\times \mathcal{D}$.
    \item If $\{ (V_k, W_k) \}_{k=1}^{\infty}\subset \mathcal{D}$ is a sequence approaching $(V_k, W_k)\to (V_0, W_0) \in \mathcal{P}\setminus\mathcal{D}$ and $\hat{x}(t;(V_k,W_k))$ are the corresponding solutions of \cref{eqn:reduced_order_model}, then
    \begin{equation}
        \max_{t\in [t_0, t_{L-1}]} \Vert \hat{x}(t;(V_k,W_k)) \Vert \to \infty \quad \mbox{as}\quad k\to \infty.
    \end{equation}
\end{enumerate}
\end{proposition}
\begin{proof}
The claims follow from standard results in the theory of ordinary differential equations that can be found in W. G. Kelly A. C. Peterson \cite{Kelly2004theory}.
We give the detailed proof in \cref{app:domain_of_existence}.
\end{proof}
In particular, \cref{prop:domain_of_existence} shows that the solutions produced by the reduced-order model are twice continuously differentiable over $\mathcal{D}$ and blow up as points outside of $\mathcal{D}$ are approached.
In the special case when the governing equations \cref{eqn:full_order_model} have a bounded Jacobian, we may dispense with $\mathcal{D}$ entirely since the projection-based reduced-order model always has a unique solution.

\subsection{Regularization and Existence of a Minimizer}
Without regularization, we cannot guarantee a priori that a sequence of subspace pairs with decreasing cost doesn't approach a point outside of the set $\mathcal{P}$ where projection operators are defined.
That is, a minimizer for the cost function \cref{eqn:ROM_cost} may not even exist in $\mathcal{P}$, in which case our optimization problem would have no solution.
In order to address this issue, we introduce a regularization function $\rho(V,W)$ into the cost \cref{eqn:ROM_cost} that ``blows up'' to $+\infty$ as the subspaces $(V,W)$ approach any point outside of $\mathcal{P}$, and nowhere else.
In order to do this, we use the fact that $(V,W)\in \mathcal{P}$ if and only if all matrix representatives $(\Phi, \Psi) \in \pi^{-1}(V,W)$ have $\det{(\Psi^T \Phi)}\neq 0$.
While this condition characterizes the set $\mathcal{P}$, we cannot use $\det{(\Psi^T \Phi)}$ directly since its nonzero value depends on the choice of representatives.
But this problem is easily solved by an appropriate normalization, leading us to define the regularization of \cref{eqn:ROM_cost} in terms of representatives according to
\begin{equation}
\boxed{
    \rho\circ \pi (\Phi, \Psi) = -\log\left( \frac{\det ( \Psi^T \Phi )^2}{\det(\Phi^T \Phi) \det(\Psi^T \Psi)} \right).
}
    \label{eqn:regularization}
\end{equation}
We observe that the function $\rho:\mathcal{P}\to\mathbb{R}$ in \cref{eqn:regularization} is well-defined because $\rho \circ \pi(\Phi,\Psi)$ does not depend on the representatives $(\Phi,\Psi)$ thanks to the product rule for determinants.

The following theorem shows that the regularization defined by \cref{eqn:regularization} has the desirable properties that it vanishes when $V=W$ and ``blows up'' as $(V,W)$ escapes the set $\mathcal{P}$.
When $V=W$, the resulting projection operator $P_{V,V}$ is the orthogonal projection onto $V$.
\begin{theorem}[Regularization]
\label{thm:regularization_properties}
The minimum value of $\rho$ defined by \cref{eqn:regularization} over $\mathcal{P}$ is zero, and this minimum value $\rho(V,W) = 0$ is attained if and only if $V=W$.
On the other hand, if $(V_0, W_0) \in \mathcal{G}_{n,r}\times\mathcal{G}_{n,r}\setminus\mathcal{P}$ and $\lbrace (V_n, W_n) \rbrace_{n=1}^{\infty}$ is a sequence of subspaces in $\mathcal{P}$ such that $(V_n, W_n) \to (V_0, W_0)$ as $n\to\infty$, then $\lim_{n\to\infty}\rho(V_n, W_n) = \infty$.
\end{theorem}
\begin{proof}
See \cref{app:regularization_and_minimizer}.
\end{proof}

We must also rule out the possibility that a sequence of subspace pairs with decreasing cost approaches a point where the reduced-order model does not have a unique solution.
By \cref{prop:domain_of_existence}, we do not have this problem when the full-order model has a bounded Jacobian since the reduced-order model always has a unique solution, i.e., $\mathcal{D} = \mathcal{P}$.
On the other hand, when $\mathcal{D}\neq \mathcal{P}$ we may accomplish this by choosing a cost function that blows up if the states of the reduced-order model blow up.
In particular, we assume the following:
\begin{assumption}
\label{asmpn:blow_up}
Let $\mathcal{D}$ be as in \cref{prop:domain_of_existence} and $\mathcal{P}$ be the subset of $(V,W) \in \mathcal{G}_{n,r}\times\mathcal{G}_{n,r}$ for which $V\oplus W^{\perp} = \mathbb{R}^n$.
If $\mathcal{D} \neq \mathcal{P}$ and $\{ (V_k, W_k) \}_{k=1}^{\infty} \subset \mathcal{D}$ is any sequence producing solutions $\hat{x}(t;(V_k,W_k))$ of the reduced-order model \cref{eqn:reduced_order_model} such that
\begin{equation}
        \max_{t\in [t_0, t_{L-1}]} \Vert \hat{x}(t;(V_k,W_k)) \Vert \to \infty \quad \mbox{as}\quad k\to \infty,
\end{equation}
then we assume that $J(V_k, W_k) \to \infty$.
\end{assumption}
In practice, this is a reasonable assumption if $g(x) \to \infty$ as $\Vert x\Vert \to \infty$ and $L_y(y) \to \infty$ as $\Vert y\Vert \to \infty$.
Alternatively, one could add a new regularization term to the cost function \cref{eqn:ROM_cost} that penalizes reduced-order model states with large magnitudes.
In \cref{cor:existence_of_minimizer} we show that a minimizer of the cost function \cref{eqn:ROM_cost} exists in the valid set $\mathcal{D}\subset \mathcal{P}$ when \cref{asmpn:blow_up} holds and we use the regularization described by \cref{eqn:regularization} with any positive weight $\gamma > 0$.

\section{Optimization Algorithm}
\label{sec:optimization_algorithm}
In this section we describe how to optimize the projection subspaces by minimizing the cost function \cref{eqn:ROM_cost} over the product of Grassmann manifolds $\mathcal{M} = \mathcal{G}_{n,r} \times \mathcal{G}_{n,r}$ using the Riemannian conjugate gradient algorithm described by H. Sato in \cite{Sato2016dai}.
We use the exponential map and parallel translation along geodesics given by A. Edelman et al. \cite{Edelman1998geometry}, and we provide an adjoint sensitivity method for computing the gradient of the cost function.
Other geometric optimization algorithms such as stochastic gradient descent \cite{Bonnabel2013stochastic, Sato2019riemannian} and quasi-Newton methods \cite{Ring2012optimization, Huang2015broyden} are also well-suited for high-dimensional problems and rely on the same key ingredients we provide here.

\subsection{Computing the Gradient}
\label{subsec:gradient}
In order to compute the gradient we endow $\mathcal{G}_{n,r}$ with a Riemannian metric and we use the product metric on $\mathcal{M} = \mathcal{G}_{n,r} \times \mathcal{G}_{n,r}$.
This allows us to perform key operations such as constructing geodesics and parallel translates of tangent vectors on $\mathcal{M}$ by treating its two components separately \cite{Lee2013introduction}.
We follow P. A. Absil et al. \cite{Absil2009optimization}, whereby the metric on $\mathcal{G}_{n,r}$ is induced by a compatible metric on $\mathbb{R}_*^{n\times r}$ acting on lifted representatives of tangent vectors in a prescribed ``horizontal space''.
The Riemannian metric we adopt for the structure space $\bar{\mathcal{M}} = \mathbb{R}_*^{n\times r}\times \mathbb{R}_*^{n\times r}$ is given by the product metric
\begin{equation}
    \langle (X_1,Y_1),\ (X_2,Y_2)\rangle_{(\Phi, \Psi)} = 
    \underbrace{\Tr{\left[(\Phi^T \Phi)^{-1} X_1^T X_2 \right]}}_{\langle X_1,\ X_2\rangle_{\Phi}}
    + \underbrace{\Tr{\left[(\Psi^T \Psi)^{-1} Y_1^T Y_2 \right]}}_{\langle Y_1,\ Y_2\rangle_{\Psi}}.
    \label{eqn:structure_space_metric}
\end{equation}
The gradient, expressed in terms of lifted representatives, is then found by computing the gradient with respect to the matrix representatives in the structure space.

For any tangent vector $\xi \in T_{p}\mathcal{M}$ and representative $\bar{p}\in\bar{\mathcal{M}}$ such that $p = \pi(\bar{p})$, there are an infinite number of possible $\bar{\xi}\in T_{\bar{p}}\bar{\mathcal{M}}$ that could serve as representatives of $\xi$ in the sense that $\xi = \D\pi(\bar{p})\bar{\xi}$.
A unique representative of $\xi$ is identified by observing that the pre-image $\pi^{-1}(p)$ of any $p\in\mathcal{M}$ is a smooth submanifold of $\bar{\mathcal{M}}$ yielding a decomposition of the tangent space $T_{\bar{p}}\bar{\mathcal{M}}$ into a direct sum of the ``vertical space'' defined by $\mathcal{V}_{\bar{p}} = T_{\bar{p}}\pi^{-1}(p)$
and the ``horizontal space'' defined as its orthogonal complement $\mathcal{H}_{\bar{p}} = \mathcal{V}_{\bar{p}}^{\perp}$.
The horizontal and vertical spaces for a product manifold are the products of the horizontal and vertical spaces for each component in the Cartesian product, and we have
\begin{equation}
    \mathcal{V}_{\Phi} = \left\{ \Phi A \ : \ A \in \mathbb{R}^{r\times r} \right\}, \qquad 
    \mathcal{H}_{\Phi} = \left\{ X\in\mathbb{R}^{n,r} \ : \ \Phi^T X = 0 \right\}, \qquad \Phi \in \mathbb{R}_*^{n,r}.
\end{equation}
Using the horizontal distribution on the structure space, we have the following:
\begin{definition}[horizontal lift \cite{Absil2009optimization}]
\label{def:horizontal_lift}
    Given $\xi \in T_p\mathcal{M}$ and a representative $\bar{p}\in \pi^{-1}(p)$,
    the ``horizontal lift'' of $\xi$ is the unique element $\bar{\xi}_{\bar{p}} \in \mathcal{H}_{\bar{p}}$ such that $\xi = \D\pi(\bar{p})\bar{\xi}_{\bar{p}}$.
\end{definition}
The horizontal lifts of a tangent vector $\xi \in T_{V}\mathcal{G}_{n,r}$ to either of the component Grassmann manifold at different representatives transform according to
\begin{equation}
    \bar{\xi}_{\Phi S} 
    = \bar{\xi}_{\Phi} S \in \mathbb{R}^{n\times r}, \qquad
    \forall S\in GL_r,
    \label{eqn:horizontal_lift_transformation}
\end{equation}
for every $\Phi \in \mathbb{R}_*^{n,r}$ with $\Range{\Phi} = V$, as shown by Example~3.6.4 in \cite{Absil2009optimization}.
The structure space metric \cref{eqn:structure_space_metric} induces a Riemannian product metric on $\mathcal{M} = \mathcal{G}_{n,r}\times \mathcal{G}_{n,r}$ defined in terms of horizontal lifts
\begin{equation}
    \left\langle (\xi_1, \zeta_1),\ (\xi_2, \zeta_2) \right\rangle_{(V,W)} = 
    \underbrace{\Tr{\left[(\Phi^T \Phi)^{-1} (\bar{\xi_1}_{\Phi})^T \bar{\xi_2}_{\Phi} \right]}}_{\langle \xi_1,\ \xi_1\rangle_{V}}
    + \underbrace{\Tr{\left[(\Psi^T \Psi)^{-1} (\bar{\zeta_1}_{\Psi})^T \bar{\zeta_2}_{\Psi} \right]}}_{\langle \zeta_1,\ \zeta_2\rangle_{W}}.
    \label{eqn:Riemannian_metric}
\end{equation}
This metric is independent of the choice of representatives $(\Phi, \Psi)\in \pi^{-1}(V,W)$ thanks to the transformation property \cref{eqn:horizontal_lift_transformation}.

An important consequence of the orthogonality of the horizontal and vertical subspaces is that the horizontal lift of the gradient of the cost function $J:\mathcal{M}\to\mathbb{R}$ is given by the gradient of $\bar{J} = J\circ \pi : \bar{\mathcal{M}} \to \mathbb{R}$ \cite{Absil2009optimization}; that is,
\begin{equation}
\label{eqn:gradient_equivalence}
    \overline{\grad J(V,W)}_{(\Phi,\Psi)} = \grad \bar{F}(\Phi, \Psi), \qquad \forall (\Phi, \Psi)\in\pi^{-1}(V,W).
\end{equation}
This means that the gradient computed with respect to the matrix representatives in the structure space is the appropriate lifted representative in the horizontal space at $(\Phi, \Psi)$ of the gradient tangent to $\mathcal{M}$ at $(V,W)$.
The gradient of the lifted cost function $\bar{J}$ may computed using the adjoint sensitivity method described below in \cref{thm:adjointOptimization}.
The analogous adjoint sensitivity method for a cost function in which the error is integrated over time, rather than being summed over $\{ t_l \}_{l=0}^{L-1}$ as in \cref{eqn:ROM_cost}, is provided by \cref{thm:adjointOptimization_integral_objective}.

\begin{theorem}[Gradient with Respect to Model Parameters]
\label{thm:adjointOptimization}
Suppose we have observation data $\lbrace  y_1, \ldots,  y_L\rbrace$ generated by the full-order model \cref{eqn:full_order_model} at sample times $t_0<\cdots<t_{L-1}$ with initial condition $x_0$ and input signal $u$.
Consider the reduced-order model representative \cref{eqn:reduced_order_model_representative} with parameters $\theta = (\Phi, \Psi)$ in the structure space $\bar{\mathcal{M}}$, which is a Riemannian manifold.
With $\pi(\theta) \in \mathcal{D}$, let $\hat{y}_i(\theta) = \hat{y}(t_i;\theta)$ be the observations at the corresponding times $t_i$ generated by \cref{eqn:reduced_order_model_representative} and let $z(t;\theta)$ denote the state trajectory of \cref{eqn:reduced_order_model_representative}.
Then the cost function
\begin{equation}
    \bar{J}(\theta) := \sum_{i=0}^{L-1} L_y(\hat{y}_i(\theta) -  y_i),
    \label{eqn:GlobalCostFunctional}
\end{equation}
measuring the error between the observations generated by the models, is differentiable at every $\theta \in \pi^{-1}(\mathcal{D})$.
Let
\begin{subequations}
\begin{align}
    F(t) &= \frac{\partial \tilde{f}}{\partial  z}( z(t;\theta),  u(t);\ \theta) &: \mathbb{R}^r\to\mathbb{R}^r \\
    S(t) &= \frac{\partial \tilde{f}}{\partial \theta}( z(t;\theta),  u(t);\ \theta) &: T_{\theta}\bar{\mathcal{M}} \to \mathbb{R}^r \\
    H(t) &= \frac{\partial \tilde{g}}{\partial  z}( z(t;\theta);\ \theta) &: \mathbb{R}^r \to \mathbb{R}^m \\
    T(t) &= \frac{\partial \tilde{g}}{\partial \theta}( z(t;\theta);\ \theta) &: T_{\theta}\bar{\mathcal{M}} \to \mathbb{R}^m,
\end{align}
\end{subequations}
denote the linearized dynamics and observation functions around $ z(t;\theta)$,
let $g_i = \grad L_y(\hat{y}_i(\theta) -  y_i)$,
and define an adjoint variable $\lambda (t)$ that satisfies
\begin{subequations}
\begin{align}
    \label{eqn:adjoint_dynamics}
    -\ddt \lambda (t) = & F(t)^*\lambda (t),\qquad t\in (t_{i},t_{i+1}],\qquad 0\le i < L-1,\\ 
    \lambda (t_i) = & \lim_{t\to t_i^+} \lambda (t) +  H(t_i)^* g_i \label{eqn:adjoin_eqn_discontinuity},\\
    \lambda (t_{L-1}) = & H(t_{L-1})^{*}g_{L-1}. \label{eqn:adjoing_eqn_final_condition}
\end{align}
\label{eqn:adjoint_equations}
\end{subequations}
Here $(\cdot)^*$ denotes the adjoint of a linear operator with respect to the inner products on the appropriate spaces.
Then the gradient of the cost function \cref{eqn:GlobalCostFunctional} is given by
\begin{equation}
\boxed{
    \grad \bar{J}(\theta) =  \left( \frac{\partial z}{\partial \theta} (t_0;\theta) \right)^*\lambda (t_0) + \int_{t_0}^{t_{L-1}} S(t)^*\lambda (t)\ \td t + \sum_{i=0}^{L-1} T(t_i)^*g_i.
    }
\label{eqn:parameter_gradient}
\end{equation}
\end{theorem}
\begin{proof}
See \cref{app:adjoint_gradient_and_required_terms}.
\end{proof}

The explicit form of each term required to compute the horizontal lift of the gradient of the cost function \cref{eqn:ROM_cost} using \cref{thm:adjointOptimization} is provided by the following \cref{prop:terms_for_adjoint}.
In order to simplify these expressions, we work with orthonormal representatives, i.e., $(\Phi, \Psi) \in \pi^{-1}(V,W)$ such that $\Phi^T \Phi = \Psi^T \Psi = I_r$ together with the additional condition $\det(\Psi^T \Phi) > 0$.
Such representatives can always be obtained via QR-factorization and adjusting the sign of a column of $\Phi$ or $\Psi$.
The horizontal lift of the gradient computed at any other representatives $(\Phi S, \Psi T)$ with $S,T\in GL_r$ can be obtained from the horizontal lift of the gradient computed at $(\Phi, \Psi)$ via \cref{eqn:horizontal_lift_transformation}.
\begin{proposition}[Required Terms for Gradient]
\label{prop:terms_for_adjoint}
We assume that the representatives $(\Phi, \Psi)\in \pi^{-1}(V,W)$ have been chosen such that $\Phi^T \Phi = \Psi^T \Psi = I_r$ and $\det(\Psi^T \Phi) > 0$, and we let $A = (\Psi^T \Phi)^{-1}$.
Then the terms required to compute the gradient of the cost function using the model \cref{eqn:reduced_order_model_representative} with respect to the representatives in the structure space via \cref{thm:adjointOptimization} are given by
\begin{equation}
    F(t)^* 
    = \left(\frac{\partial \tilde{f}}{\partial  z}\big( z(t),  u(t)\big)\right)^T
    \label{eqn:F}
\end{equation}
\begin{multline}
    S(t)^*v = \bigg( \left( \frac{\partial f}{\partial x}\big(\Phi z(t), u(t)\big) \right)^T  \Psi A^T v z(t)^T - \Psi A^T v \tilde{f}(z(t), u(t))^T ,\ \\
    \left( f\big(\Phi z(t), u(t)\big) - \Phi \tilde{f}\big(z(t), u(t)\big) \right) v^T A \bigg)
    \qquad \forall v\in\mathbb{R}^r,
\label{eqn:S}
\end{multline}
\begin{equation}
    H(t)^*
    = \left(\frac{\partial \tilde{g}}{\partial  z}\big( z(t)\big)\right)^T,
    \label{eqn:H}
\end{equation}
\begin{equation}
    T(t)^*w = \left( \left( \frac{\partial g}{\partial x}\big(\Phi z(t)\big)\right)^T  w z(t)^T,\ 0\right)
    \qquad \forall w\in\mathbb{R}^{m},
    \label{eqn:T}
\end{equation}
\begin{equation}
    \left( \frac{\partial z}{\partial (\Phi, \Psi)} \big(t_0;(\Phi, \Psi)\big) \right)^T v
    = \Big(-\Psi A^T v z(t_0)^T ,\ 
    \big(x_0 - \Phi z(t_0)\big) v^T A \Big)
    \qquad \forall v\in\mathbb{R}^r.
\label{eqn:dz0_adj}
\end{equation}
The gradient of the regularization function \cref{eqn:regularization} is given by
\begin{equation}
    \grad (\rho\circ\pi)(\Phi, \Psi) = 2\big( \Phi - \Psi A^T,\   \Psi - \Phi A \big). 
    \label{eqn:regularization_gradient}
\end{equation}
\end{proposition}
\begin{proof}
See \cref{app:adjoint_gradient_and_required_terms}.
\end{proof}

We provide \cref{alg:gradient2}, below, to compute the gradient according to \cref{thm:adjointOptimization}, with the appropriate terms given in \cref{prop:terms_for_adjoint}.
In \cref{subapp:integrated_model_error}, we provide \cref{alg:gradient_int_obj} for computing the gradient of an objective where the modeling error is integrated over time, rather than being summed over $\{t_l\}_{l=0}^{L-1}$.
\begin{algorithm}
\caption{Compute the cost function gradient with respect to $(\Phi,\Psi)$}
\label{alg:gradient2}
\begin{algorithmic}[1]
\STATE{\textbf{input}: orthonormal representatives $(\Phi, \Psi)\in\pi^{-1}(V,W)$ with $\det(\Psi^T\Phi)>0$, initial condition $x_0$, observations $\{y_l\}_{l=0}^{L-1}$, regularization weight $\gamma$.}
\STATE{Assemble and simulate the ROM representative \cref{eqn:reduced_order_model_representative} from initial condition $z_0 = \Psi^T x_0$, storing predicted outputs $\{\hat{y}_l\}_{l=0}^{L-1}$ and trajectory $z(t)$ via interpolation. \label{algstep:assemble_ROM}}
\STATE{Initialize the gradient: $\grad \bar{J} \leftarrow T(t_{L-1})^*\grad L_y(\hat{y}_{L-1} - y_{L-1})$.}
\STATE{Compute adjoint variable at final time: $\lambda(t_{L-1}) = H(t_{L-1})^* \grad L_y(\hat{y}_{L-1} - y_{L-1})$.}
\FOR{$l=L-2, L-3, \ldots, 0$}
    \STATE{Solve the adjoint equation \cref{eqn:adjoint_dynamics} backwards in time over the interval $[t_l, t_{l+1}]$ using the linearized ROM dynamics \cref{eqn:F} and store $\lambda(t)$ on this interval.}
    \STATE{Compute the integral component of \cref{eqn:parameter_gradient} over the interval $[t_l, t_{l+1}]$: $\grad \bar{J} \leftarrow \grad \bar{J} + \int_{t_l}^{t_{l+1}} S(t)^*\lambda(t)\td t$ using Gauss-Legendre quadrature.\label{algstep:int_component}}
    \STATE{Add $l$th element of the sum in \cref{eqn:parameter_gradient}: $\grad \bar{J} \leftarrow \grad \bar{J} +  T(t_l)^*\grad L_y(\hat{y}_l - y_l)$.}
    \STATE{Add ``jump'' \cref{eqn:adjoin_eqn_discontinuity} to adjoint variable: $\lambda(t_l) \leftarrow \lambda(t_l) +  H(t_l)^* \grad L_y(\hat{y}_l - y_l)$.}
\ENDFOR
\STATE{Add gradient due to initial condition: $\grad \bar{J} \leftarrow \grad \bar{J} + \left( \frac{\partial z}{\partial (\Phi, \Psi)} (t_0) \right)^*\lambda(t_0)$.}
\STATE{Normalize by trajectory length: $\grad \bar{J} \leftarrow \grad \bar{J} / L$.}
\STATE{Add regularization: $\grad \bar{J} \leftarrow \grad \bar{J} + \gamma \grad (\rho\circ\pi)(\Phi, \Psi)$.}
\RETURN{$\grad \bar{J}$}
\end{algorithmic}
\end{algorithm}

The computational cost of \cref{alg:gradient2} is dominated by the steps that require evaluation of objects resembling the full-order dynamics, namely steps \ref{algstep:assemble_ROM} and \ref{algstep:int_component}.
These are of three kinds: evaluating the nonlinear right-hand side $f(x,u)$, acting on a vector with the linearized right-hand side $\partial f (x,u)/\partial x$, or acting with its transpose $(\partial f (x,u)/\partial x)^T$.
For a quadratically-bilinear full-order model and an $r$-dimensional reduced-order model, assembling the ROM (step \ref{algstep:assemble_ROM}) requires $O(r^2)$ FOM-like evaluations. 
Evaluating $S(t)^*\lambda(t)$ in \ref{algstep:int_component} using \cref{eqn:S} involves querying $f(x,u)$ and $(\partial f (x,u)/\partial x)^T$ acting on a vector. 
Hence, the cost (per iteration) of step \ref{algstep:int_component} is $O(q)$ FOM-like evaluations, where $q$ is the number of quadrature points used to approximate the integral over the interval $[t_l,t_{l+1}]$. 
When using high-order quadrature, one may take $q$ to be between one and ten. 
Thus, the total cost to compute the gradient is $O(r^2 + qL)$ FOM-like evaluations. 
Most (if not all) modern fluid flow solvers are equipped with the necessary functionality to perform all the aforementioned FOM-like evaluations, so the method that we propose can be easily integrated with existing software.
\begin{remark}
    For a general nonlinear system with sparse coupling between states the Discrete Empirical Interpolation Method (DEIM) \cite{Chaturantabut2010nonlinear} could eliminate the costly ROM assembly step by replacing $f$ with $\Pi f$ in \cref{eqn:reduced_order_model}, where $\Pi$ is the DEIM projector.
\end{remark}

\subsection{Geometric Conjugate Gradient Algorithm}
In \cref{alg:conj_grad_alg}, below, we give the implementation details for the
geometric conjugate gradient method described by H. Sato \cite{Sato2016dai},
with the required retraction and vector transport provided by the exponential
map and parallel translation along geodesics described by Theorems~2.3~and~2.4
in A. Edelman et al.\ \cite{Edelman1998geometry}.
It is also common to work with a non-exponential retraction, which we discuss in \cref{subapp:another_retraction}.
Given a search direction $\eta_k = (\xi_k, \zeta_k) \in T_{p_k}(\mathcal{G}_{n,r}\times\mathcal{G}_{n,r})$ at the current iterate $p_k = (V_k,W_k)$ and a step size $\alpha_k \geq 0$, the next iterate is computed using the exponential map \cite{doCarmo1992Riemannian, Edelman1998geometry} according to
\begin{equation}
    p_{k+1} = \exp_{p_k}(\alpha_k \eta_k) = \left( \exp_{V_k}(\alpha_k \xi_k), \ \exp_{W_k}(\alpha_k\zeta_k) \right).
\label{eqn:Riemannian_line_search_exp}
\end{equation}
Confining our attention to the one-dimensional objective $J_k(\alpha) = J(\exp_{p_k}(\alpha \eta_k))$ defined along the resulting geodesic, the step size $\alpha_k$ is selected in order to satisfy the Wolfe conditions,
\begin{subequations}
\begin{align}
    J_k(\alpha_k) &\leq J_k(0) + c_1 \alpha_k J_k'(0) \label{eqn:Wolfe_1} \\
    J_k'(\alpha_k) &\geq c_2 J_k'(0), \label{eqn:Wolfe_2}
\end{align}
\label{eqn:Wolfe_conditions}
\end{subequations}
where $0 < c_1 < c_2 < 1$ are user-specified constants and $J_k'$ denotes the derivative of $J_k$.
Such a step size can always be found, and we use the simple bisection method described in \cite{Burke2004line} to find one.

The search direction incorporates second-order information about the cost function by combining the gradient of the cost at the current iterate with the previous search direction.
The previous search direction $\eta_{k-1} \in T_{p_{k-1}}\mathcal{M}$ is moved to the tangent space at the current iterate $T_{p_k}\mathcal{M}$ via parallel translation \cite{doCarmo1992Riemannian} along the geodesic, denoted $\mathcal{T}_{p,\eta}:T_p\mathcal{M} \to T_{R_p(\eta)}\mathcal{M}$.
We use the explicit formula for parallel translation along geodesics on the Grassman manifold given by Theorem~2.4 in~\cite{Edelman1998geometry} and the fact that $\mathcal{T}_{(V,W), (\xi_1,\zeta_1)}(\xi_2,\zeta_2) = (\mathcal{T}_{V, \xi_1}\xi_2,\ \mathcal{T}_{W, \zeta_1}\zeta_2)$ on the product manifold $\mathcal{M} = \mathcal{G}_{n,r}\times\mathcal{G}_{n,r}$.
The next search direction is computed according to
\begin{equation}
    \eta_k = -\grad J(p_k) + \beta_k \mathcal{T}_{p_{k-1}, \alpha_{k-1}\eta_{k-1}}(\eta_{k-1}),
\end{equation}
where the coefficient $\beta_k$ is defined differently for different conjugate gradient algorithms.
We use the Riemannian Dai-Yuan coefficient proposed by H. Sato \cite{Sato2016dai} 
and given by
\begin{equation}
    \beta_k = \frac{\left\langle \grad J(p_k),\ \grad J(p_k)\right\rangle_{p_k}}{\left\langle \grad J(p_{k}),\ \mathcal{T}_{\alpha_{k-1}\eta_{k-1}}(\eta_{k-1}) \right\rangle_{p_k} - \left\langle \grad J(p_{k-1}),\ \eta_{k-1} \right\rangle_{p_{k-1}}},
    \label{eqn:Riemannian_DY_coeff}
\end{equation}
since it yields excellent performance in practice and provides guaranteed convergence when the step sizes satisfy the Wolfe conditions \cite{Wolfe1969convergence}.

\begin{algorithm}
\caption{Geometric conjugate gradient algorithm for model reduction}
\label{alg:conj_grad_alg}
\begin{algorithmic}[1]
\STATE{\textbf{Input:} orthonormal representatives $(\Phi_0, \Psi_0)$ of init. subspaces with $\det(\Psi_0^T \Phi_0) > 0$, stopping threshold $\varepsilon > 0$, and Wolfe cond. coefficients $0 < c_1 < c_2 < 1$}
\STATE{Compute cost $\bar{J}(\Phi_0, \Psi_0)$ and gradient $\grad \bar{J}_0$ using \cref{alg:gradient2}}
\STATE{Initialize the search direction $(X_0, Y_0) = \grad \bar{J}_0$ and set $k=0$}
\WHILE{$\left\langle \grad \bar{J}_k, \grad \bar{J}_k \right\rangle_{(\Phi_k, \Psi_k)} > \varepsilon$, given by \cref{eqn:structure_space_metric},}
    \STATE{Compute SVDs $X_k = U_X\Sigma_X V_X^T$, $Y_k = U_Y\Sigma_Y V_Y^T$ with $V V^T = I_r$.}
    \STATE{Define geodesic curves (via Thm.~2.3, \cite{Edelman1998geometry}) $$\Phi(\alpha) = \left[ \Phi_k V_X \cos(\alpha \Sigma_X) + U_X \sin(\alpha \Sigma_X) \right] V_X^T,$$ $$\Psi(\alpha) = \left[ \Psi_k V_Y \cos(\alpha \Sigma_Y) + U_Y \sin(\alpha \Sigma_Y) \right] V_Y^T$$ and line-search objective $J_k(\alpha) = \bar{J}(\Phi(\alpha), \Psi(\alpha))$.}
    \STATE{Compute step size $\alpha_k$ satisfying Wolfe
      conditions \cref{eqn:Wolfe_conditions} using bisection \cite{Burke2004line}, and orthonormal representatives of next iterate $(\Phi_{k+1}, \Psi_{k+1}) = (\Phi(\alpha_k), \Psi(\alpha_k))$.}%
    \STATE{Compute parallel translation of search direction (via Thm.~2.4, \cite{Edelman1998geometry}) $$ \tilde{X}_k = \left[ - \Phi_k V_X \sin(\alpha_k \Sigma_X) + U_X\cos(\alpha_k\Sigma_X) \right] U_X^T  X_k + X_k - U_X U_X^T  X_k,$$ $$ \tilde{Y}_k = \left[ - \Psi_k V_Y \sin(\alpha_k \Sigma_Y) + U_Y\cos(\alpha_k\Sigma_Y) \right] U_Y^T  Y_k + Y_k - U_Y U_Y^T  Y_k .$$}
    \STATE{Multiply first column of $\Phi_{k+1}$ and $\tilde{X}_k$ by $\sgn\det(\Psi_{k+1}^T \Phi_{k+1})$.}
    \STATE{Compute cost $\bar{J}(\Phi_{k+1}, \Psi_{k+1})$ and gradient $\grad \bar{J}_{k+1}$ using \cref{alg:gradient2}}
    \STATE{Using \cref{eqn:structure_space_metric}, compute Riemannian Dai-Yuan coefficient $$ \beta_{k+1} = \frac{\langle \grad \bar{J}_{k+1}, \grad \bar{J}_{k+1} \rangle_{(\Phi_{k+1}, \Psi_{k+1})}}{\langle \grad \bar{J}_{k+1}, (\tilde{X}_{k}, \tilde{Y}_{k}) \rangle_{(\Phi_{k+1}, \Psi_{k+1})} + \langle \grad \bar{J}_{k}, (X_k, Y_k) \rangle_{(\Phi_k, \Psi_k)}} $$}
    \STATE{Compute next search direction $(X_{k+1}, Y_{k+1}) = \grad \bar{J}_{k+1} + \beta_{k+1}(\tilde{X}_k, \tilde{Y}_k)$.}
    \STATE{Update $k \leftarrow k + 1$}
\ENDWHILE
\RETURN{orthonormal representatives $(\Phi_K, \Psi_K)$ of the optimized projection subspaces and the final cost $\bar{J}(\Phi_K, \Psi_K)$}
\end{algorithmic}
\end{algorithm}

\subsection{Convergence Guarantees}
The proofs of convergence for the geometric conjugate gradient algorithms described in \cite{Ring2012optimization, Sato2015new, Sato2016dai} (with retraction provided by the exponential map) rely on Lipschitz assumptions for the derivative of $J\circ \exp_{p_k}$ along the search direction $\eta_k$.
In particular, if there is a fixed Lipschitz constant $L_J$ so that for each iteration $k=0,1,2,\ldots$, we have
\begin{equation}
    \left\vert D(J\circ \exp_{p_k})(\alpha_k \eta_k)\eta_k - D(J\circ \exp_{p_k})(0)\eta_k \right\vert \leq L_J \alpha_k \Vert \eta_k\Vert_{p_k}^2,
    \label{eqn:Lipschitz_iterates_exp}
\end{equation}
then the Riemannian generalization of Zoutendijk's theorem given by Theorem~2 in \cite{Ring2012optimization} (Theorem~4.1 in \cite{Sato2016dai}) holds and
Theorem~4.2 in \cite{Sato2016dai} guarantees convergence of the geometric conjugate gradient algorithm in the sense that
\begin{equation}
\boxed{
    \liminf_{k\to\infty} \big\Vert \grad J(V_k, W_k) \big\Vert_{(V_k, W_k)} = 0.
    }
    \label{eqn:grad_convergence}
\end{equation}
In other words, the conjugate gradient algorithm will produce iterates with arbitrarily small gradients, which may be used as a stopping condition.
Fortunately, the Lipschitz condition \cref{eqn:Lipschitz_iterates_exp} is easily verified, and we obtain the following convergence result:
\begin{theorem}
\label{thm:CG_convergence_exp}
Suppose that there is a compact subset $\mathcal{D}_c$ of the domain $\mathcal{D}$ (defined in \cref{prop:domain_of_existence}) such that for every iteration $k=0,1,2,\ldots$, we have
\begin{equation}
    \gamma_k(t) = \exp_{p_k}(t\eta_k) \in \mathcal{D}_c \qquad \forall t\in [0,\alpha_k].
\end{equation}
Let $\nabla$ denote the Riemannian connection on $\mathcal{G}_{n,r}\times\mathcal{G}_{n,r}$ with metric given by \cref{eqn:Riemannian_metric}.
Then the Lipschitz condition~\cref{eqn:Lipschitz_iterates_exp} holds with
\begin{equation}
    L_J = \max_{\substack{(p,\xi)\in T\mathcal{M} : \\
    p \in \mathcal{D}_c, \ 
    \Vert \xi\Vert_p = 1}} \big\Vert (\nabla_{\xi} \grad J)(p) \big\Vert_p < \infty,
    \label{eqn:Lipschitz_const_for_exp}
\end{equation}
and the geometric conjugate gradient algorithm with Dai-Yuan coefficient \cref{eqn:Riemannian_DY_coeff} and $\alpha_k$ satisfying the Wolfe conditions \cref{eqn:Wolfe_conditions} converges in the sense of \cref{eqn:grad_convergence}.
\end{theorem}
\begin{proof}
See \cref{app:CG_convergence}.
\end{proof}
\noindent
The Lipschitz estimate in \cref{thm:CG_convergence_exp} also guarantees the convergence of other geometric conjugate gradient algorithms such as the Riemannian Fletcher-Reeves method with strengthened Wolfe conditions presented in \cite{Ring2012optimization}.
Similar results also hold for a common non-exponential retraction as we discuss in \cref{subapp:another_retraction}.
\begin{remark}
As we discuss in \cref{app:CG_convergence}, it is always possible to find step sizes $\alpha_k$ that satisfy the Wolfe conditions and the assumption in \cref{thm:CG_convergence_exp}.
However, guaranteeing that the step size produces a path contained in a pre-defined compact set $\mathcal{D}_c$ requires modifying the line-search method.
In practice, we did not find such a modification necessary and the simple bisection method in \cite{Burke2004line} was sufficient to produce converging iterates in \cref{alg:conj_grad_alg}.
\end{remark}

\section{Simple Nonlinear System with an Important Low-Energy Feature}
\label{sec:toy_model}
In this section, we illustrate our method on a simple example system for which
existing approaches to nonlinear model reduction perform poorly.  
In particular, we consider the system
\begin{equation}
\begin{split}
    \dot{x}_1 &= -x_1 + 20 x_1 x_3 + u \\
    \dot{x}_2 &= -2 x_2 + 20 x_2 x_3 + u \\
    \dot{x}_3 &= -5 x_3 + u \\
    y &= x_1 + x_2 + x_3,
\end{split}
\label{eqn:toy_model}
\end{equation}
and we compare our method (TrOOP) with 
two-dimensional projection-based models obtained using subspaces determined by 
POD, balanced truncation of the linearized system, Quadratic-Bilinear (QB) balanced truncation \cite{Benner2017balanced}, and the Quadratic-Bilinear Iterative Rational Krylov Algorithm (QB-IRKA) presented in \cite{Benner2018H2}.
The QB balancing method had similar, but slightly worse performance than QB-IRKA, so we shall only show the results using QB-IRKA.
We confine our attention to nonlinear impulse-responses with magnitudes $u_0 \in [0,1]$.
These responses can be obtained by considering the output of \cref{eqn:toy_model} with $u\equiv 0$ and known initial condition 
$x(0) = u_0 (1,\ 1,\ 1)$.
Two such responses with $u_0 = 0.5$ and $u_0 = 1$ are shown in \cref{fig:toy_model_training_trajectories}. 

The key feature of \cref{eqn:toy_model} is that the state $x_3$ plays a very important role in the dynamics of the states $x_1$ and $x_2$, while remaining small by comparison due to its fast decay rate.
In fact, for $u_0 > 1/5$ we have $\dot{y}(0) > 0$ and the output experiences transient growth due to the nonlinear interaction of $x_1$ and $x_2$ with $x_3$.
These nonlinear interactions become dominant for larger $u_0$, but are neglected
completely by model reduction techniques like balanced truncation that consider
only the linear part of~\cref{eqn:toy_model}.
\cref{fig:toy_model_training_trajectories} shows the result of such an
approach, in which we obtain a nonlinear reduced-order model by Petrov-Galerkin
projection of~\cref{eqn:toy_model} onto a two-dimensional subspace determined by
balanced truncation of the linearized system.
As shown in the figure, the resulting model over-predicts the transient growth by an amount that increases with $u_0$.
Techniques such as QB balancing and QB-IRKA extend the region of validity for the reduced-order models by considering second-order terms in the Volterra series for the response, yet still have deteriorating accuracy with increasing $u_0$ due to the neglected higher-order terms.

On the other hand, a two-dimensional POD-based model retains the most energetic states, which align closely with $x_1$ and $x_2$, and essentially ignores the important low-energy state $x_3$.
Consequently, the POD-based model of \cref{eqn:toy_model} does not predict any transient growth as shown in \cref{fig:toy_model_training_trajectories}.

\begin{figure}
    \centering
    \subfloat[training trajectories\label{fig:toy_model_training_trajectories}]{
    \begin{tikzonimage}[trim=20 10 40 20, clip=true, width=0.455\textwidth]{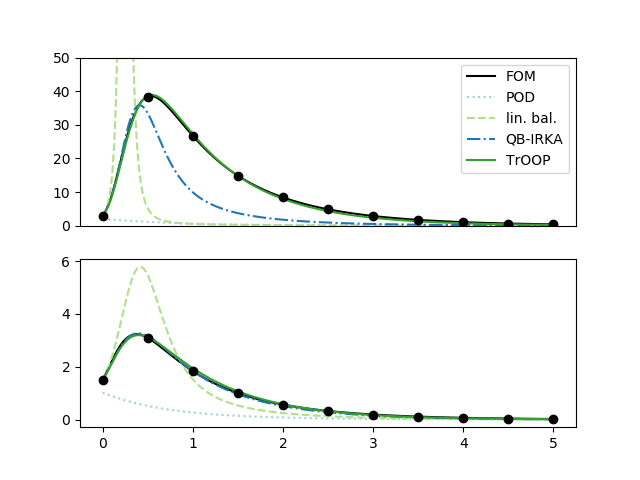}
    \node[rotate=90] at (0.0, 0.73) {\footnotesize $y$ for $u_0 = 1$};
    \node[rotate=90] at (0.0, 0.285) {\footnotesize $y$ for $u_0 = 0.5$};
    \node[rotate=0] at (0.55, 0.01) {\footnotesize $t$};
    %
    \end{tikzonimage}
    }
    \subfloat[error on testing trajectories\label{fig:toy_model_testing_performance}]{
    \begin{tikzonimage}[trim=10 10 40 20, clip=true, width=0.465\textwidth]{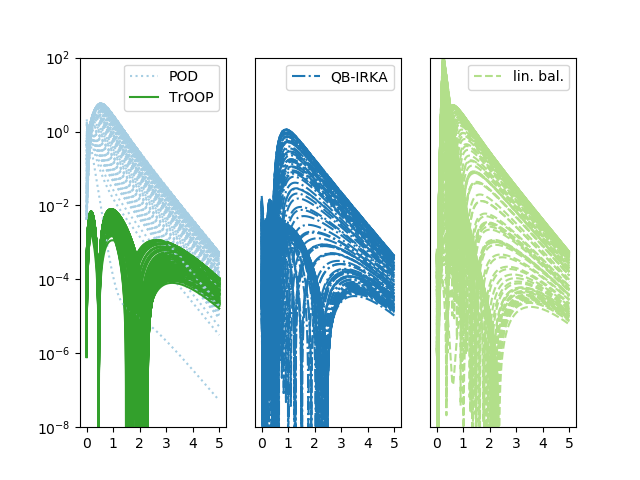}
    \node[rotate=90] at (0.0, 0.5) {\footnotesize $(\hat{y} - y)^2 / \avg{(y^2)}$};
    \node[rotate=0] at (0.315, 0.01) {\footnotesize $t$};
    \node[rotate=0] at (0.79, 0.01) {\footnotesize $t$};
    \end{tikzonimage}
    }
    \caption{In panel (a), we show the outputs generated by the full-order model \cref{eqn:toy_model} and various two-dimensional projection-based reduced-order models in response to impulses with magnitudes $u_0 = 0.5$ and $u_0 = 1$ at $t=0$.
    The sample points used to construct the objective function \cref{eqn:toy_model_cost} used to optimize the projection operator are shown as black dots.
    In panel (b), we show the normalized square errors of the reduced-order model predictions in response to $100$ impulses at $t=0$ whose magnitudes $u_0$ were drawn uniformly at random from the interval $[0,1]$. }
    \label{fig:toy_model_impulse}
\end{figure}
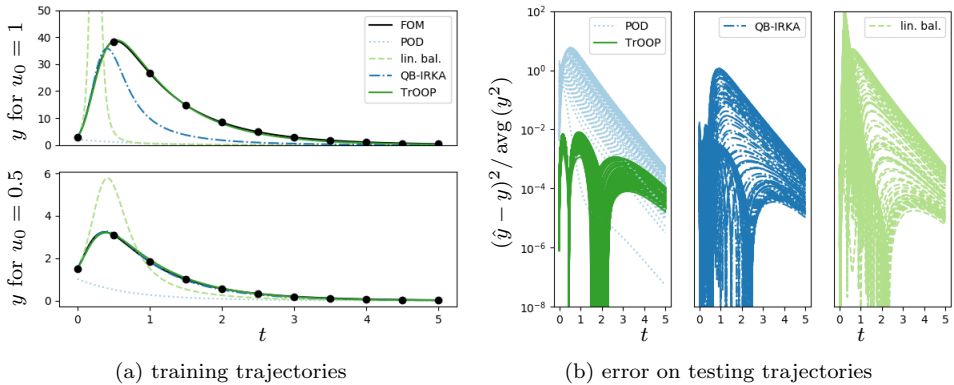

In order to find a two-dimensional reduced-order model of \cref{eqn:toy_model} using TrOOP, we collected the two impulse-response trajectories shown in \cref{fig:toy_model_training_trajectories} and used the $L = 11$ equally spaced samples shown for each trajectory to define the cost function
\begin{equation}
    J(V,W) = \sum_{u_0 \in \{ 0.5, 1.0 \}} \frac{1}{\sum_{l=0}^{L-1} (y\vert_{u_0}(t_l))^2 }\sum_{l=0}^{L-1} \left( \hat{y}\vert_{u_0}(t_l) - y\vert_{u_0}(t_l) \right)^2 + \gamma \rho(V,W),
    \label{eqn:toy_model_cost}
\end{equation}
with $\gamma = 10^{-3}$ (although we note that the results were not sensitive to the choice of $\gamma$).
The normalizing factor in the cost for each trajectory was used to penalize the error relative to the average energy content of the trajectory, rather than in an absolute sense which would be dominated by the trajectory with $u_0 = 1$.
Starting from an initial model formed by balanced truncation, the conjugate gradient algorithm described above with Wolfe conditions defined by $c_1 = 0.01$ and $c_2 = 0.1$ achieved convergence with a gradient magnitude smaller than $10^{-4}$ after $86$ steps.

In \cref{fig:toy_model_training_trajectories}, we see that the resulting reduced-order model trajectories very closely match the trajectories used to find the oblique projection.
Moreover, we tested the predictions of the reduced-order models on $100$ impulse-response trajectories with $u_0$ drawn uniformly at random from the interval $[0,1]$.
The square output prediction errors for each trajectory normalized by the average output energy of the full-order model are shown in \cref{fig:toy_model_testing_performance}.
We observe that the POD-based model is poor regardless of the impulse magnitude
$u_0$, whereas Petrov-Galerkin projection onto subspaces determined by linear balanced truncation performs well when $u_0$ is very
close to~$0$, but poorly when $u_0$ is closer to~$1$.
The projection subspaces obtained by QB-IRKA (and QB balancing) yield models that are accurate in a larger neighborhood of the origin than balanced-truncation of the linearized dynamics, yet still perform poorly for large $u_0$.
On the other hand, the reduced-order model we found using TrOOP produces very accurate predictions for all impulse response magnitudes in the desired range.
This model also has excellent predictive performance with different input signals, even though it was optimized using only two impulse-responses.
For instance, \cref{fig:toy_model_sinusoidal} shows the predictions of the reduced-order models in response to a sinusoidal input $u(t) = \sin(t)$ with zero initial condition.

\begin{figure}
    \centering
    \begin{tikzonimage}[trim=40 0 40 10, clip=true, width=0.7\textwidth]{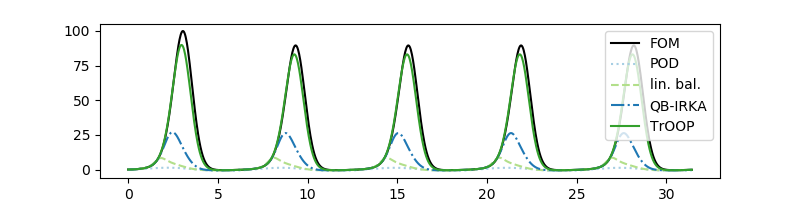}
    \node[rotate=90] at (0.0, 0.5) {\footnotesize $y$};
    \node[rotate=0] at (0.5, -0.05) {\footnotesize $t$};
    \end{tikzonimage}
    \caption{We show the responses of \cref{eqn:toy_model} and the reduced-order models to input $u(t) = \sin(t)$.}
    \label{fig:toy_model_sinusoidal}
\end{figure}

\section{Reduction of a High-Dimensional Nonlinear Fluid Flow}
\label{sec:jet_flow}

\label{subsec:flow_description}
In this section we set out to develop a reduced-order model capable of predicting the response of an incompressible jet flow to impulsive disturbances in the proximity of the nozzle. 
We consider the evolution of an axisymmetric jet flow over the spatial domain $\Omega = \{\left(\xi,z\right)
\lvert\ \xi \in \left[0,L_\xi \right], z \in
\left[0,L_z\right]\}$.
Here, $\xi$ denotes the radial direction and $z$ denotes the axial direction.
Velocities are nondimensionalized by the centerline velocity~$U_0$, lengths by the jet
diameter~$D_0$, and pressure by $\rho U_0^2$, where $\rho$ is the fluid density.
Letting $q = (u,v)$ denote the (dimensionless) velocity vector with axial
component $u$ and radial component $v$, and letting $p$ be the (dimensionless)
pressure field, we may write the governing equations in cylindrical coordinates as
\begin{align}
    \frac{\partial u}{\partial t} = -u\frac{\partial u}{\partial z} - v\frac{\partial u}{\partial \xi} &- \frac{\partial p}{\partial z} + \frac{1}{Re}\left(\frac{1}{\xi}\frac{\partial}{\partial \xi}\left(\xi\frac{\partial u}{\partial \xi}\right) + \frac{\partial^2 u}{\partial z^2} \right) \label{eq:ns_u}\\
    \frac{\partial v}{\partial t} = -u\frac{\partial v}{\partial z} - v\frac{\partial v}{\partial \xi} &- \frac{\partial p}{\partial \xi} + \frac{1}{Re}\left(\frac{1}{\xi}\frac{\partial}{\partial \xi}\left(\xi\frac{\partial v}{\partial \xi}\right) - \frac{v}{\xi^2} + \frac{\partial^2 v}{\partial z^2} \right)\label{eq:ns_v} \\
    \frac{\partial u}{\partial z} + \frac{1}{\xi}\frac{\partial}{\partial \xi}\left(\xi v \right) &= 0 \label{eq:cont},
\end{align}
where $Re = U_0 D_0/\nu$ is the Reynolds number and $\nu$ is the kinematic viscosity of the fluid.
Throughout, we take $Re = 1000$.
The algebraic constraint in formula \eqref{eq:cont} may be used to eliminate pressure from formulas \eqref{eq:ns_u} and \eqref{eq:ns_v},
as discussed in \cref{app:pressure_app}.
We impose zero gradient boundary conditions on the velocity at the top boundary $\xi=L_\xi$ and at the outflow boundary $z=L_z$,  and we let the inflow velocity be
\begin{equation}
    u(\xi,0) = \frac{1}{2}\left(1 - \tanh\left[\frac{1}{4\theta_0}\left(\xi-\frac{1}{\xi}\right)\right] \right),
\end{equation}
where $\theta_0$ is a dimensionless thickness, which we fix at $\theta_0 = 0.025$.
The equations of motion are integrated in time using the fractional step method described in \cite{perot} in conjunction with the second-order Adams-Bashforth multistep scheme.
The spatial discretization is performed on a fully-staggered grid of size $N_z \times N_\xi =  250 \times 200$ and with $L_z = 10$ and $L_\xi = 4$.
If we let the state be composed of the axial and radial velocities at the cell faces, then the state dimension for this flow is $2(N_z\times N_\xi) = 10^5$.
The solver has been validated against some of the results presented in \cite{ardali_jet19}, for which we observed very good quantitative agreement.  
Throughout this section, the inner product on the state space is given by 
\begin{equation}
\label{eq:ns_inner_prod}
    \langle f,g \rangle = \int_{\Omega}f(\xi,z)^Tg(\xi,z) \,\xi \,\td \xi\,\td z.
\end{equation}
This may be transformed into a Euclidean inner product by scaling the elements of the state space by~$\sqrt{\xi}$.
We take our observations, $y$, to be the full velocity field on the spatial grid scaled by $\sqrt{\xi}$.

\subsection{Results}
For the described flow configuration, there exists a stable steady-state solution, which we will denote $Q$. 
Any perturbation $q'$ about the steady-state solution will grow while advecting downstream and it will eventually leave the computational domain through the outflow located at $z = L_z$. 
During the growth process, nonlinear effects become dominant and lead to the formation of complicated vortical structures.
In this section we seek to develop a reduced-order model of the growth
of these disturbances in response to impulses that
enter the radial momentum equation \eqref{eq:ns_v} through a velocity
perturbation localized near $\xi=1/2$ and $z=1$.
In particular, the perturbation
has the form $B(\xi,z)w(t)$, where
\begin{equation}
\label{eq:B_op}
    B(\xi,z) = 
    \exp \left\{-\frac{(\xi - 1/2)^2
      +(z - 1)^2}{\theta_0}\right\}.
\end{equation}
We simulate the response of the flow to a given impulse $w(t) = \alpha \delta(t)$, with $\alpha \in \mathbb{R}$, by integrating the governing equations \eqref{eq:ns_u}--\eqref{eq:cont} with initial condition \begin{equation}
    q(0) = Q + q'(0), \quad \mbox{where} \quad
    q'(0) = (0,B\alpha).
\end{equation}
Here we construct reduced-order models to capture the response of the flow to impulses with $-1.0\leq \alpha \leq 1.0$ from the initial time $t=0$ to a final time ($t\approx 30$) when all disturbances have left the computational domain through the outflow boundary located at $z=L_z$.

\begin{figure}
    \centering
    \subfloat[testing data\label{fig:energy_test}]{
    \begin{tikzonimage}[trim=20 20 40 30, clip=true, width=0.41\textwidth]{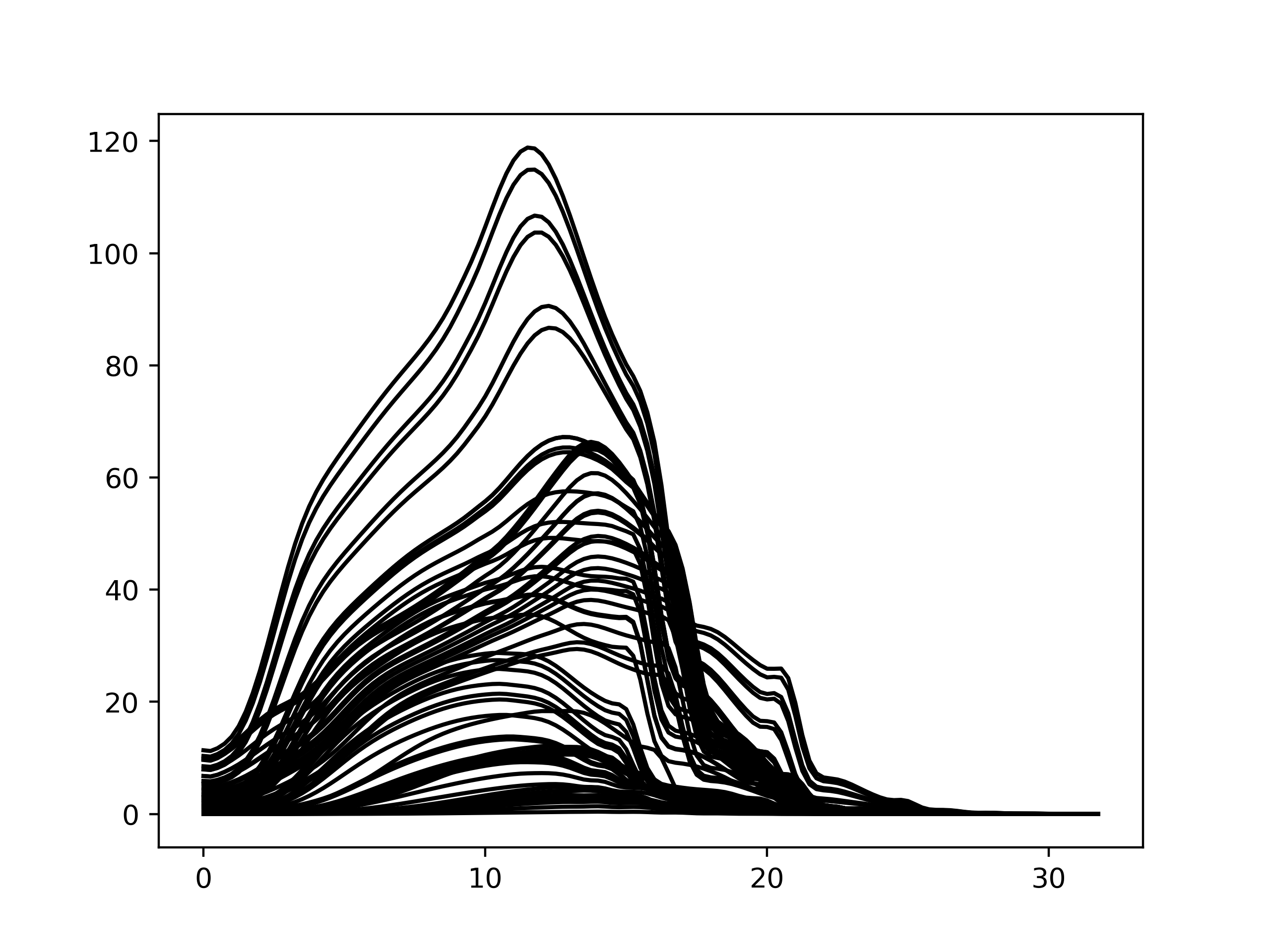}
    \node[rotate=0] at (0.55, -0.03) {\footnotesize $t$};
    \node[rotate=90] at (-0.03,0.5) {\footnotesize Energy $= \lVert y_{m}(t) \rVert^2$};
    \draw [-stealth,color=red] (0.45,0.1) -- (0.2,0.7);
    \node at (0.25,0.75) {$\textcolor{red}{|\alpha|}$};
    \end{tikzonimage}
    }
    \hspace{2ex}
    \subfloat[error on testing trajectories\label{fig:error_plot}]{
    \begin{tikzonimage}[trim=20 20 40 30, clip=true, width=0.41\textwidth]{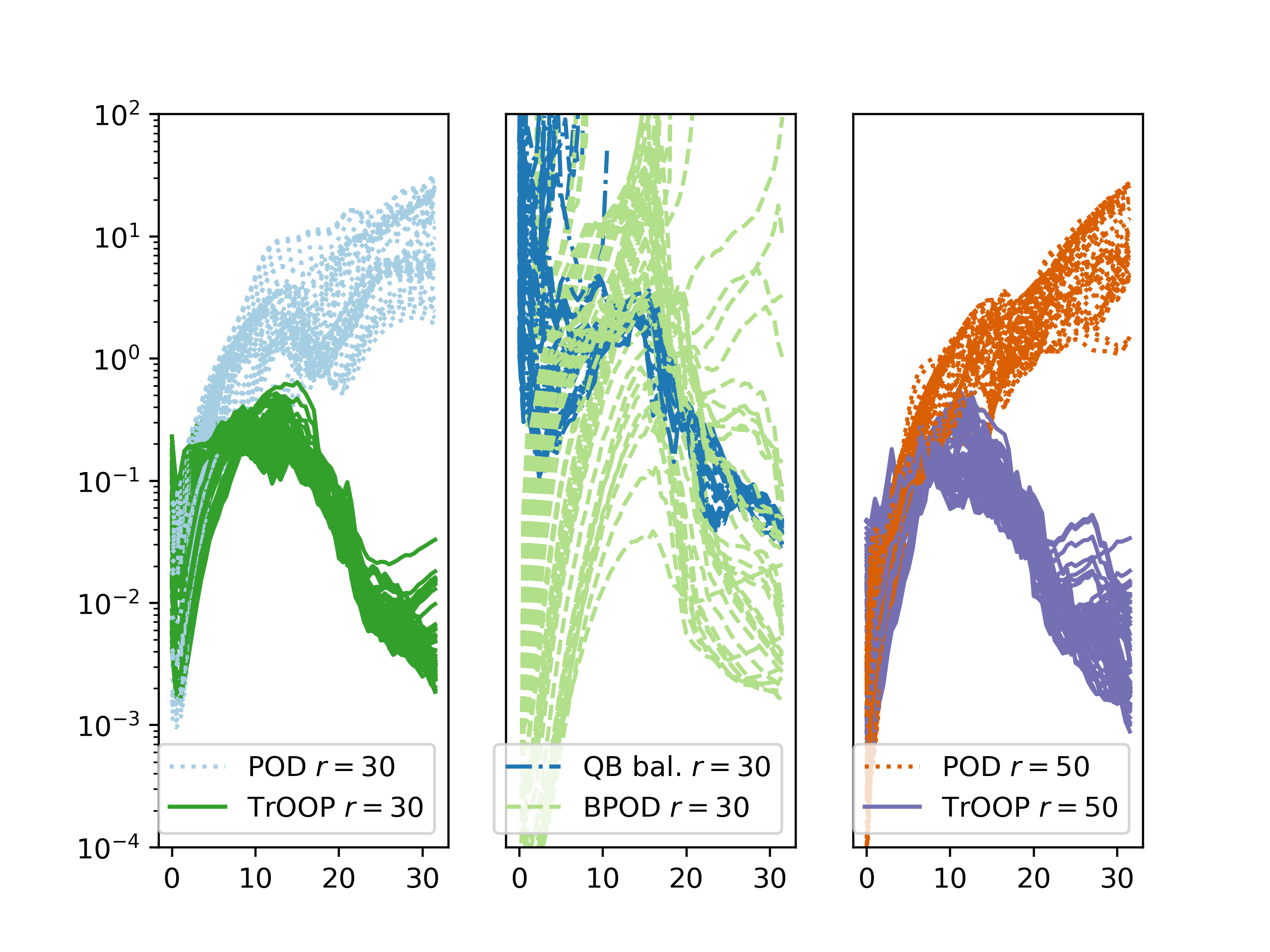}
    \node[rotate=90] at (-0.03,0.5) {\footnotesize $\lVert \hat{y}_{m}(t)-y_{m}(t) \rVert^2/E_m$};
    \node at (0.22,-0.03) {\footnotesize $t$};
    \node at (0.53,-0.03) {\footnotesize $t$};
    \node at (0.85,-0.03) {\footnotesize $t$};
    \end{tikzonimage}
    }
    \caption{In panel (a) we show the time history of the energy of the impulse responses in the training data set.
    In panel (b) we show the square error across all training trajectories for the optimal reduced order model and for the POD-based model with dimensions $r=30$ and $r=50$, and for the BPOD-based model and for the QB-balancing model of dimension $r=30$.}
    \label{fig:energy}
\end{figure}
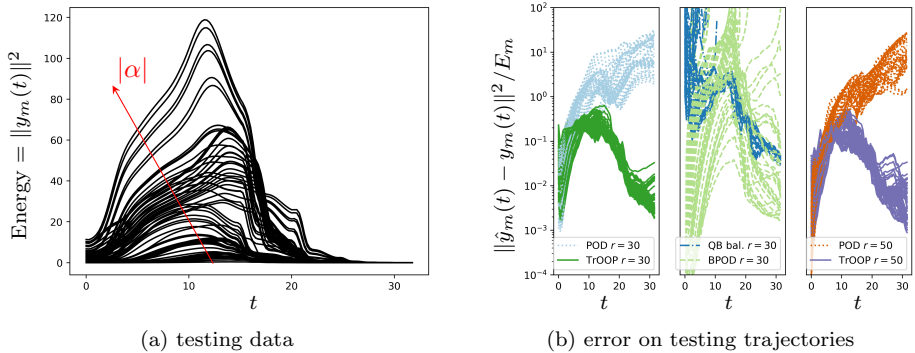
We proceed as follows. 
We generate a training set of $M = 14$ trajectories corresponding to values $\alpha \in \pm \{0.01, 0.1,0.2,0.4,0.6,0.8,1.0\}$, and from each trajectory we observe $L = 64$ equally-spaced snapshots of velocity perturbations about the base flow~$Q$. 
Let $y_{m,l}$ denote the $l$th velocity snapshot in the $m$th trajectory and let $\hat{y}_{m,l}$ denote the corresponding prediction obtained by integrating the reduced-order model from the initial condition $\hat{q}'_{m,0} = P_{V,W}q'_{m,0}$. 
Letting $E_m = L^{-1}\sum_{l=0}^{L-1}\Vert y_{m,l}\Vert^2$ denote the average energy along the $m$th trajectory, we seek to minimize the cost function 
\begin{equation}
\label{eqn:jet_flow_cost_fcn}
    J(V,W) = \frac{1}{ML}\sum_{m = 0}^{M-1}\frac{1}{E_m}\sum_{l = 0}^{L-1}\lVert \hat{y}_{m,l} - y_{m,l}\rVert^2 + \gamma \rho(V,W),
\end{equation}
where $\gamma = 10^{-3}$.
The optimization was carried out using \cref{alg:conj_grad_alg} with a $r$-dimensional model obtained by POD of all available training snapshots as the initial guess. 
The integrals in \cref{alg:gradient2} were computed using Gauss-Legendre quadrature with two Gauss-Legendre points between adjacent FOM datapoints. 
Here, we train two models: one with $r=30$ and one with $r=50$, where the first $30$ POD modes accounted for $98.6\%$ of the training energy, while the first $50$ accounted for $99.6\%$ of the training energy.

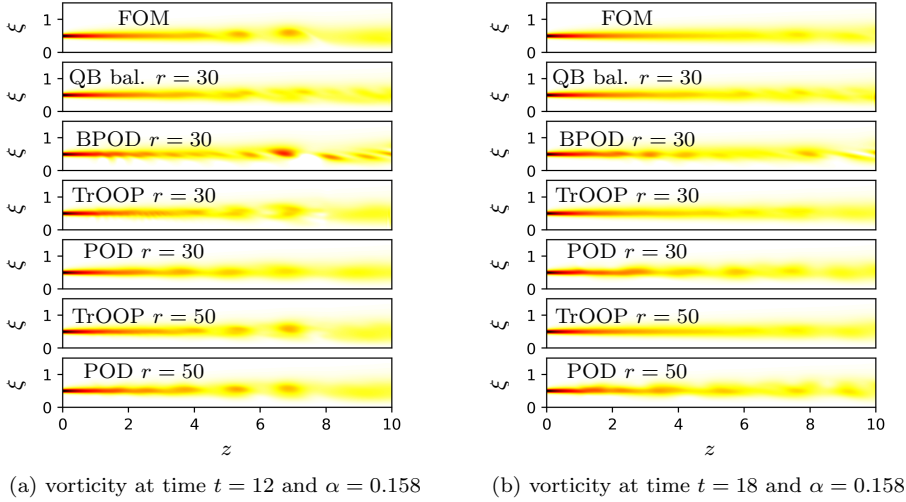
\begin{figure}
    \centering
    \subfloat[vorticity at time $t=12$ and $\alpha=0.158$\label{fig:snap_low_0}]{
    \begin{tikzonimage}[trim=100 10 90 30, clip=true, width=0.42\textwidth]{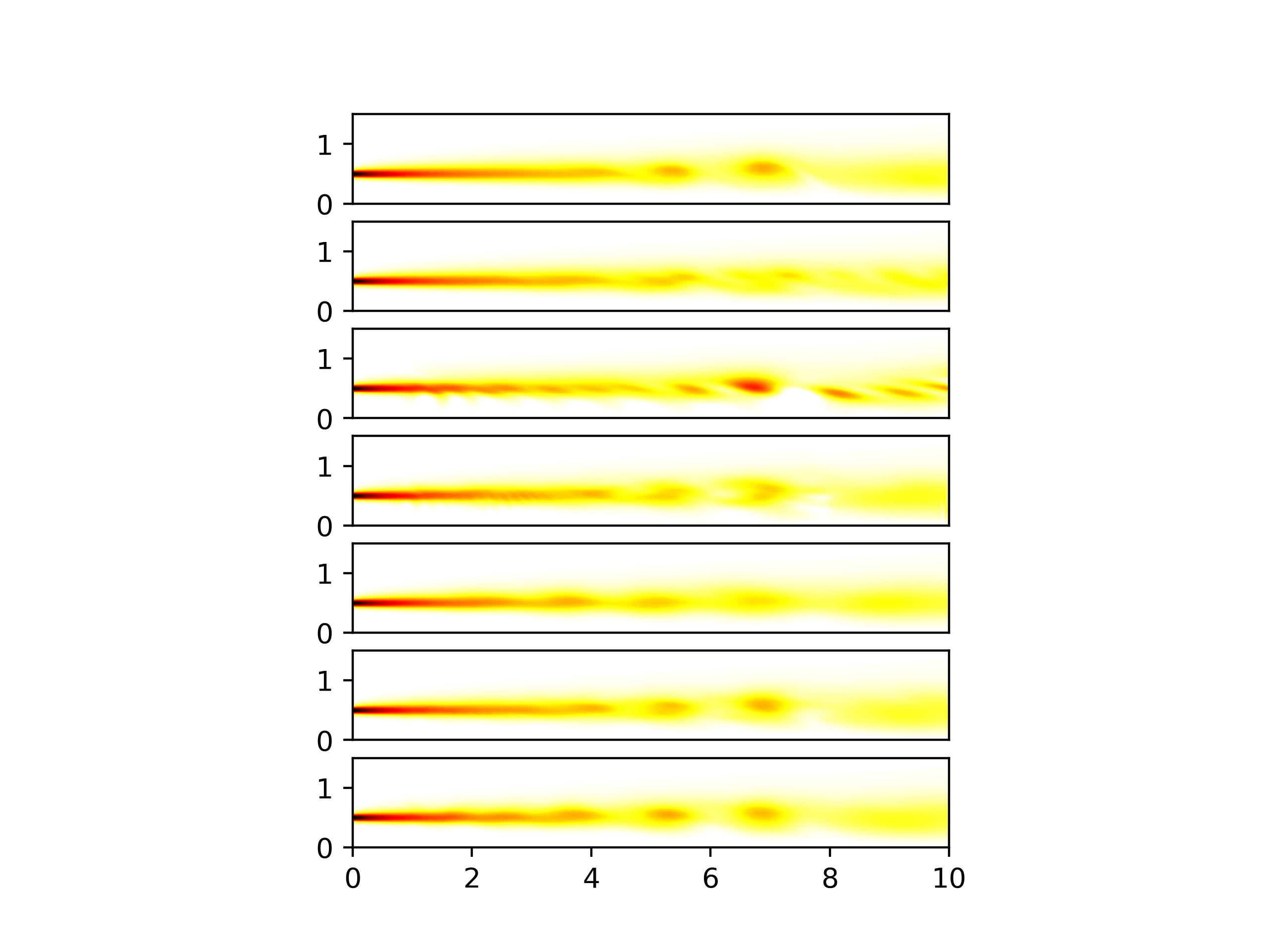}
    \node[rotate=0] at (0.5, 0.0) {\footnotesize $z$};
    \node[rotate=90] at (-0.01,0.14) {\footnotesize $\xi$};
    \node[rotate=90] at (-0.01,0.27) {\footnotesize $\xi$};
    \node[rotate=90] at (-0.01,0.91) {\footnotesize $\xi$};
    \node[rotate=90] at (-0.01,0.76) {\footnotesize $\xi$};
    \node[rotate=90] at (-0.01,0.64) {\footnotesize $\xi$};
    \node[rotate=90] at (-0.01,0.52) {\footnotesize $\xi$};
    \node[rotate=90] at (-0.01,0.40) {\footnotesize $\xi$};
    \node at (0.3,0.17) {\footnotesize POD $r=50$};
    \node at (0.3,0.29) {\footnotesize TrOOP $r=50$};
    \node at (0.3,0.43) {\footnotesize POD $r=30$};
    \node at (0.3,0.55) {\footnotesize TrOOP $r=30$};
    \node at (0.3,0.67) {\footnotesize BPOD $r=30$};
    \node at (0.3,0.80) {\footnotesize QB bal. $r=30$};
    \node at (0.3,0.93) {\footnotesize FOM};
    \end{tikzonimage}
    }
    \hspace{2ex}
    \subfloat[vorticity at time $t=18$ and $\alpha=0.158$\label{fig:snap_low_1}]{
    \begin{tikzonimage}[trim=100 10 90 30, clip=true, width=0.42\textwidth]{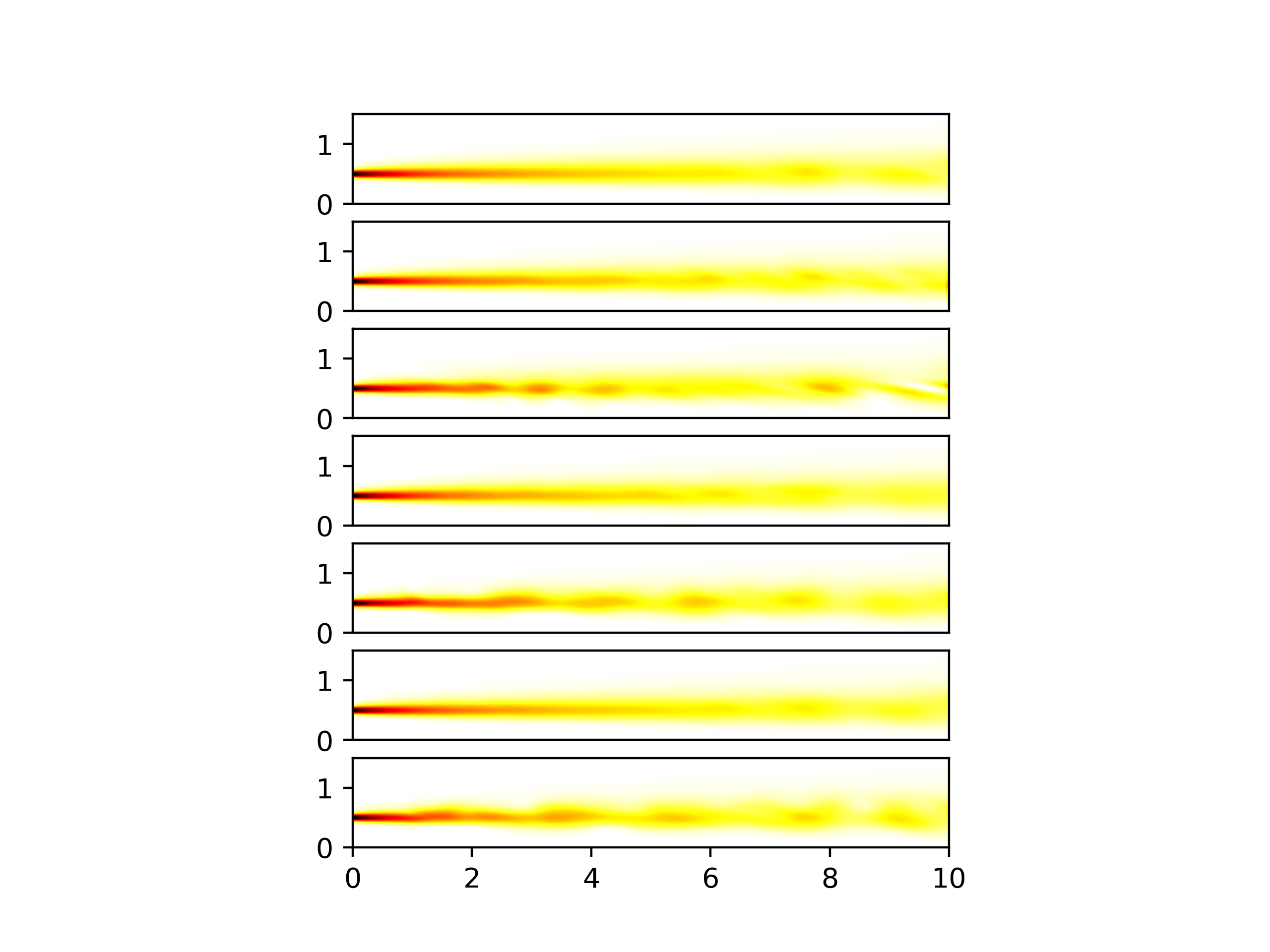}
    \node[rotate=0] at (0.5, 0.0) {\footnotesize $z$};
    \node[rotate=90] at (-0.01,0.14) {\footnotesize $\xi$};
    \node[rotate=90] at (-0.01,0.27) {\footnotesize $\xi$};
    \node[rotate=90] at (-0.01,0.91) {\footnotesize $\xi$};
    \node[rotate=90] at (-0.01,0.76) {\footnotesize $\xi$};
    \node[rotate=90] at (-0.01,0.64) {\footnotesize $\xi$};
    \node[rotate=90] at (-0.01,0.52) {\footnotesize $\xi$};
    \node[rotate=90] at (-0.01,0.40) {\footnotesize $\xi$};
    \node at (0.3,0.17) {\footnotesize POD $r=50$};
    \node at (0.3,0.29) {\footnotesize TrOOP $r=50$};
    \node at (0.3,0.43) {\footnotesize POD $r=30$};
    \node at (0.3,0.55) {\footnotesize TrOOP $r=30$};
    \node at (0.3,0.67) {\footnotesize BPOD $r=30$};
    \node at (0.3,0.80) {\footnotesize QB bal. $r=30$};
    \node at (0.3,0.93) {\footnotesize FOM};
    \end{tikzonimage}
    }
    \caption{In panel (a) we show a vorticity snapshot (i.e., $\nabla\times (q'+Q)$) at time $t=12$ from the trajectory generated by the impulse with $\alpha = 0.158$. In panel (b) we show the analog of panel (a) with $t=18$ and $\alpha = 0.158$. The colorbar ranges from $0$ (white) to approximately $9$ (red).}
    \label{fig:snaps_low}
\end{figure}

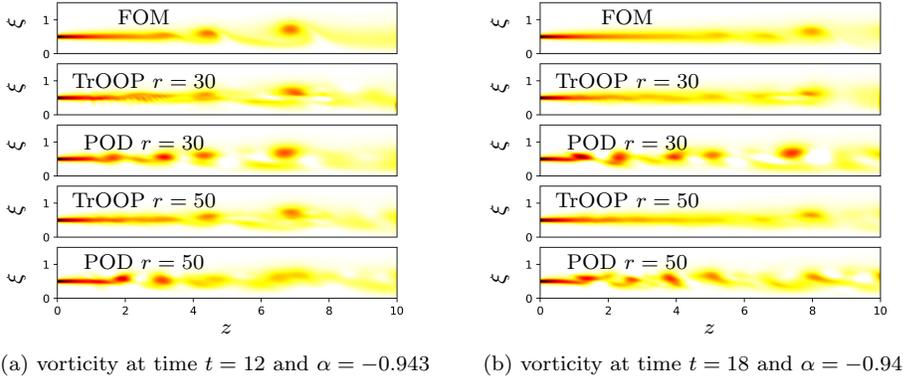
\begin{figure}
    \centering
    \subfloat[vorticity at time $t=12$ and $\alpha=-0.943$\label{fig:snap_high_0}]{
    \begin{tikzonimage}[trim=50 10 40 30, clip=true, width=0.42\textwidth]{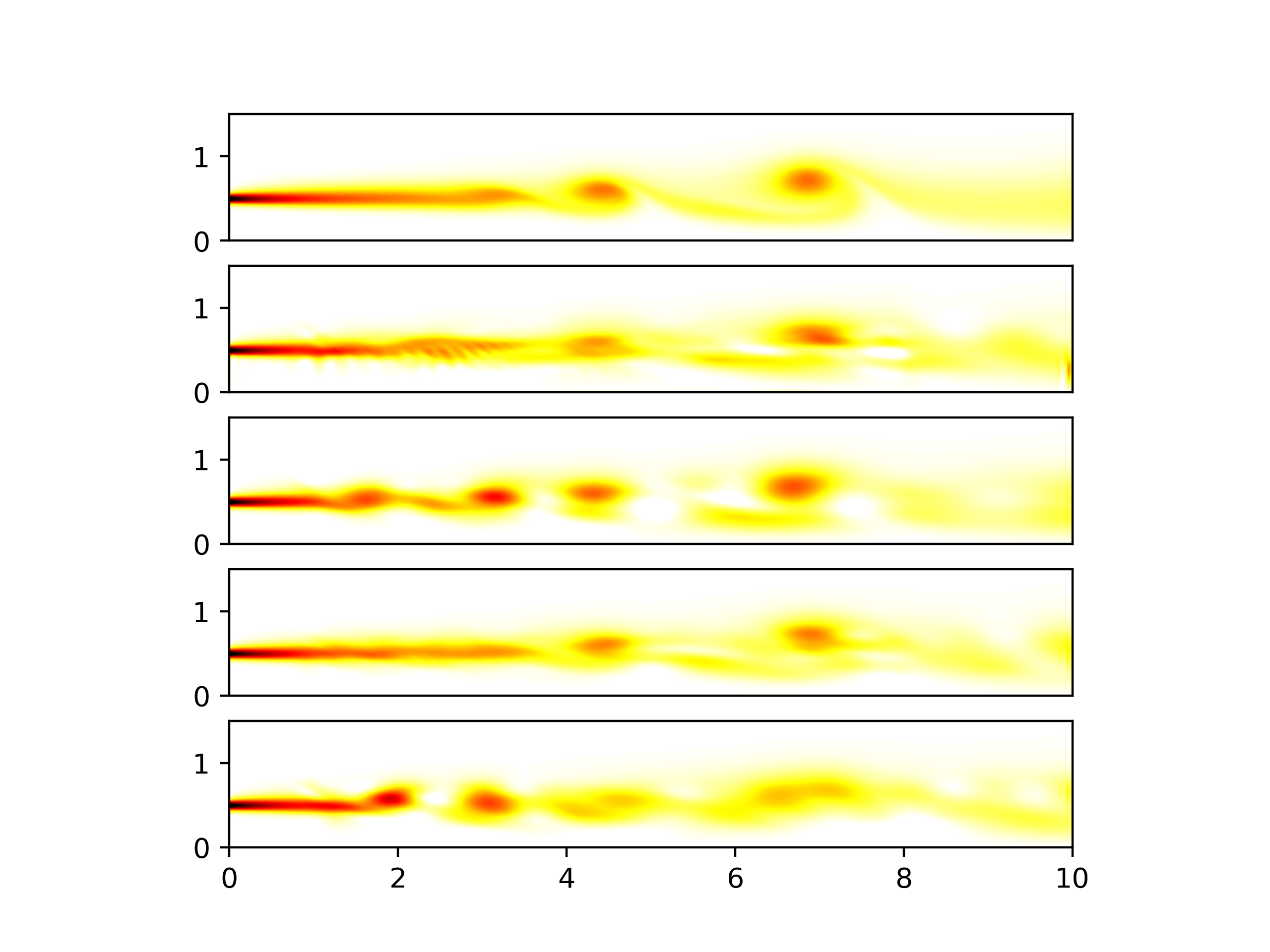}
    \node[rotate=0] at (0.5, 0.0) {\footnotesize $z$};
    \node[rotate=90] at (-0.01,0.15) {\footnotesize $\xi$};
    \node[rotate=90] at (-0.01,0.35) {\footnotesize $\xi$};
    \node[rotate=90] at (-0.01,0.54) {\footnotesize $\xi$};
    \node[rotate=90] at (-0.01,0.71) {\footnotesize $\xi$};
    \node[rotate=90] at (-0.01,0.90) {\footnotesize $\xi$};
    \node at (0.3,0.20) {\footnotesize POD $r=50$};
    \node at (0.3,0.38) {\footnotesize TrOOP $r=50$};
    \node at (0.3,0.55) {\footnotesize POD $r=30$};
    \node at (0.3,0.73) {\footnotesize TrOOP $r=30$};
    \node at (0.3,0.92) {\footnotesize FOM};
    \end{tikzonimage}
    }
    \hspace{2ex}
    \subfloat[vorticity at time $t=18$ and $\alpha=-0.943$\label{fig:snap_high_1}]{
    \begin{tikzonimage}[trim=50 10 40 30, clip=true, width=0.42\textwidth]{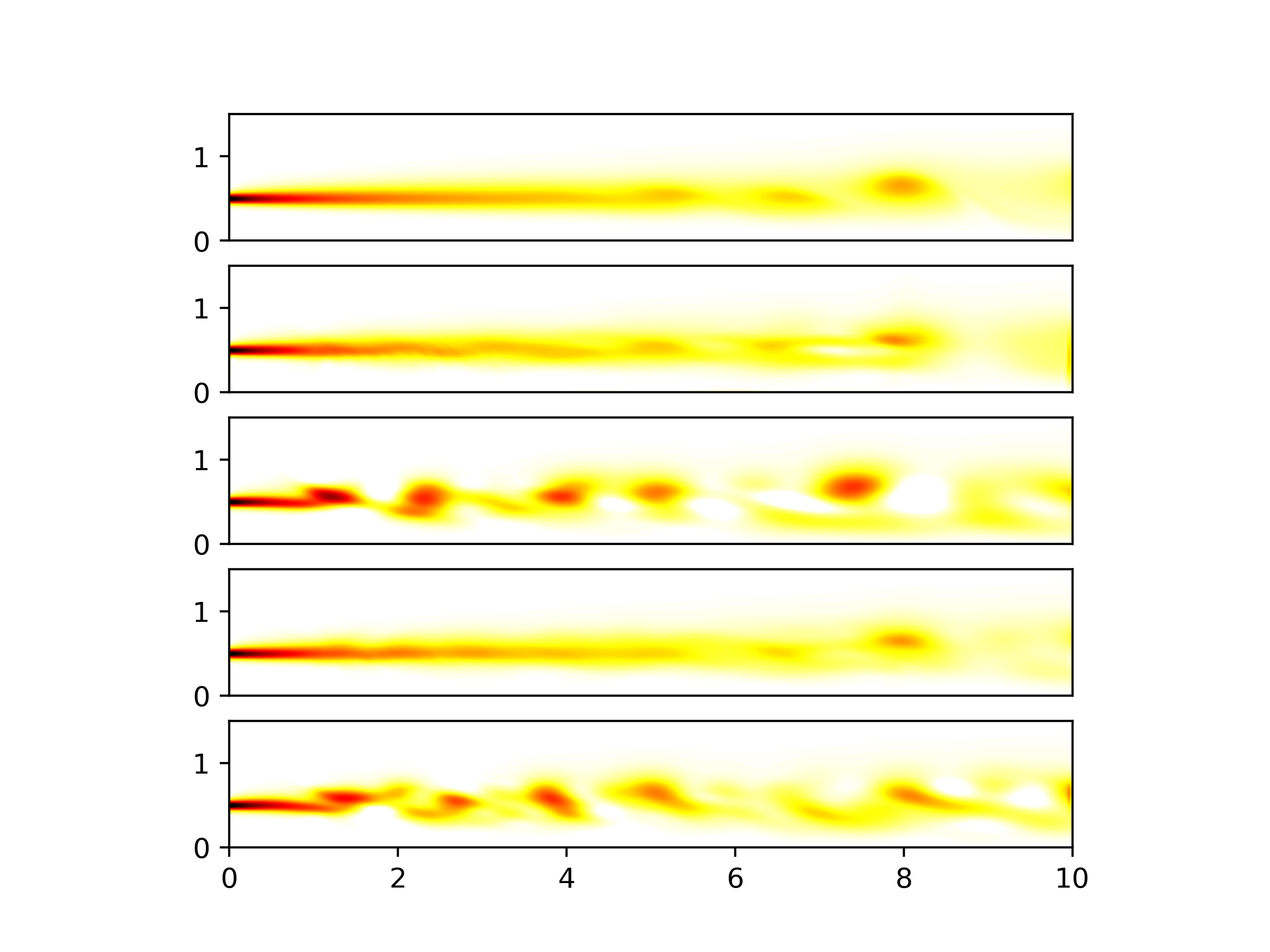}
    \node[rotate=0] at (0.5, 0.0) {\footnotesize $z$};
    \node[rotate=90] at (-0.01,0.15) {\footnotesize $\xi$};
    \node[rotate=90] at (-0.01,0.35) {\footnotesize $\xi$};
    \node[rotate=90] at (-0.01,0.54) {\footnotesize $\xi$};
    \node[rotate=90] at (-0.01,0.71) {\footnotesize $\xi$};
    \node[rotate=90] at (-0.01,0.90) {\footnotesize $\xi$};
    \node at (0.3,0.20) {\footnotesize POD $r=50$};
    \node at (0.3,0.38) {\footnotesize TrOOP $r=50$};
    \node at (0.3,0.55) {\footnotesize POD $r=30$};
    \node at (0.3,0.73) {\footnotesize TrOOP $r=30$};
    \node at (0.3,0.92) {\footnotesize FOM};
    \end{tikzonimage}
    }
    \caption{Analog of \cref{fig:snaps_low} except with $\alpha = -0.943$. The BPOD and QB balancing predictions are not shown because they blew up after a few time steps.}
    \label{fig:snaps_high}
\end{figure}

We compare the models obtained using TrOOP to projection-based models of the same dimension using subspaces determined by POD, BPOD, and QB-balancing on a set of $M = 65$ unseen impulse responses.
For $50$ of these, $\alpha$ was drawn uniformly at random from $[-1.0,1.0]$, while the remaining $15$ were drawn uniformly from $[-0.1,0.1]$.
The energy content of the testing set is shown in \cref{fig:energy_test}.
Observe the range of behavior for different values
of~$\alpha$, reflecting the strong nonlinearity of this flow.
The BPOD model was obtained following the procedure discussed in \cite{Rowley2005model}, with a $30$-dimensional output projection that accounted for more than $99.9\%$ of the energy. 
Our implementation of the QB balancing method is discussed in \cref{app:qb_jet}.

The performance of each model is shown in \cref{fig:error_plot}, with the exception of BPOD and QB-balancing with $r=50$.
The predictions from these models blew up for virtually all amplitudes $\alpha$.
\Cref{fig:error_plot} shows that the POD/Galerkin models accurately represent the initial growth of the perturbations at all amplitudes, but they perform poorly at long time. 
The QB-balancing model exhibits large errors at initial times and it blows up for many of the testing trajectories with larger values of $\alpha$. 
The BPOD-based model performs very well for small amplitudes $\alpha$, but it too performs poorly or even blows up for larger values of $\alpha$.
By contrast, the models obtained using TrOOP are accurate over the entire time horizon at every amplitude and capture the initial transient growth of the perturbations as well as the long-time decay. 

Snapshots extracted from the trajectories with $\alpha = 0.158$ and $\alpha = -0.943$ are respectively shown in \cref{fig:snaps_low} and \cref{fig:snaps_high}. 
Results are not shown for BPOD and QB-balancing in \cref{fig:snaps_high} because these models blew up after a few time steps at the higher amplitude.
At both amplitudes the optimized models correctly predict the location and strength of the vortical structures that form in response to the initial impulse,
while the other reduced-order models exhibit spurious vortical structures or begin to lose predictive accuracy at long times. 
\begin{figure}
    \centering
    \subfloat[forcing with $-0.3 B(r,z)\cos (t/2)$\label{fig:sin_small}]{
    \begin{tikzonimage}[trim=20 70 40 80, clip=true, width=0.42\textwidth]{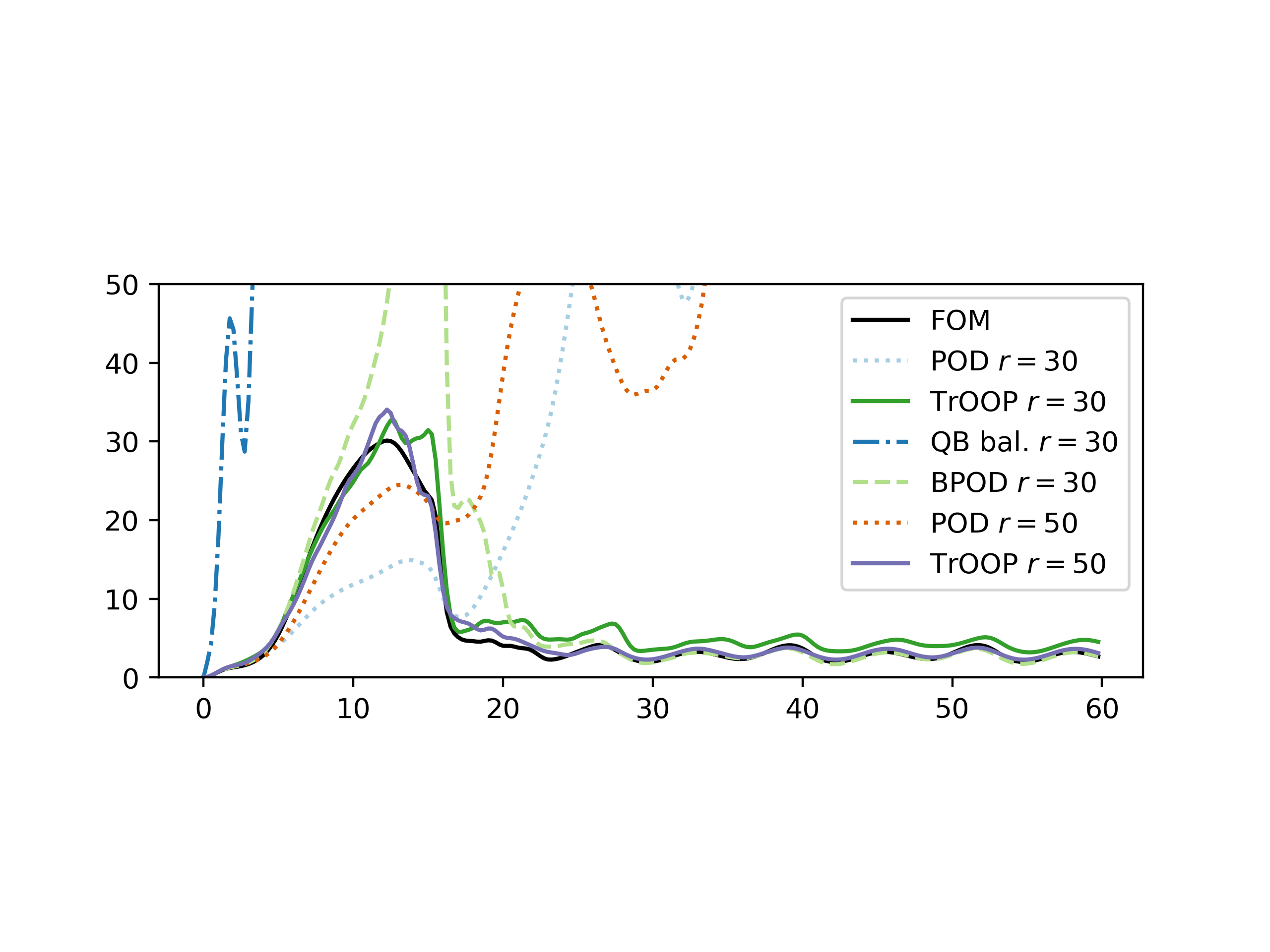}
    \node[rotate=0] at (0.55, 0.0) {\footnotesize $t$};
    \node[rotate=90] at (-0.04,0.55) {\footnotesize Energy $=\lVert y(t) \rVert^2$};
    \end{tikzonimage}
    }
    \hspace{2ex}
    \subfloat[forcing with $-0.5 B(r,z)\cos (t/2)$\label{fig:sin_large}]{
    \begin{tikzonimage}[trim=20 70 40 80, clip=true, width=0.42\textwidth]{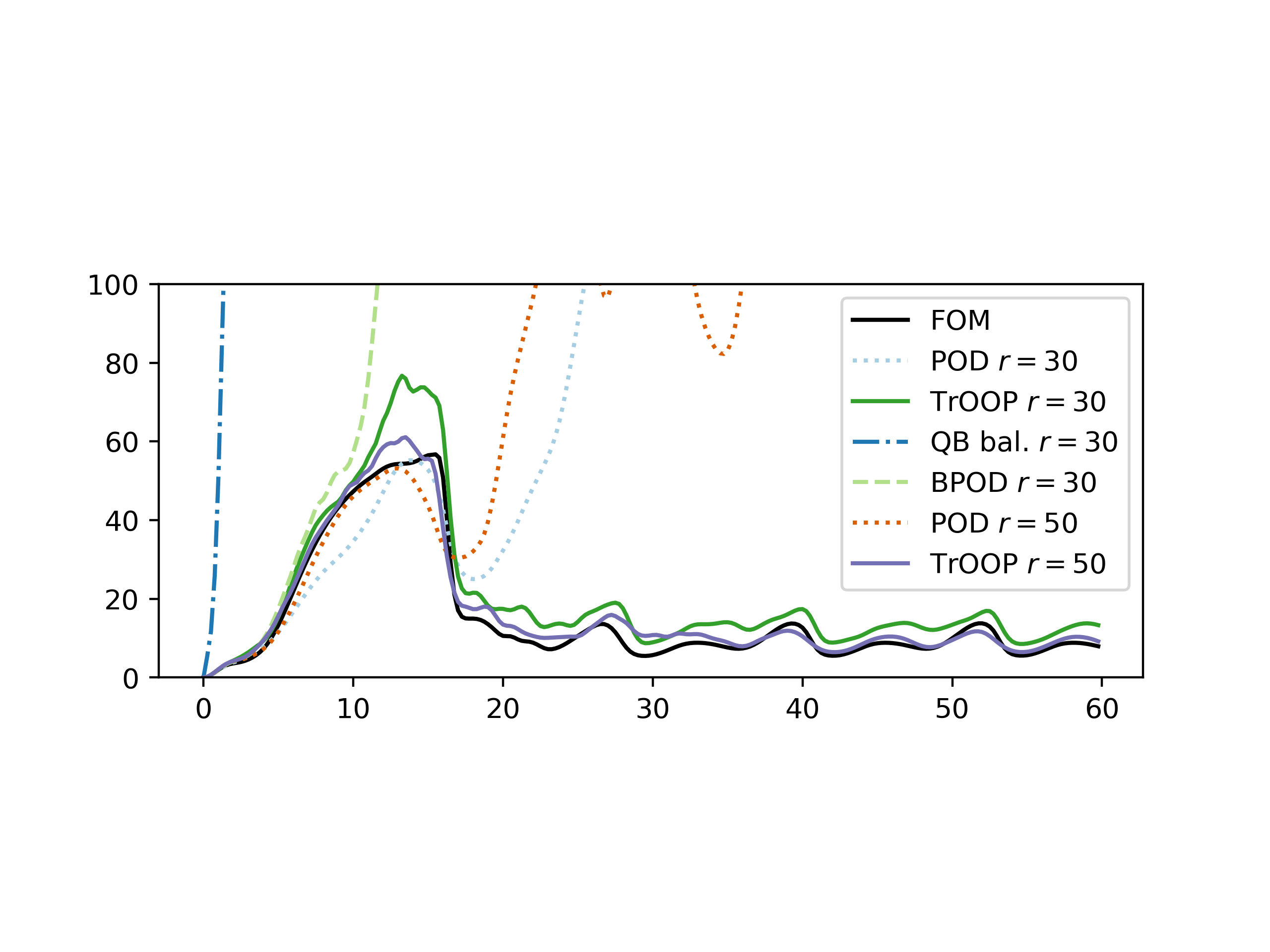}
    \node[rotate=0] at (0.55, 0.0) {\footnotesize $t$};
    \node[rotate=90] at (-0.04,0.55) {\footnotesize Energy $=\lVert y(t) \rVert^2$};
    \end{tikzonimage}
    }
    \caption{In panel (a) we show the energy of the response of the flow to a radial velocity input $w(r,z,t) = -\beta B(r,z)\cos\left(t/2\right)$ with $\beta = -0.3$. Panel (b) is the analog of panel (a) with $\beta = -0.5$.}
    \label{fig:sin}
\end{figure}

The models found using TrOOP are also able to predict the response of the flow to other types of input signals.
For example, we consider a radial velocity input of the form $w(r,z,t) = -\beta B(r,z)\cos\left(t/2\right)$ 
for values $\beta = 0.3$ and $\beta = 0.5$. 
The results are shown in \cref{fig:sin_small} and \cref{fig:sin_large}, where we plot the predicted energy of the velocity field over time.
For both amplitudes, our models correctly capture the qualitative nature of the response of the flow, and the $50$-dimensional model also exhibits very good quantitative agreement. 
By contrast, the QB-balancing model ``blows up" at early times, both POD/Galerkin models blow up at later times, and the BPOD model either exhibits extremely large transient growth at $\beta =0.3$ or it blows up for $\beta = 0.5$.
It is worth mentioning that the $30$-dimensional BPOD model has excellent performance on the low-amplitude post-transient response.

\subsection{Computational cost and considerations}
Here, we provide a brief comparison of the computational costs of each method for the jet flow in terms of the number of times an object resembling the right-hand side of the full-order model is evaluated, i.e, $f(x, u)$, $(\partial f(x,u)/\partial x) v$, or $(\partial f(x,u)/\partial x)^T v$ acting on a single vector $v\in\mathbb{R}^n$.
Such evaluations dominate the computational cost of each method we considered.
We recall that TrOOP assembles the reduced-order model at each line search iteration using queries to $f(x, u)$, while the gradient is computed using \cref{alg:gradient2} by querying $f(x, u)$ and $(\partial f(x,u)/\partial x)^T v$ at quadrature points along each trajectory.
Since solving the Lyapunov equations for balanced truncation \cite{Moore1981principal} and its quadratic bilinear extension \cite{Benner2017balanced} were infeasible for our system with $10^5$ states, we used BPOD \cite{Rowley2005model} and an analogous snapshot-based approximation for QB-balancing described in \cref{app:qb_jet} involving similar queries.
In fact, to the best of our knowledge, the QB-balancing algorithm has never been applied to systems with state dimension larger than $\sim 10^3$.
\Cref{tab:cost_comparison} summarizes the total cost of each method, and \cref{fig:convergence_plot} shows the progress of the conjugate gradient algorithm against the number of iterations and the total number of FOM-like evaluations.

\begin{figure}
    \centering
    \begin{tikzonimage}[trim=20 90 20 100, clip=true, width=0.47\textwidth]{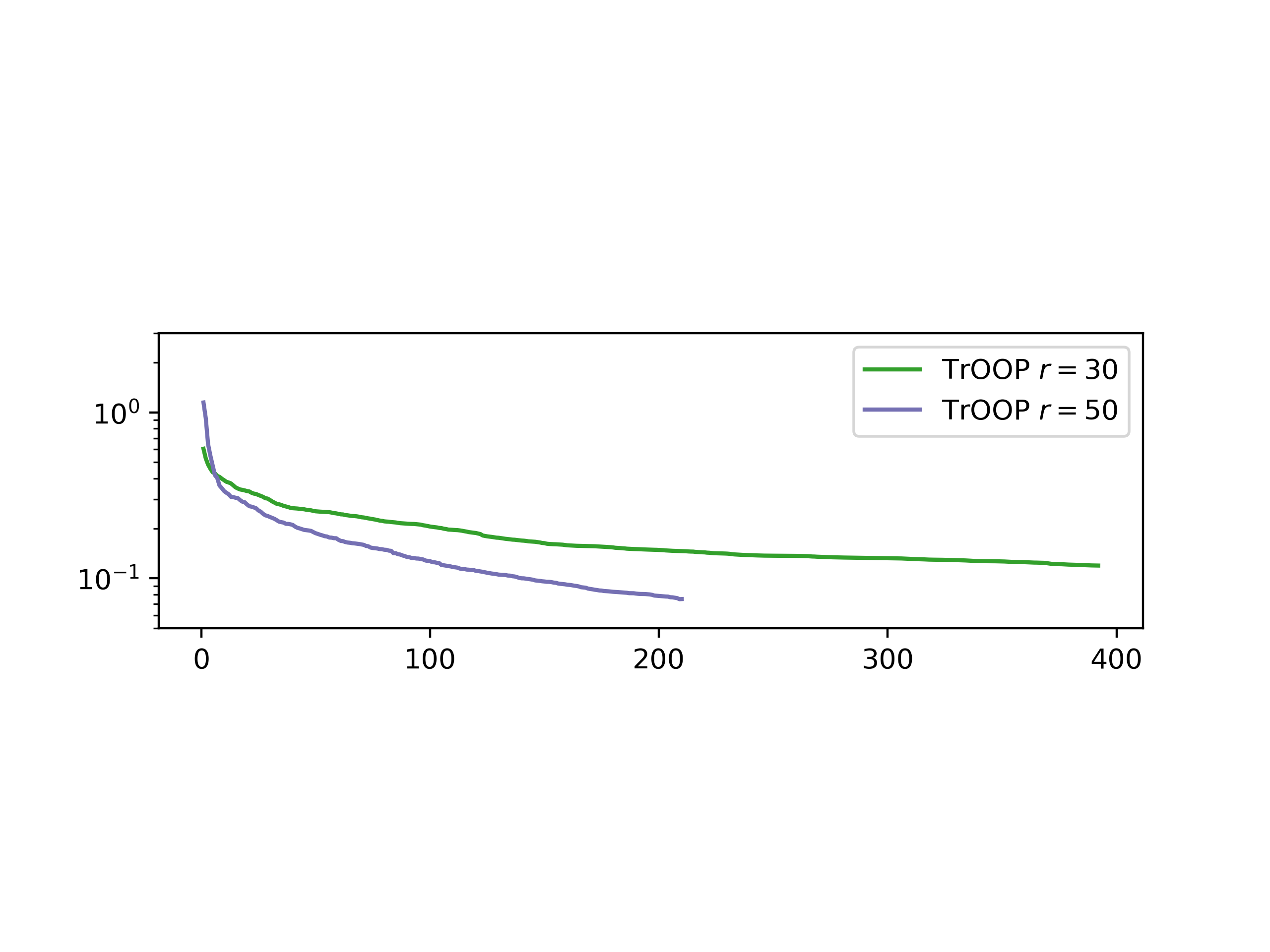}
        \node[rotate=0] at (0.525, 0.0) {\footnotesize iterations};
        \node[rotate=90] at (-0.01, 0.52) {\footnotesize objective};
    \end{tikzonimage}
    \begin{tikzonimage}[trim=20 90 20 100, clip=true, width=0.47\textwidth]{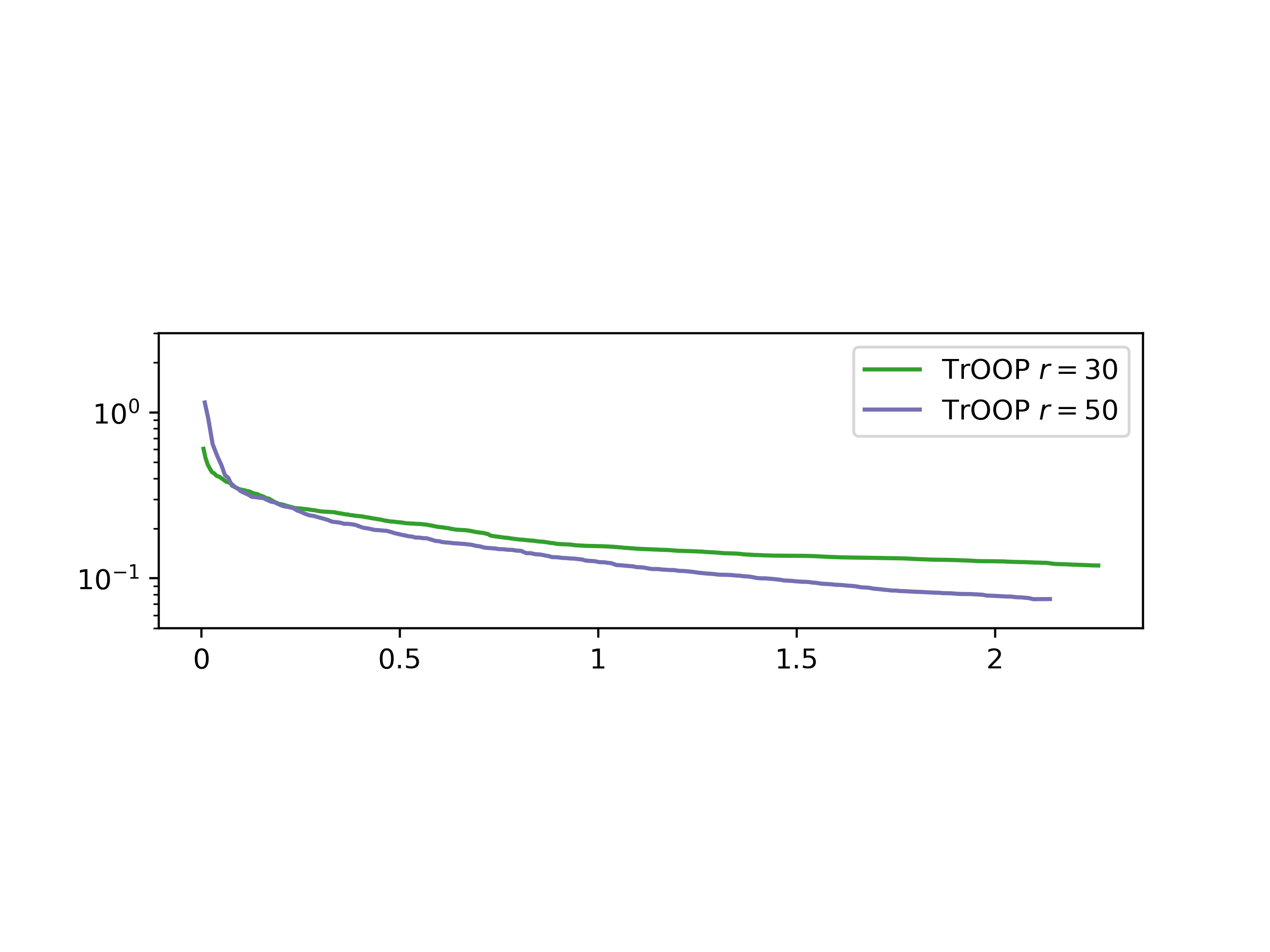}
        \node[rotate=0] at (0.525, 0.0) {\footnotesize FOM-like evaluations};
        \node[rotate=0] at (0.95, 0.11) {\tiny $(\times 10^6)$};
    \end{tikzonimage}
    \caption{Values of the jet flow optimization objective \cref{eqn:jet_flow_cost_fcn} versus conjugate gradient iterations (left) and FOM-like evaluations (right).}
    \label{fig:convergence_plot}
\end{figure}

\begin{table}[]
    \centering
    \begin{tabular}{c|c|c|c|c|c}
        method & FOM sim. & BPOD & QB bal. & TrOOP $r=30$ & TrOOP $r=50$ \\
        \hline
         FOM evals. & $6\times 10^3$ & $2\times 10^5$ & $1.3\times 10^6$ & $2.3\times 10^6$ & $2.1\times 10^6$
    \end{tabular}
    \caption{We compare the number of FOM-like evaluations for each model-reduction technique on the jet flow to a single simulation of the FOM from $t=0$ to $t=30$ using the time step $\Delta t = 5\times 10^{-3}$ and the second-order Adams-Bashforth method.}
    \label{tab:cost_comparison}
\end{table}

\section{Conclusions}
\label{sec:conclusions}
We have introduced a reduced-order modeling approach for large-scale nonlinear dynamical systems based on optimizing oblique projections of the governing equations to minimize prediction error over sampled trajectories.
We implemented a provably convergent geometric conjugate gradient algorithm in order to optimize a regularized trajectory prediction error over the product of Grassmann manifolds defining the projection operators. 
The method, referred to as Trajectory-based Optimization for Oblique Projections (TrOOP), is compared to existing projection-based reduced-order modeling techniques where the projection subspaces are found using POD, balanced truncation, and techniques for quadratic bilinear systems.
We considered a simple three-dimensional system with an important low-energy feature as well as a nonlinear axisymmetric jet flow with $10^5$ state variables.
In both cases, the models obtained using TrOOP vastly outperform the models obtained using other methods in the highly nonlinear regimes far away from equilibria, while achieving comparable performance to the best alternatives near equilibria.
The primary limitation of our approach is that a sufficiently large collection of trajectories must be used to sample the system's behavior and avoid over-fitting.
Based on algebraic considerations, the total number of sample data should exceed
the dimension $2nr - 2r^2$ of the product of Grassmann manifolds over which we
optimize, where $n$ is the state dimension and $r$ is the dimension of the
reduced-order model.  
When a large number of trajectories are used, it may be advantageous to employ a stochastic gradient descent algorithm \cite{Bonnabel2013stochastic, Sato2019riemannian} with randomized ``minibatches'' of trajectories.


\bibliographystyle{siamplain}
\bibliography{jabbrv,references}

\appendix

\section{Proof of \cref{thm:topology_of_P} (Topology of $\mathcal{P}$)}
\label{app:topology_of_P}
To prove that $\mathcal{P}$ is open in $\mathcal{G}_{n,r}\times \mathcal{G}_{n,r}$, consider the function
\begin{equation}
    F\circ \pi (\Phi, \Psi) = \frac{\det ( \Psi^T \Phi )^2}{\det(\Phi^T \Phi) \det(\Psi^T \Psi)}, \qquad (\Phi, \Psi)\in \mathbb{R}_*^{n\times r}\times \mathbb{R}_*^{n\times r}.
\end{equation}
The function $F$ is well-defined because the above expression does not depend on the representatives $\Phi, \Psi$ due to the product rule for determinants.
Moreover, $F$ is smooth, and hence continuous, thanks to the smoothness of $F\circ \pi$ and Theorem~4.29 in J. M. Lee \cite{Lee2013introduction}.
As can easily be shown, two subspaces $V,W \in\mathcal{G}_{n,r}$ satisfy $V\oplus W^{\perp} = \mathbb{R}^n$, and hence define an oblique projection operator if and only if every pair of matrix representatives $(\Phi, \Psi) \in \pi^{-1}(V,W)$ satisfy $\det(\Psi^T\Phi) \neq 0$.
Consequently, $\mathcal{P}$ is open because it is the pre-image $\mathcal{P} = F^{-1}((0,\infty))$ of the open set $(0,\infty)$ under the continuous function $F$.

To prove that $\mathcal{P}$ is dense in $\mathcal{G}_{n,r}\times \mathcal{G}_{n,r}$, consider a pair of subspaces $(V,W)\in (\mathcal{G}_{n,r}\times \mathcal{G}_{n,r})\setminus \mathcal{P}$ and representatives $(\Phi, \Psi)\in \mathbb{R}_*^{n\times r}\times \mathbb{R}_*^{n\times r}$ such that $\pi(\Phi, \Psi) = (V,W)$.
Consider the full-sized singular value decomposition
\begin{equation}
    \Phi^T \Psi = U\Sigma Q^T
\end{equation}
and define the continuously parameterized set of matrices
\begin{equation}
    \Psi_t = \Psi + t \Phi (\Phi^T \Phi)^{-1} U Q^T, \qquad t\geq 0,
\end{equation}
giving rise to a continuously parameterized set of subspaces $(V, W_t) = \pi(\Phi, \Psi_t) \in \mathcal{G}_{n,r}\times \mathcal{G}_{n,r}$.
We observe that for all $t > 0$, we have
\begin{equation}
    \det(\Phi^T \Psi_t) = \det\left( U\Sigma Q^T + t U Q^T \right) = 
    \det(U) \det(Q) \det\left(\Sigma + t I\right) \neq 0.
\end{equation}
Therefore, $(V, W_0) = (V,W) \notin \mathcal{P}$, but $(V, W_t) \in \mathcal{P}$ for all $t > 0$, from which it follows that $\mathcal{P}$ is dense in $\mathcal{G}_{n,r}\times \mathcal{G}_{n,r}$.

Since we are working with manifolds, connectedness and path-connectedness are equivalent.
In order to prove the connectedness part of \cref{thm:topology_of_P}, we will need the following result:
\begin{lemma}
\label{lem:R_star_is_connected}
If $1\leq r < n$, then $\mathbb{R}_*^{n\times r}$ is connected.
\end{lemma}
\begin{proof}
We shall connect any two matrices $\Phi_0, \Phi_1 \in \mathbb{R}_*^{n\times r}$ by a continuous path.
Let $T_0, T_1 \in GL_n$ be invertible matrices whose first $r$ columns are given by $\Phi_0$ and $\Phi_1$ respectively. 
Since $r < n$ we may choose the sign of the last column of $T_0$ so that $\sgn\det{T_0} = \sgn\det{T_1}$.
Since the general linear group $GL_n$ has two connected components, corresponding to matrices with positive and negative determinants \cite{Hall2015lie}, there is a continuous path $t\mapsto T_t \in GL_n$ connecting $T_0$ at $t=0$ to $T_1$ at $t=1$.
Finally we observe that the first $r$ columns of the matrices $T_t \in GL_n$ form a continuous path connecting $\Phi_0$ and $\Phi_1$ in $\mathbb{R}_*^{n\times r}$.
\end{proof}

To prove that $\mathcal{P}$ is connected we need only consider the case when $r < n$ since when $r=n$ we have $\mathcal{P} = \mathcal{G}_{n,n}\times \mathcal{G}_{n,n}$, which consists of a single element $(\mathbb{R}^n, \mathbb{R}^n)$.
Choose any $(V_0, W_0),\ (V_1, W_1)\in\mathcal{P}$ and let $(\Phi_0, \Psi_0)$ and $(\Phi_1, \Psi_1)$ be representatives of $(V_0, W_0)$ and $(V_1, W_1)$ respectively such that $\Psi_0^T \Phi_0 = I_r$ and $\Psi_1^T \Phi_1 = I_r$.
Such representatives may always be found by first choosing any representatives $\tilde{\Phi}_0, \tilde{\Psi}_0$ of the subspaces $V_0, W_0$ and then letting $\Phi_0 = \tilde{\Phi}_0$ and $\Psi_0 = \tilde{\Psi}_0 (\Phi_0^T \tilde{\Psi}_0)^{-1}$.
We may do the same for $\Phi_1$ and $\Psi_1$.
In order to construct our path in $\mathcal{P}$, we first consider any continuously parameterized matrices $(\tilde{\Phi}_t, \tilde{\Psi}_t) \in \mathbb{R}_*^{n\times r}\times \mathbb{R}_*^{n\times r}$, $0\leq t \leq 1$ furnished by \cref{lem:R_star_is_connected} such that $(\tilde{\Phi}_0, \tilde{\Psi}_0) = (\Phi_0, \Psi_0)$ and $(\tilde{\Phi}_1, \tilde{\Psi}_1) = (\Phi_1, \Psi_1)$.
Our approach will be to modify these matrices to avoid singularities.

Since $t \mapsto \det{(\tilde{\Psi}_t^T  \tilde{\Phi}_t)}$ is a continuous function, there exists $\varepsilon > 0$ such that for every $t \in [0, \varepsilon) \cup (1-\varepsilon, 1]$, the matrix $\tilde{\Psi}_t^T  \tilde{\Phi}_t$ is invertible.
Moreover, $\varepsilon > 0$ may be chosen small enough so that
\begin{equation}
    \hat{\Phi}_t = \tilde{\Phi}_t (\tilde{\Psi}_t^T \tilde{\Phi}_t)^{-1} \qquad \forall t\in [0, \varepsilon) \cup (1-\varepsilon, 1],
\end{equation}
is sufficiently close to $\tilde{\Phi}_t$ that any convex combination of $\hat{\Phi}_t$ and $\tilde{\Phi}_t$ has linearly independent columns.
We observe that on $[0, \varepsilon) \cup (1-\varepsilon, 1]$, $t\mapsto \hat{\Phi}_t$ is continuous and $\tilde{\Psi}_t^T \hat{\Phi}_t = I_r$.
Furthermore, $\hat{\Phi}_t$ agrees with the original $\Phi_0$ at $t=0$ and with $\Phi_1$ at $t=1$.

Let $\varphi:\mathbb{R} \to [0,1]$ be a smooth function such that $\varphi(t) = 1$ for every $t \in (-\infty, 8\varepsilon/10] \cup [1-8\varepsilon/10, \infty)$ and  $\varphi(t) = 0$ for every $t \in [9\varepsilon/10, 1-9\varepsilon/10]$.
We define the continuous set of matrices
\begin{equation}
    \Phi_t = \left\lbrace \begin{matrix}
    \tilde{\Phi}_t & & t \in [9\varepsilon/10, 1-9\varepsilon/10] \\
    \varphi(t) \hat{\Phi}_t + (1-\varphi(t))\tilde{\Phi}_t & & \mbox{otherwise}
    \end{matrix} \right.
\end{equation}
for $0\leq t\leq 1$ and we observe that $\Phi_t$ agrees with the original matrix $\Phi_0$ at $t=0$ and with $\Phi_1$ at $t=1$.
Now let $\psi:\mathbb{R} \to [0,1]$ be a smooth function such that $\psi(t) = 1$ for every $t \in (-\infty, 1\varepsilon/10] \cup [1-1\varepsilon/10, \infty)$ and  $\psi(t) = 0$ for every $t \in [2\varepsilon/10, 1-2\varepsilon/10]$.
We now define the continuous set of matrices
\begin{equation}
    \Psi_t = \left\lbrace \begin{matrix}
    \Phi_t & & t \in [2\varepsilon/10, 1-2\varepsilon/10] \\
    \psi(t) \tilde{\Psi}_t + (1-\psi(t))\Phi_t & & \mbox{otherwise}
    \end{matrix} \right.
\end{equation}
for $0\leq t\leq 1$ and we observe that $\Psi_t$ agree with the original matrix $\Psi_0$ at $t=0$ and with $\Psi_1$ at $t=1$.
Finally, we observe that
\begin{equation}
    \Phi_t^T \Psi_t = \left\lbrace \begin{matrix}
    \Phi_t^T \Phi_t & & t \in [2\varepsilon/10, 1-2\varepsilon/10] \\
    \psi(t) I_r + (1-\psi(t))\hat{\Phi}_t^T \hat{\Phi}_t & & \mbox{otherwise}
    \end{matrix} \right.
\end{equation}
for if $t \in [0, 2\varepsilon/10) \cup (1-2\varepsilon/10, 1]$ then $\Phi_t = \hat{\Phi}_t$ and $\hat{\Phi}_t^T \tilde{\Psi}_t = I_r$.
Therefore, $\Phi_t^T \Psi_t$ is a positive-definite matrix for every $t\in [0,1]$ and so $(V_t, W_t) = \pi(\Phi_t, \Psi_t) \in \mathcal{P}$ is a continuous path between $(V_0,W_0)\in\mathcal{P}$ and $(V_1,W_1)\in\mathcal{P}$.

Finally, we conclude with
\begin{lemma}
The submanifold $\mathcal{P}\subset \mathcal{G}_{n,r}\times \mathcal{G}_{n,r}$ is diffeomorphic to
\begin{equation}
    \mathbb{P} = \left\{ P \in \mathbb{R}^{n\times n} \ : \ P^2 = P \quad \mbox{and} \quad \rank(P) = r \right\}.
\end{equation}
\end{lemma}
\begin{proof}
The map $\phi : \mathcal{P} \to \mathbb{P}$ defined by
\begin{equation}
    \phi\circ\pi(\Phi, \Psi) = \Phi (\Psi^T \Phi)^{-1}\Psi^T 
\end{equation}
is smooth thanks to Theorem~4.29 in J. M. Lee \cite{Lee2013introduction}.
Moreover, $\phi$ is injective for if $(V_i, W_i) \in \mathcal{P}$, $i\in \{0, 1 \}$ are subspace pairs with representatives $(\Phi_i,\Psi_i) \in \pi^{-1}(V_i, W_i)$ satisfying
\begin{equation}
    \Phi_0 (\Psi_0^T \Phi_0)^{-1}\Psi_0^T  = \Phi_1 (\Psi_1^T \Phi_1)^{-1}\Psi_1^T ,
\end{equation}
then $V_0 = \Range(\Phi_0) = \Range(\Phi_1) = V_1$ and $W_0 = \Range(\Psi_0) = \Range(\Psi_1) = W_1$.
To show that $\phi$ is surjective, choose any $P\in \mathbb{P}$ and consider a compact singular value decomposition $P = U \Sigma Q^T $.
It is clear that $\Sigma$ is an invertible $r\times r$ diagonal matrix and the condition $P^2 = P$ implies that
\begin{equation}
    \Sigma Q^T  U \Sigma = \Sigma \quad \Rightarrow \quad Q^T U = \Sigma^{-1}.
\end{equation}
Therefore, $P = U (Q^T U)^{-1} Q^T $ for some $Q,U\in\mathbb{R}_*^{n,r}$ such that $\det(Q^T U) \neq 0$.
Taking $(V,W) = \pi(U,Q)$ we obtain $\phi(V,W) = P$.

It now remains to show that $\phi^{-1}: \mathbb{P} \to \mathcal{P}$ is differentiable.
Let $\{ e_i \}_{i=1}^n$ be an orthonormal basis for the state space $\mathbb{R}^n$.
If $I = \{ i_1, \ldots, i_r \} \subset \{1, \ldots, n \}$ is a subset of $r$ indices, let
\begin{equation}
    E_I = \begin{bmatrix} e_{i_1} & \cdots & e_{i_r} \end{bmatrix} \in \mathbb{R}^{n\times r}.
\end{equation}
Choose $P\in \mathbb{P}$ and let $I,J \subset \{1, \ldots, n \}$ be sets of indices with $\vert I \vert = \vert J \vert = r$ elements such that $\Range (P^T E_I) = \Range (P^T )$ and $\Range (P E_J) = \Range (P)$.
Since $\phi$ is bijective, we have 
\begin{equation}
    P = \phi(\Range (P),\ \Range (P^T )) =  (P E_J) \left[ (P^T  E_I)^T  (P E_J) \right]^{-1} (P^T  E_I)^T .
\end{equation}
Most importantly, the same sets of indices $I$ and $J$ satisfy the above properties for every $\tilde{P}$ in a sufficiently small neighborhood of $P$ in $\mathbb{P}$.
The maps
\begin{equation}
   P \mapsto P E_J, \qquad P \mapsto P^T E_I,
\end{equation}
are smooth and so 
\begin{equation}
    P \mapsto \pi(P E_J, P^T  E_I)
\end{equation}
is the smooth inverse of $\phi$ over a small neighborhood of $P$ in $\mathbb{P}$. 
Since such a smooth inverse exists near every $P\in\mathcal{P}$ it follows that $\phi$ is a diffeomorphism.
\end{proof}

\section{Proof of \cref{prop:domain_of_existence} (Properties of ROM Solutions)}
\label{app:domain_of_existence}
Any solution of \cref{eqn:reduced_order_model} is unique since \cref{eqn:reduced_order_model} is smooth.
This is a trivial consequence of Gr\"{o}nwall's inequality (Corollary~8.62 in \cite{Kelly2004theory}). 
Suppose that $\hat{x}_0$ and $\hat{x}_1$ are two solutions of \cref{eqn:reduced_order_model} over the interval $[t_0, t_{L-1}]$ at the same $(V,W)\in\mathcal{P}$. 
Since these solutions are are continuous in time, they are contained in some closed ball $\bar{B}\subset\mathbb{R}^n$.
Since $(x,t)\mapsto f(x,u(t))$ is continuously differentiable by \cref{asmpn:FOM_is_C2}, it is $L$-Lipschitz in $\bar{B}$ for some finite $L$ and we have
\begin{equation}
\begin{split}
    \Vert \hat{x}_0(t) - \hat{x}_1(t) \Vert &\leq \int_{t_0}^t \left\Vert P_{V,W}\left( f(\hat{x}_0(s), u(s)) - f(\hat{x}_1(s), u(s)) \right) \right\Vert \td s \\
    &\leq L \Vert P_{V,W}\Vert_{\text{op}} \int_{t_0}^t \Vert \hat{x}_0(s) - \hat{x}_1(s)\Vert \td s.
\end{split}
\end{equation}
By Gr\"{o}nwall's inequality, it follows that
\begin{equation}
    \Vert \hat{x}_0(t) - \hat{x}_1(t) \Vert \leq 0,
\end{equation}
which implies that $\hat{x}_0(t) = \hat{x}_1(t)$ for all $t\in [t_0, t_{L-1}]$.

Suppose that a solution $\hat{x}_0(t) = \hat{x}(t;(V_0, W_0))$ exists for a given $(V_0, W_0)\in\mathcal{P}$.
Since $\hat{x}_0$ is continuous over the finite interval $[t_0, t_{L-1}]$, it is bounded and contained in the open ball
\begin{equation}
    B = \left\{ x\in\mathbb{R}^n \ : \ \Vert x \Vert < \sup_{t\in [t_0, t_{L-1}]} \Vert \hat{x}_0(t)\Vert + 1 \right\}.
\end{equation}
Moreover, by \cref{asmpn:FOM_is_C2} it follows that every $(t,x) \mapsto P_{V,W} f(x, u(t))$ with $(V,W)\in\mathcal{P}$ is bounded and Lipschitz on $\bar{B}$.
Therefore, for any $(V,W)\in\mathcal{P}$ such that $P_{V,W}x_0 \in B$, the Picard-Lindelof theorem (Theorem~8.13 in W. G. Kelly A. C. Peterson \cite{Kelly2004theory}), ensures the reduced-order model \cref{eqn:reduced_order_model} has a unique solution $\hat{x}(t;(V,W))$ in $B$ over an interval $[t_0, \alpha]$ for some $\alpha > t_0$.
Moreover, the extension theorem for ODEs (Theorem~8.33 \cite{Kelly2004theory}) implies that the solution $\hat{x}(t;(V,W))$ of \cref{eqn:reduced_order_model} exists in $B$ for all time $t \geq t_0$, or there is a finite $\omega$ so that $\hat{x}(t;(V,W))$ remains in $B$ for $t\in [t_0, \omega)$ and $\hat{x}(t;(V,W)) \to \partial B$ as $t\to \omega^-$.
To be precise, the latter means that $\hat{x}(t;(V,W))$ leaves any compact subset of $B$ as $t\to \omega^-$.

For the sake of producing a contradiction, suppose that there is a sequence $\{ (V_k, W_k) \}_{k=1}^{\infty}$ such that $(V_k, W_k) \to (V_0, W_0)$ and for which the reduced-order model does not have a solution on $[t_0, t_{L-1}]$. 
Since the map $\phi:(V,W)\mapsto P_{V,W}$ is smooth by \cref{thm:topology_of_P}, we may assume that each $(V_k, W_k)$ is already in a sufficiently small neighborhood of $(V_0, W_0)$ such that $P_{V_k, W_k}x_0 \in B$.
Consequently, each reduced-order model solution $\hat{x}_k(t) = \hat{x}(t;(V_k, W_k))$ exist and remains in $B$ over some maximal interval $[t_0, \omega_k)$ with $t_0 < \omega_k < t_{L-1}$ and $\hat{x}_k(t) \to \partial B$ as $t\to \omega_k^{-}$.

To produce a contradiction, we show that $\hat{x}_k(t)$ remains close to $\hat{x}_0(t)$ over the interval $[t_0, \omega_k)$ for sufficiently large $k$, which will be at odds with $\hat{x}_k(t) \to \partial B$ as $t\to \omega_k^{-}$.
For $t\in [t_0, \omega_k)$ we have the following bound on the difference between the trajectories
\begin{multline}
    \left\Vert \hat{x}_k(t) - \hat{x}_0(t) \right\Vert 
    \leq \left\Vert (P_{V_k,W_k}-P_{V_0,W_0})x_0 \right\Vert \\
    + \int_{t_0}^{t} \left\Vert (P_{V_k,W_k} - P_{V_0,W_0})f(\hat{x}_k(s), u(s)) \right\Vert \td s \\
    + \int_{t_0}^{t} \left\Vert P_{V_0,W_0}f(\hat{x}_k(s), u(s)) - P_{V_0,W_0}f(\hat{x}_0(s), u(s)) \right\Vert \td s.
\end{multline}
Let $\Vert \cdot \Vert_{\text{op}}$ denote the induced norm (operator norm) and observe that since $(t,x) \mapsto f(x,u(t))$ is continuously differentiable with respect to $x$ by \cref{asmpn:FOM_is_C2}, there are finite constants $M$ and $L$ such that
\begin{alignat}{3}
    \Vert f(x, u(t)) \Vert &\leq M \qquad &&\forall x\in \bar{B}, \ \forall t\in [t_0, t_{L-1}] \\
    \Vert f(x, u(t)) - f(z, u(t)) \Vert &\leq L \Vert x - z \Vert \qquad &&\forall x,z\in\bar{B}, \ \forall t\in [t_0, t_{L-1}].
\end{alignat}
Therefore, we have
\begin{multline}
    \left\Vert \hat{x}_k(t) - \hat{x}_0(t) \right\Vert 
    \leq \Vert P_{V_k,W_k} - P_{V_0, W_0} \Vert_{\text{op}} \left( \Vert x_0\Vert + M(t_{L-1} - t_0) \right) \\
    + L \Vert P_{V_0,W_0} \Vert_{\text{op}} \int_{t_0}^t \left\Vert \hat{x}_k(s) - \hat{x}_0(s) \right\Vert \td s.
\end{multline}
Applying Gr\"{o}nwall's inequality (Corollary~8.62 in \cite{Kelly2004theory}), we see that
\begin{equation}
    \left\Vert \hat{x}_k(t) - \hat{x}_0(t) \right\Vert \leq 
    \Vert P_{V_k,W_k} - P_{V_0, W_0} \Vert_{\text{op}} \left( \Vert x_0\Vert + M(t_{L-1} - t_0) \right) e^{L \Vert P_{V_0,W_0} \Vert_{\text{op}} t}.
    \label{eqn:Gronwall_neighboring_trajs}
\end{equation}
Since $\phi:(V,W)\mapsto P_{V,W}$ is continuous,
\begin{equation}
    (V_k,W_k)\to(V_0,W_0) \quad \Rightarrow \quad \Vert P_{V_k,W_k} - P_{V_0,W_0} \Vert_{\text{op}} \to 0,
\end{equation}
and so \cref{eqn:Gronwall_neighboring_trajs} implies that $\left\Vert \hat{x}_k(t) - \hat{x}_0(t) \right\Vert \to 0$ uniformly over $t\in [t_0, \omega_k)$ as $k\to\infty$.
In particular, we may take $K$ sufficiently large so that for any $k \geq K$ then 
\begin{equation}
    \left\Vert \hat{x}_k(t) - \hat{x}_0(t) \right\Vert \leq \frac{1}{2} \qquad \forall t\in [t_0, \omega_k),
\end{equation}
contradicting the fact that $\hat{x}_k(t) \to \partial B$ as $t\to \omega_k^-$.
Therefore, there is an open neighborhood of $(V_0,W_0)$ in $\mathcal{P}$ in which the reduced order model \cref{eqn:reduced_order_model} has a unique solution over the time interval $[t_0, t_{L-1}]$, which establishes the openness of $\mathcal{D}$ in $\mathcal{P}$.

Since $\mathcal{D}$ is open in $\mathcal{P}$ it follows that there is a set $\mathcal{D}'\in\mathcal{G}_{n,r}\times\mathcal{G}_{n,r}$ such that $\mathcal{D} = \mathcal{D}'\cap\mathcal{P}$.
Since $\mathcal{P}$ is open in $\mathcal{G}_{n,r}\times\mathcal{G}_{n,r}$ by \cref{thm:topology_of_P}, it follows that $\mathcal{D}$ is open in $\mathcal{G}_{n,r}\times\mathcal{G}_{n,r}$ since it is a finite intersection of open sets.

Now, let us turn our attention to proving that $\mathcal{D} = \mathcal{P}$ when $f$ has bounded $x$-derivatives.
Since the partial derivatives of $f$ with respect to $x$ are bounded, it follows that for any $(V,W)\in\mathcal{P}$ there is a constant $L$ such that
\begin{equation}
    \left\Vert P_{V,W}f(x_1,u(t)) - P_{V,W}f(x_2,u(t))\right\Vert \leq L\Vert x_1 - x_2 \Vert \qquad \forall x_1,x_2\in\mathbb{R}^n, \quad \forall t\in\mathbb{R}
\end{equation}
and so $(x,t)\mapsto P_{V,W}f(x,t)$ is Lipschitz in $x$ uniformly over $t$.
A trivial modification of Theorem~7.3 in H. Brezis \cite{Brezis2010functional} shows that a solution of the reduced-order model \cref{eqn:reduced_order_model} exists on the interval $[t_0,\infty)$.
As a consequence, $\mathcal{D} = \mathcal{P}$.

Now we shall establish the $k$-times continuous differentiability of $(t,(V,W))\mapsto \hat{x}(t;(V,W))$ with respect to $(V,W)$ over $[t_0, t_{L-1}]\times \mathcal{D}$.
In particular, this means that $(t,(V,W))\mapsto \hat{x}(t;(V,W))$ is $k$-times continuously differentiable with respect to a smooth coordinate system defined in a neighborhood of any $(t,(V_0,W_0))\in [t_0, t_{L-1}]\times \mathcal{D}$.
Recall that by \cref{thm:topology_of_P}, the set of rank-$r$ projection matrices $\mathbb{P}$ is smoothly diffeomorphic to the $2nr - 2r^2$ dimensional submanifold $\mathcal{P}\subset \mathcal{G}_{n,r}\times \mathcal{G}_{n,r}$.
Let $\psi : \mathbb{R}^{2nr-2r^2} \to \mathcal{U}\subset \mathcal{D}$ be a local parameterization of a an open neighborhood $\mathcal{U}\subset\mathcal{D}$ containing the point $(V_0,W_0)$.
Letting $\phi : (V,W) \mapsto P_{V,W}$ be the diffeomorphism established by \cref{thm:topology_of_P}, the map $P = \phi \circ \psi$ is a smooth parameterization of the open subset $\phi(\mathcal{U})\subset\mathbb{P}$ containing the projection operator $P_{V_0,W_0}$.
It suffices to show that the solution $\hat{x}(t;\psi(p))$ is continuously differentiable with respect to $p\in \mathbb{R}^{2nr-2r^2}$.

In order to prove the smoothness result, we need the following generalization of Theorem~8.43 in \cite{Kelly2004theory}, which we prove via the same induction process used to prove an analogous, but slightly weaker result given by Theorem~D5 in \cite{Lee2013introduction}.
\begin{lemma}[Smoothness of non-autonomous ODEs]
\label{lem:smoothness_of_ODEs}
Let $D$ be an open subset of $\mathbb{R}\times \mathbb{R}^n$ and $f: D \to \mathbb{R}^n$ a vector-valued function such that $(x,t)\mapsto f(x,t)$ has continuous partial derivatives with respect to $x$ up to order $k\geq 1$ on $D$.
Then, the initial value problem
\begin{equation}
    \ddt x = f(x, t), \qquad x(t_0) = x_0
\end{equation}
has a unique solution, denoted $x(t;x_0)$, with maximal interval of existence $t\in(\alpha, \omega)\subset \mathbb{R}$.
Over this interval, the function $(t,x_0) \mapsto x(t;x_0)$ has continuous partial derivatives with respect to $x_0$ up to order $k$.
\end{lemma}
\begin{proof}
We proceed by induction on $k$, with the base case $k=1$ being provided by Theorem~8.43 in \cite{Kelly2004theory}.
We assume the result is true for some $k\geq 1$ and we suppose that $(x,t)\mapsto f(x,t)$ has continuous partial derivatives with respect to $x$ up to order $k+1$.
By the induction hypothesis, $(t,x_0) \mapsto x(t;x_0)$ has partial derivatives with respect to $x_0$ up to order $k$.
Furthermore, $(t,x_0)\mapsto \frac{\partial x}{\partial t}(t;x_0) = f(x(t;x_0), t)$ has continuous partial derivatives with respect to $x_0$ up to order $k$ by the chain rule.
Thanks to these continuity properties, we may differentiate 
\begin{equation}
    x(t; x_0) = x_0 + \int_{t_0}^t f(x(\tau; x_0), \tau) \td \tau
\end{equation}
with respect to $x_0$ under the integral, yielding
\begin{equation}
    \frac{\partial x}{\partial x_0}(t; x_0) = I + \int_{t_0}^t \frac{\partial f}{\partial x}( x(\tau; x_0), \tau) \frac{\partial x}{\partial x_0}(\tau; x_0) \td \tau.
\end{equation}
By the fundamental theorem of calculus, $\frac{\partial x}{\partial x_0}(t; x_0)$ satisfies
\begin{equation}
    \frac{\partial}{\partial t}\frac{\partial x}{\partial x_0}(t; x_0) = \frac{\partial f}{\partial x}(x(t; x_0),t) \frac{\partial x}{\partial x_0}(t; x_0), \qquad
    \frac{\partial x}{\partial x_0}(t_0; x_0) = I.
\end{equation}
Consider the initial value problem
\begin{equation}
    \ddt \begin{bmatrix}
    z^a \\
    z^b
    \end{bmatrix} = 
    \begin{bmatrix}
    f(z^a, t) \\
    \frac{\partial f}{\partial x}(z^a, t) z^b
    \end{bmatrix} =: F((z^a, z^b), t), \qquad
    \begin{bmatrix}
    z^a(0) \\
    z^b(0)
    \end{bmatrix} = 
    \begin{bmatrix}
    z^a_0 \\
    z^b_0
    \end{bmatrix},
    \label{eqn:smoothness_lemma_augmented_IVP}
\end{equation}
and observe that it has a unique solution $z^a(t;(x_0, I)) = x(t; x_0)$,  $z^b(t; (x_0, I)) = \frac{\partial x}{\partial x_0}(t; x_0)$ over the interval $t\in (\alpha, \omega)$ when $(z^a_0, z^b_0) = (x_0, I) \in \mathbb{R}^n\times \mathbb{R}^{n\times n}$.
Here, uniqueness follows from the smoothness of $F$.
In particular, $(z, t) \mapsto F(z,t)$ has continuous partial derivatives with respect to $z=(z^a, z^b)$ up to order $k$ since $(x,t)\mapsto f(x,t)$ has continuous partial derivatives with respect to $x$ up to order $k+1$.
By the induction hypothesis, the solution $(t,z_0) \mapsto z(t;z_0)=(z^a(t;z_0), z^b(t;z_0))$ of \cref{eqn:smoothness_lemma_augmented_IVP} has continuous partial derivatives with respect to $z_0=(z^a_0, z^b_0)$ up to order $k$ about the points $(t,z_0)$ with $t\in (\alpha, \omega)$ and $(z^a_0, z^b_0) = (x_0, I)$.
Finally, since
\begin{equation}
    \frac{\partial x}{\partial x_0}(t; x_0) = z^b(t; (x_0, I)),
\end{equation}
it follows that $(t; x_0) \mapsto x(t; x_0)$ has continuous partial derivatives with respect to $x_0$ up to order $k+1$, completing the proof.
\end{proof}

We define the augmented state variable $w = (w^a, w^b) \in \mathbb{R}^{n}\times \mathbb{R}^{2nr-2r^2}$ whose dynamics are described by
\begin{equation}
    \ddt w = F(w, t) :=
    \begin{bmatrix}
    P(w^b) f(w^a, u(t)) \\
    0_{2nr-r^2}
    \end{bmatrix} \qquad 
    w(0) = w_0.
    \label{eqn:smoothness_prop_augmented_sys}
\end{equation}
Clearly, we have $\hat{x}(t;\psi(p)) = w^a(t;w_0)$ when $w_0 = (P(p) x_0, p)$.
It is also clear that $(w,t)\mapsto F(w,t)$ is $k$-times continuously differentiable with respect to $w$ when $(x,t)\mapsto f(x, u(t))$ is $k$-times continuously differentiable with respect to $x$, thanks to the smoothness of the map $P = \phi \circ \psi$.
Applying \cref{lem:smoothness_of_ODEs} to the initial value problem \cref{eqn:smoothness_prop_augmented_sys} shows that $(t,w_0)\mapsto w(t;w_0)$ has continuous partial derivatives up to order $k$ with respect to $w_0$ at each point $(t,w_0)$ with $t\in [t_0,t_{L-1}]$ and $w_0 = (P(p) x_0, p)$, $p\in \mathbb{R}^{2nr-2r^2}$.
Therefore, $(t,p)\mapsto \hat{x}(t;\psi(p)) = w^a(t;(P(p) x_0, p))$ has continuous partial derivatives up to order $k$ with respect to $p$ for $(t,p)\in [t_0, t_{L-1}]\times \mathbb{R}^{2nr-2r^2}$, from which it follows that $(t,(V,W)) \mapsto \hat{x}(t;(V,W))$ is continuously differentiable up to order $k$ with respect to $(V,W)$ for every $(t,(V,W)) \in [t_0, t_{L-1}]\times\mathcal{D}$.

Finally, suppose that $\{(V_k, W_k)\}_{k=1}^{\infty}\subset \mathcal{D}$ is a sequence approaching $(V_k, W_k)\to (V_0,W_0)\in\mathcal{P}\setminus\mathcal{D}$.
Denote $\hat{x}_k(t) = \hat{x}(t;(V_k, W_k))$ and $\hat{x}_0(t) = \hat{x}(t;(V_0, W_0))$, and let $[t_0, \omega_R)$ be the maximum interval of existence for $\hat{x}_0$ in an open ball $B_R\subset \mathbb{R}^n$ of radius $R > \Vert P_{V_0,W_0} x_0\Vert$ centered about the origin.
Here we have again made use of the extension theorem for solutions of ordinary differential equations (Theorem~8.33 \cite{Kelly2004theory}).
It is clear that since $(x,t)\mapsto f(x,u(t))$ is continuously differentiable with respect to $x$ by \cref{asmpn:FOM_is_C2}, it is $L$-Lipschitz and bounded by $M$ on $\overline{B_R}$ for some finite $L$ and $M$.
Let us suppose that $k$ is sufficiently large so that $P_{V_k, W_k}x_0 \in B_R$ and let $[t_0,\omega_k)$ denote the maximum interval of existence for $\hat{x}_k(t)$ in $B_R$.
There are two possibilities, either $\omega_k < \omega_R$ which implies that $\sup_{t\in [t_0, t_{L-1}]} \Vert \hat{x}_k(t)\Vert \geq R$, or $\omega_k \geq \omega_R$ which implies that $\hat{x}_k(t)\in B_R$ for every $t\in [t_0, \omega_R)$.
In the second case, the same Gr\"{o}nwall argument we used above shows that that 
\begin{equation}
    \left\Vert \hat{x}_k(t) - \hat{x}_0(t) \right\Vert \leq 
    \Vert P_{V_k,W_k} - P_{V_0, W_0} \Vert_{\text{op}} \left( \Vert x_0\Vert + M(t_{L-1} - t_0) \right) e^{L \Vert P_{V_0,W_0} \Vert_{\text{op}} t}
\end{equation}
for every $t\in [t_0, \omega_R)$.
Since, $P_{V_k,W_k} \to P_{V_0, W_0}$, we may take $k$ sufficiently large so that the above inequality implies that $\left\Vert \hat{x}_k(t) - \hat{x}_0(t) \right\Vert \leq R/2$ for every $t\in [t_0, \omega_R)$ when $\hat{x}_k(t)\in B_R$ for every $t\in [t_0, \omega_R)$.
Since $\hat{x}_0(t) \to \partial B_R$ as $t\to \omega_R^-$, we must have $\sup_{t\in [t_0, t_{L-1}]} \Vert \hat{x}_k(t)\Vert \geq R/2$ in the case when $\hat{x}_k(t)\in B_R$ for every $t\in [t_0, \omega_R)$.
It follows that for sufficiently large $k$, we always have $\sup_{t\in [t_0, t_{L-1}]} \Vert \hat{x}_k(t)\Vert \geq R/2$.
Since $R$ was arbitrary, it follows that
\begin{equation}
    \sup_{t\in [t_0, t_{L-1}]} \Vert \hat{x}_k(t)\Vert \to \infty \quad \mbox{as} \quad k \to \infty.
\end{equation}
Furthermore, the $\sup$ is actually a $\max$ because the trajectories $\hat{x}_k$ are continuous.
This completes the proof of \cref{prop:domain_of_existence}.

\section{Regularization and Existence of a Minimizer}
\label{app:regularization_and_minimizer}

\begin{proof}[Proof of \cref{thm:regularization_properties} (Regularization)]
We begin by showing that $\rho(V,W)\to + \infty$ as $(V,W) \to (V_0, W_0) \in \mathcal{G}_{n,r}\times \mathcal{G}_{n,r}$.
Let $\Phi_0, \Psi_0 \in \pi^{-1}(V_0, W_0)$.
By the local submersion theorem \cite{Guillemin2010differential}, there is an open neighborhood $\mathcal{V} \subset \mathcal{G}_{n,r}\times \mathcal{G}_{n,r}$ containing $(V_0, W_0)$ and an open neighborhood $\mathcal{U} \subset \mathbb{R}_*^{n\times r}\times\mathbb{R}_*^{n\times r}$ containing $(\Phi_0, \Psi_0)$ together with local parameterizations $\phi : \mathbb{R}^{2 n r} \to \mathcal{U}$ and $\psi:\mathbb{R}^{2 n r - 2r^2} \to \mathcal{V}$ of these neighborhoods such that $(\Phi_0, \Psi_0)=\phi(0)$, $(V_0, W_0) = \psi(0)$, and
\begin{equation}
    (\psi^{-1}\circ\pi\circ\phi) (x_1, \ldots, x_{2nr}) = (x_1, \ldots, x_{2nr-2r^2}).
\end{equation}
Since $(V_n, W_n)\to (V_0, W_0)$ there exist $N$ such that for every $n \geq N$, we have $(V_n, W_n)\in\mathcal{V}$.
Let $z^{(n)} = \psi^{-1}(V_n, W_n)$ be the coordinates of these subspace pairs for $n \geq N$ and let us choose the representatives of these subspaces whose coordinates are $x^{(n)} = (z^{(n)}, 0, \ldots, 0) \in \mathbb{R}^{2nr}$, i.e., let $(\Phi_n, \Psi_n) = \phi(z^{(n)},0,\ldots,0)$.
It is clear that $z^{(n)} \to 0$ as $n\to\infty$ and so we have $(\Phi_n, \Psi_n) \to (\Phi_0, \Psi_0)$ as $n\to\infty$ by continuity of the local parameterizaions.
Since the determinant is a continuous function, we have
\begin{equation}
\begin{split}
    \lim_{n\to\infty} \det{(\Psi_n^T \Phi_n)} &= \det{(\Psi_0^T \Phi_0)} = 0, \\
    \lim_{n\to\infty} \det{(\Phi_n^T \Phi_n)} &= \det{(\Phi_0^T \Phi_0)} > 0, \\ 
    \lim_{n\to\infty} \det{(\Psi_n^T \Psi_n)} &= \det{(\Psi_0^T \Psi_0)} > 0
\end{split}
\end{equation}
and so it follows that
\begin{equation}
    \rho(V_n, W_n) = \rho\circ\pi(\Phi_n, \Psi_n) \to \infty \quad \mbox{as}\quad n\to\infty.
\end{equation}

Now we seek a minimum of $\rho$ by first considering the function $F:\mathcal{G}_{n,r}\times \mathcal{G}_{n,r}\to\mathbb{R}$ defined by
\begin{equation}
    F\circ\pi(\Phi,\Psi) = \frac{\det ( \Psi^T \Phi )^2}{\det(\Phi^T \Phi) \det(\Psi^T \Psi)}
\end{equation}
and observing that it is continuous on $\mathcal{G}_{n,r}\times \mathcal{G}_{n,r}$.
Since $\mathcal{G}_{n,r}$ is compact, it follows that $\mathcal{G}_{n,r}\times \mathcal{G}_{n,r}$ is also compact, and so $F$ attains its maximum.
Moreover, if $V=W$ then obviously $(V,W)\in\mathcal{P}$ and choosing the columns of $\Phi = \Psi$ to be an orthonormal basis for $V$, we find that 
\begin{equation}
    F(V,V) = \frac{\det ( \Phi^T \Phi )^2}{\det(\Phi^T \Phi) \det(\Phi^T \Phi)} = 1
    \quad \Rightarrow \quad \rho(V,V) = -\log F(V,V) = 0.
\end{equation}
Consequently, the maximum value of $F$ is at least $1$ and so any subspace pair $(V_{\text{m}}, W_{\text{m}})$ that maximizes $F$ must lie in $\mathcal{P} = F^{-1}((0,\infty))$ and also minimize $R = - \log F$.
Since $\rho$ is a smooth function on the open set $\mathcal{P}$ (see \cref{thm:topology_of_P}), a necessary condition for $(V_{\text{m}}, W_{\text{m}})$ to be a minimizer of $\rho$ is $\D \rho(V_{\text{m}}, W_{\text{m}})(\xi,\eta) = 0$ for every $(\xi,\eta)\in T_{(V_{\text{m}}, W_{\text{m}})}\mathcal{G}_{n,r}\times \mathcal{G}_{n,r}$.
Let $(\Phi_{\text{m}}, \Psi_{\text{m}})\in \pi^{-1}(V_{\text{m}},W_{\text{m}})$ be representatives of minimizing subspaces such that $\det(\Psi_{\text{m}}^T \Phi_{\text{m}}) > 0$, e.g., by flipping the sign on a column of $\Phi_{\text{m}}$.
Then for every pair of matrices $(X,Y) \in \mathbb{R}^{n\times r}\times\mathbb{R}^{n\times r}$ we have
\begin{multline}
    0 = \D(\rho\circ\pi)(\Phi_{\text{m}},\Psi_{\text{m}})(X,Y)
    = \Tr{\left\{(\Phi_{\text{m}}^T \Phi_{\text{m}})^{-1}(\Phi_{\text{m}}^T X + X^T \Phi_{\text{m}})\right\}} \\
    + \Tr{\left\{(\Psi_{\text{m}}^T \Psi_{\text{m}})^{-1}(\Psi_{\text{m}}^T Y + Y^T \Psi_{\text{m}})\right\}}
    - 2\Tr{\left\{(\Psi_{\text{m}}^T \Phi_{\text{m}})^{-1}(\Psi_{\text{m}}^T X + Y^T \Phi_{\text{m}})\right\}}.
\end{multline}
Here, we have differentiated the log-determinants of matrices with positive determinants using the formula $\D ( M \mapsto \log \det M ) V = \Tr(M^{-1}V)$ given by Theorem~2 in Section~8.4 of \cite{Magnus2007matrix}.
Applying permutation identities for the trace and collecting terms we have
\begin{multline}
    0 = \Tr{\left\{ \left[(\Phi_{\text{m}}^T \Phi_{\text{m}})^{-1}\Phi_{\text{m}}^T  - (\Psi_{\text{m}}^T \Phi_{\text{m}})^{-1}\Psi_{\text{m}}^T  \right] X \right\}} \\
    + \Tr{\left\{ Y^T  \left[ \Psi_{\text{m}} (\Psi_{\text{m}}^T \Psi_{\text{m}})^{-1} - \Phi_{\text{m}} (\Psi_{\text{m}}^T \Phi_{\text{m}})^{-1} \right] \right\}}
\end{multline}
for every $(X,Y) \in \mathbb{R}^{n\times r}\times\mathbb{R}^{n\times r}$, which implies that
\begin{equation}
    (\Phi_{\text{m}}^T \Phi_{\text{m}})^{-1}\Phi_{\text{m}}^T  = (\Psi_{\text{m}}^T \Phi_{\text{m}})^{-1}\Psi_{\text{m}}^T  
    \quad \mbox{and}\quad 
    \Psi_{\text{m}} (\Psi_{\text{m}}^T \Psi_{\text{m}})^{-1} = \Phi_{\text{m}} (\Psi_{\text{m}}^T \Phi_{\text{m}})^{-1}.
\end{equation}
The above is true only if $\Range\Phi_{\text{m}} = \Range\Psi_{\text{m}}$; and so a necessary condition for $(V_{\text{m}}, W_{\text{m}})$ to minimize $\rho$ over $\mathcal{P}$ is that $V_{\text{m}} = W_{\text{m}}$.
But we have already seen that $\rho(V,W) = 0$ when $V=W$, proving that zero is the minimum value of $\rho$, and the minimum is attained if and only if the subspaces $(V,W)$ satisfy $V=W$.
\end{proof}

\begin{corollary}[Existence of a Minimizer]
\label{cor:existence_of_minimizer}
Let $\mathcal{D}$ be as in \cref{prop:domain_of_existence}, and take $\gamma > 0$.
We assume that $\mathcal{D}$ is nonempty.
Then a minimizer of \cref{eqn:ROM_cost} exists in $\mathcal{D}$; that is, there exists a pair of subspaces $(V_{\text{op}}, W_{\text{op}})\in\mathcal{D}$ such that
\begin{equation}
    J(V_{\text{op}}, W_{\text{op}}) \le J(V,W),\qquad \text{for all $(V,W)\in\mathcal{D}$.}
\end{equation}
Let the set of subspaces defining orthogonal projection operators be denoted
\begin{equation}
    \mathcal{P}_0 = \left\{ (V,W)\in\mathcal{P} \ : \ V=W \right\}
\end{equation}
and assume that $\mathcal{D}\cap \mathcal{P}_0$ is nonempty.
Then, as $\gamma \to \infty$, any choice of minimizers, denoted $(V_{\text{op}}(\gamma), W_{\text{op}}(\gamma))$, approaches $\mathcal{D}\cap\mathcal{P}_0$.
Furthermore, the corresponding cost, temporarily denoted $J(V,W ;\ \gamma)$ to emphasize the dependence on $\gamma$, approaches the minimum over $\mathcal{D}\cap \mathcal{P}_0$, i.e.,
 \begin{equation}
     \lim_{\gamma \to\infty} J(V_{\text{op}}(\gamma), W_{\text{op}}(\gamma) ;\
     \gamma) \le J(V,V), \qquad\text{for all $V\in\mathcal{G}_{n,r}$.}
 \end{equation}
 Note that $J(V,V)$ does not depend on $\gamma$, since $\rho(V,V)=0$.
\end{corollary}
\begin{proof}
Choose a sequence \\
$\{(V_n, W_n)\}_{n=1}^{\infty}$ in $\mathcal{D}$ such that
\begin{equation}
    \lim_{n\to\infty} J(V_n, W_n) = \inf_{(V,W)\in\mathcal{D}} J(V,W) < \infty.
\label{eqn:infemizing_sequence}
\end{equation}
Since the Grassmann manifold $\mathcal{G}_{n,r}$ is compact, it follows that $\mathcal{G}_{n,r}\times\mathcal{G}_{n,r}$ is compact, and so there exists a convergent subsequence $(V_{n_k}, W_{n_k}) \to (V_0, W_0)$ for some $(V_0, W_0) \in \mathcal{G}_{n,r}\times\mathcal{G}_{n,r}$.
We must have $(V_0, W_0)\in\mathcal{P}$; for if not, \cref{thm:regularization_properties} tells us that $\rho(V_{n_k}, W_{n_k}) \to +\infty$ and so $J(V_{n_k}, W_{n_k}) \to +\infty$ as $k\to\infty$ because $L_y \geq 0$, contradicting \cref{eqn:infemizing_sequence}.
Furthermore, $(V_0, W_0) \in \mathcal{D}$, for if not then \cref{prop:domain_of_existence} and \cref{asmpn:blow_up} imply that $J(V_{n_k}, W_{n_k}) \to +\infty$ as $k\to\infty$, contradicting \cref{eqn:infemizing_sequence}.
The cost function $J$ defined by \cref{eqn:ROM_cost} is continuously differentiable on $\mathcal{D}$ because the ROM solution at each sample time $(V,W)\mapsto \hat{x}(t_i,(V,W))$ is continuously differentiable by \cref{prop:domain_of_existence} and the regularization function defined by \cref{eqn:regularization} is smooth.
Since $J$ is continuous on $\mathcal{D}$, we have
\begin{equation}
    \inf_{(V,W)\in\mathcal{D}} J(V,W) = \lim_{k\to\infty} J(V_{n_k}, W_{n_k}) = J(V_0, W_0),
\end{equation}
proving that $(V_0, W_0)$ achieves the minimum value of $J$ over $\mathcal{D}$.

To prove the second claim about the behavior as $\gamma \to \infty$, we begin by observing that minimizing $J$ over $\mathcal{P}_0$ is equivalent to minimizing $V \mapsto J(V,V)$ over $\mathcal{G}_{n,r}$.
This minimization does not depend on $\gamma$ because $\rho(V,V) = 0$ by \cref{thm:regularization_properties}.
Let us begin by showing that a minimizer of $V \mapsto J(V,V)$ over $\mathcal{G}_{n,r}$ exists.
Let $\{ (V_k, V_k) \}_{k=1}^{\infty}\subset \mathcal{D}\cap\mathcal{P}_0$ be a sequence such that
\begin{equation}
    J(V_k, V_k) \to \inf_{V\in \mathcal{G}_{n,r}} J(V,V) < \infty \quad \mbox{as}\quad k \to \infty.
\end{equation}
Since $\mathcal{G}_{n,r}$ is compact, we may pass to a convergent subsequence, still denoted by $\{ (V_k, V_k) \}_{k=1}^{\infty}$, such that $V_k \to V_0\in \mathcal{G}_{n,r}$.
Clearly, we have $(V_0, V_0) \in \mathcal{P}_0$.
If $(V_0,V_0) \notin\mathcal{D}$ then \cref{prop:domain_of_existence} and \cref{asmpn:blow_up} imply that $J(V_k, V_k) \to \infty$.
But this contradicts the fact that the sequence $\{ J(V_k,V_k) \}_{k=1}^{\infty}$ approaches the infemum of $J$ over $\mathcal{P}_0$, which is finite if $\mathcal{D}\cap\mathcal{P}_0\neq \emptyset$.
Therefore, $(V_0, V_0) \in \mathcal{D}\cap\mathcal{P}_0$ and since $J$ is continuous over $\mathcal{D}$ it follows that
\begin{equation}
    \inf_{V\in \mathcal{G}_{n,r}} J(V,V) = \lim_{k\to\infty} J(V_k, V_k) = J(V_0, V_0),
\end{equation}
i.e., $(V_0,V_0)$ achieves the minimum value of $J$ over $\mathcal{P}_0$.

Suppose, for the sake of producing a contradiction, that there is an open neighborhood $\mathcal{U} \subset \mathcal{D}$ containing $\mathcal{D}\cap\mathcal{P}_0$ such that for every $\Gamma > 0$, there exists $\gamma \geq \Gamma$ such that $(V_{\text{op}}(\gamma), W_{\text{op}}(\gamma)) \notin \mathcal{U}$.
Then $\mathcal{G}_{n,r}\times\mathcal{G}_{n,r}\setminus\mathcal{U}$ is a closed subset of $\mathcal{G}_{n,r}\times\mathcal{G}_{n,r}$ and hence is compact.
By the same argument presented above, $\rho$ attains its minimum over $\mathcal{P} \setminus \mathcal{U}$, and this value is strictly greater than zero by \cref{thm:regularization_properties}.
Consequently, we would have
\begin{equation}
    J(V_0, V_0) \geq J(V_{\text{op}}(\gamma), W_{\text{op}}(\gamma);\ \gamma) 
    \geq \gamma \min_{(V,W)\in\mathcal{P}\setminus\mathcal{U}}\rho(V,W) \to \infty 
    \quad \mbox{as}\quad \gamma \to\infty,
\end{equation}
contradicting the fact that $J(V_0, V_0) < \infty$.
Therefore, for every open neighborhood $\mathcal{U}\subset\mathcal{D}$ of $\mathcal{D}\cap\mathcal{P}_0$, there exists $\Gamma > 0$ such that for every $\gamma \geq \Gamma$, we have $(V_{\text{op}}(\gamma), W_{\text{op}}(\gamma))\in\mathcal{U}$.

Finally, suppose for the sake of producing a contradiction that there exists $\varepsilon > 0$ such that for every $\Gamma > 0$ there exists $\gamma \geq \Gamma$ such that $J(V_{\text{op}}(\gamma), W_{\text{op}}(\gamma);\ \gamma) \leq J(V_0, V_0) - \varepsilon$.
By continuity of the objective on $\mathcal{D}$, we know that the non-empty set
\begin{equation}
    \mathcal{U} = \left\{ (V,W)\in\mathcal{D} \ : \ J(V,W;\ 0) > J(V_0, V_0) - \varepsilon \right\}
\end{equation}
is open in $\mathcal{D}$ and contains $\mathcal{D}\cap\mathcal{P}_0$.
And so for every $\Gamma > 0$ we have a $\gamma \geq \Gamma$ such that
\begin{equation}
    J(V_{\text{op}}(\gamma), W_{\text{op}}(\gamma);\ 0) \leq J(V_{\text{op}}(\gamma), W_{\text{op}}(\gamma);\ \gamma) \leq J(V_0, V_0) - \varepsilon,
\end{equation}
which implies that $(V_{\text{op}}(\gamma), W_{\text{op}}(\gamma)) \notin \mathcal{U}$, contradicting the fact that 
\begin{equation}
    (V_{\text{op}}(\gamma), W_{\text{op}}(\gamma)) \to \mathcal{D}\cap\mathcal{P}_0 \quad \mbox{as} \quad \gamma \to \infty.
\end{equation}
Therefore, we conclude that $J(V_{\text{op}}(\gamma), W_{\text{op}}(\gamma);\ \gamma) \to J(V_0, V_0)$ as $\gamma \to \infty$.
\end{proof}

\section{Adjoint-Based Gradient and Required Terms}
\label{app:adjoint_gradient_and_required_terms}
\begin{proof}[Proof of \cref{thm:adjointOptimization} (Adjoint-Based Gradient)]
Thanks to \cref{prop:domain_of_existence}, the state of the reduced-order model representative \cref{eqn:reduced_order_model_representative}
at any $t\in [t_0, t_{L-1}]$ given by
\begin{equation}
    z_t: (\Phi, \Psi) \mapsto
    z(t; (\Phi,\Psi)) = (\Psi^T \Phi)^{-1}\Psi^T \hat{x}(t; \pi(\Phi, \Psi))
\end{equation}
is continuously differentiable at each $\theta = (\Phi,\Psi)$ in the structure space for which $\pi(\theta)\in \mathcal{D}$.
We shall compute the gradients of the component functions
\begin{equation}
    J_i(\theta) := L_y(\hat{y}_i(\theta) -  y_i)
\end{equation}
and use linear superposition to construct the gradient of \cref{eqn:GlobalCostFunctional}.
Letting $\xi\in T_{\theta}\bar{\mathcal{M}}$ be a tangent vector at $\theta \in \pi^{-1}(\mathcal{D})$, we observe that $v(t) = \D z_t(\theta) \xi$ satisfies
\begin{equation}
    \ddt v(t) - F(t) v(t) = S(t) \xi, \qquad
    v(t_0) = \left(\frac{\partial}{\partial \theta} z(t_0;\theta)\right) \xi.
    \label{eqn:LinearizedDynamics}
\end{equation}
Formally, this is because
\begin{equation}
    z_t(\theta) = z_0(\theta) + \int_{t_0}^t \tilde{f}(z_{\tau}(\theta), u(\tau); \theta) \td \tau,
\end{equation}
and we may differentiate under the integral to give
\begin{equation}
    \D z_t(\theta) \xi = \D z_0(\theta) \xi + \int_{t_0}^t\left[ F(\tau) \D z_{\tau}(\theta) \xi + S(\tau)\xi \right] \td \tau,
\end{equation}
since the integrand $(\tau, \theta) \mapsto \tilde{f}(z_{\tau}(\theta), u(\tau); \theta)$ is continuously differentiable.
Differentiating $\theta \mapsto \hat{y}_i(\theta) = \tilde{g}(z(t_i;\theta);\theta)$, we also obtain
\begin{equation}
    \D \hat{y}_i(\theta) \xi = H(t_i) v(t_i) + T(t_i) \xi.
\end{equation}
The resulting derivative of each component of the objective is given by
\begin{multline}
    \D J_{i}(\theta) \xi 
    = \left\langle \grad J_{i}(\theta),\ \xi \right\rangle_{\theta} 
    = \left\langle H(t_{i})^{*}\grad L_{y}(\hat{y}_i(\theta) - y_{i}),\ v(t_{i}) \right\rangle \\
    + \left\langle T(t_{i})^{*}\grad L_{y}(\hat{y}_i(\theta) - y_{i}),\ \xi \right\rangle_{\theta},
    \label{eqn:SubobjectivePerturbation}
\end{multline}
and we wish to express its dependence explicitly on $\xi$ in order to compute $\grad J_{i}(\theta)$.
The second term is already in the desired form, so we will focus on revealing the implicit dependence of the first term on $\xi$.
To this end, we denote the first term of \cref{eqn:SubobjectivePerturbation} by $\beta_i = \left\langle H(t_{i})^{*}\grad L_{y}(\hat{y}_i(\theta) - y_{i}),\ v(t_{i}) \right\rangle$ and
we construct a signal $\lambda _i(t)$ so that
\begin{equation}
    \beta_i = \left\langle \lambda _i(t_0),\ v(t_0) \right\rangle + 
    \int_{t_0}^{t_{i}}\left\langle \lambda _i(t),\ S(t)\xi\right\rangle \ \td t.
    \label{eqn:SubobjectiveAdjointPerturbation}
\end{equation}
If such a signal can be constructed, we may write the derivative of the sub-objective in terms of inner products between the gradients and the tangent vector $\xi$, yielding
\begin{equation}
    \beta_i = 
    \left\langle \left( \frac{\partial}{\partial \theta} z(0;\theta) \right)^* \lambda (t_0),\ \xi \right\rangle_{\theta} 
    + \left\langle \int_{t_0}^{t_i} S(t)^* \lambda _i(t) \ \td t,\ \xi \right\rangle_{\theta}.
\end{equation}
To construct $\lambda _i(t)$, we substitute the linearized dynamics \cref{eqn:LinearizedDynamics} into \cref{eqn:SubobjectiveAdjointPerturbation} and integrate by parts
\begin{equation}
    \begin{aligned}
        \beta_i = & \left\langle \lambda _i(t_0),\ v(t_0) \right\rangle +
        \int_{t_0}^{t_{i}}\left\langle \lambda _i(t),\ 
        \ddt v(t) - F(t)v(t)
        \right\rangle \ \td t \\
        = & \left\langle \lambda _i(t_i),\  v(t_i)\right\rangle + 
        \int_{t_0}^{t_{i}}\left\langle -\ddt \lambda _i(t) - F(t)^*\lambda _i(t),\ v(t)
        \right\rangle \ \td t.
    \end{aligned}
\end{equation}
We find that the above equals the first term of \cref{eqn:SubobjectivePerturbation} for all signals $v(t)$, when the adjoint variable $\lambda _i(t)$ satisfies the sub-objective adjoint equations
\begin{equation}
    -\ddt \lambda _{i}(t) = F(t)^{*}\lambda _{i}(t),\qquad
    \lambda _{i}(t_{i}) = H(t_{i})^{*}\grad L_{y}(\hat{y}_i(\theta) - y_{i}).
\end{equation}
These equations are linear and therefore have a unique solution over any time interval.
Finally, by linear superposition we can write the derivative of the entire objective in terms of inner products with the tangent vector $\xi$ as
\begin{multline}
    \left\langle \grad J_{i}(\theta),\ \xi \right\rangle_{\theta} =
    \left\langle \left( \frac{\partial z}{\partial \theta} (t_0;\theta) \right)^* \lambda (t_0),\ \xi \right\rangle_{\theta}
    + \left\langle \int_{t_0}^{t_{L-1}} S(t)^* \lambda (t) \ \td t, \ \xi \right\rangle_{\theta} \\
    + \left\langle \sum_{i=0}^{L-1} T(t_{i})^{*}\grad L_{y}(\hat{y}_{i}(\theta) - y_{i}),\ \xi \right\rangle_{\theta}
    \label{eqn:ObjectiveVariation}
\end{multline}
using an adjoint variable $\lambda (t) = \sum_{i=0}^{L-1} \chi_{[t_0, t_i]}(t) \lambda _i(t)$, which is the unique solution of \cref{eqn:adjoint_equations} over the interval $[t_0, t_{L-1}]$.
Here, 
\begin{equation}
    \chi_{[t_0, t_i]}(t) = \left\{ \begin{matrix}[ll] 1, & \mbox{if $t\in [t_0, t_i]$}, \\ 0, & \mbox{otherwise}, \end{matrix} \right.
\end{equation}
denotes the indicator function for the interval $[t_0, t_i]$.
\end{proof}

\begin{proof}[Proof of \cref{prop:terms_for_adjoint} (Required Terms for Gradient)]
Our proof of each expression follows directly from the definition of the adjoint of a linear operator between finite-dimensional real inner product spaces.
Choosing a pair of vectors $v,w\in\mathbb{R}^r$ with the Euclidean inner product, we have
\begin{equation}
    \left\langle F(t)v,\ w \right\rangle = \left(\frac{\partial \tilde{f}}{\partial  z}( z(t),  u(t)) v\right)^T w
    = \left\langle v,\ \left(\frac{\partial \tilde{f}}{\partial  z}( z(t),  u(t))\right)^T w \right\rangle,
\end{equation}
which implies \cref{eqn:F}.
In precisely the same way we obtain \cref{eqn:H}.

Now we consider a vector $(X, Y) \in T_{(\Phi,\Psi)}\bar{\mathcal{M}}$ with the inner product \cref{eqn:structure_space_metric} and a vector $w \in \mathbb{R}^{m}$ with the Euclidean inner product.
Differentiating \cref{eqn:reduced_order_model_representative} and applying the permutation identity for the trace yields
\begin{equation}
    \left\langle T(t)(X, Y),\ w \right\rangle = \left\langle \frac{\partial g}{\partial x}(\Phi z(t)) X z(t),\ w \right\rangle
    = \Tr\left( z(t) w^T \frac{\partial g}{\partial x}(\Phi z(t)) X \right).
\end{equation}
Taking the adjoint twice, we obtain
\begin{equation}
\begin{aligned}
    \left\langle T(t)(X, Y),\ w \right\rangle &= \Tr\left\{ \left[ \left( \frac{\partial g}{\partial x}(\Phi z(t))\right)^T  w z(t)^T \right]^T  X \right\} \\
    &= \left\langle \left( \left( \frac{\partial g}{\partial x}(\Phi z(t))\right)^T  w z(t)^T,\ 0\right), \ (X,\ Y) \right\rangle_{(\Phi, \Psi)},
\end{aligned}
\end{equation}
from which we conclude that \cref{eqn:T} holds for all $w \in \mathbb{R}^{m}$.

Consider a vector $(X, Y) \in T_{(\Phi,\Psi)}\bar{\mathcal{M}} = T_{\Phi}\mathbb{R}_*^{n,r} \times T_{\Psi}\mathbb{R}_*^{n,r}$ with the inner product \cref{eqn:structure_space_metric} and a vector $v \in \mathbb{R}^{r}$ with the Euclidean inner product.
To simplify our expressions we let $A = (\Psi^T \Phi)^{-1}$, decompose $S(t)(X,Y) = S_{\Phi}(t)X + S_{\Psi}(t)Y$, and observe that the adjoint is also decomposed according to
\begin{equation}
    S(t)^*v = \left( S_{\Phi}(t)^*v, \ S_{\Psi}(t)^*v \right).
\end{equation}
Differentiating \cref{eqn:reduced_order_model_representative}, we obtain
\begin{multline}
    S_{\Phi}(t)X 
    = \frac{\partial \tilde{f}}{\partial \Phi}(z(t), u(t);(\Phi, \Psi))X \\
    = A\Psi^T  \frac{\partial f}{\partial x}(\Phi z(t), u(t)) X z(t)
    -A \Psi^T X \underbrace{A \Psi^T f(\Phi z(t), u(t))}_{\tilde{f}(z(t), u(t); (\Phi, \Psi) )}.
\end{multline}
Here, we have differentiated the inverse of a matrix using the well-known formula $\D (M\mapsto M^{-1}) V = - M^{-1} V M^{-1}$ given by Theorem~3 in Section~8.4 of \cite{Magnus2007matrix}.
Using the Euclidean inner product on $\mathbb{R}^r$, the permutation identity for the trace, the definition of the adjoint, and the inner product on the component $T_{\Phi}\mathbb{R}_*^{n,r}$, we write
\begin{equation}
\begin{aligned}
    \left\langle v, \ S_{\Phi}(t)X \right\rangle 
    &= v^T A\Psi^T  \frac{\partial f}{\partial x}(\Phi z(t), u(t)) X z(t)
    - v^T A \Psi^T X \tilde{f}(z(t), u(t)) \\
    &= \Tr\left[ \left( z(t) v^T A\Psi^T  \frac{\partial f}{\partial x}(\Phi z(t), u(t)) - \tilde{f}(z(t), u(t)) v^T A \Psi^T  \right) X \right] \\
    &= \left\langle S_{\Phi}(t)^*v, \ X \right\rangle_{\Phi} = \Tr\left[ \left( S_{\Phi}(t)^*v \right)^T  X \right].
\end{aligned}
\end{equation}
Therefore, the first term of \cref{eqn:S} is given by
\begin{equation}
    S_{\Phi}(t)^*v = \left( \frac{\partial f}{\partial x}(\Phi z(t), u(t)) \right)^T  \Psi A^T v z(t)^T - \Psi A^T v \tilde{f}(z(t), u(t))^T.
\end{equation}
Likewise, we have
\begin{multline}
    S_{\Psi}(t)Y 
    = \frac{\partial \tilde{f}}{\partial \Psi}(z(t), u(t);(\Phi, \Psi))Y \\
    = A Y^T  f(\Phi z(t), u(t)) - A Y^T  \Phi \underbrace{A \Psi^T f(\Phi z(t), u(t))}_{\tilde{f}(z(t), u(t); (\Phi, \Psi) )},
\end{multline}
from which we obtain
\begin{equation}
\begin{aligned}
    \left\langle v, \ S_{\Psi}(t)Y \right\rangle 
    &= v^T A Y^T  \left( f(\Phi z(t), u(t))
    - \Phi \tilde{f}(z(t), u(t) ) \right) \\
    &= \Tr\left[ Y^T  \left( f(\Phi z(t), u(t))
    - \Phi \tilde{f}(z(t), u(t) ) \right) v^T A \right] \\
    &= \left\langle Y, \ S_{\Psi}(t)^*v \right\rangle_{\Psi} = \Tr\left[ Y^T  S_{\Psi}(t)^*v \right].
\end{aligned}
\end{equation}
Therefore, the second term of \cref{eqn:S} is given by
\begin{equation}
    S_{\Psi}(t)^*v = \left( f(\Phi z(t), u(t))
    - \Phi \tilde{f}(z(t), u(t) ) \right) v^T A.
\end{equation}

The derivative of the initial condition $z_0(\Phi, \Psi) = z(t_0; (\Phi, \Psi))$ in \cref{eqn:reduced_order_model_representative} decomposes according to
\begin{equation}
\frac{\partial z_0}{\partial (\Phi, \Psi)}(\Phi, \Psi)(X,Y) = \frac{\partial z_0}{\partial \Phi}(\Phi, \Psi) X + \frac{\partial z_0}{\partial \Psi}(\Phi, \Psi) Y,
\end{equation}
and so the adjoint is given by
\begin{equation}
    \left( \frac{\partial z_0}{\partial (\Phi, \Psi)}(\Phi, \Psi) \right)^*v 
    = \left( \left( \frac{\partial z_0}{\partial \Phi}(\Phi, \Psi) \right)^*v , \ \left( \frac{\partial z_0}{\partial \Psi}(\Phi, \Psi) \right)^*v \right).
\end{equation}
Differentiating the initial condition in \cref{eqn:reduced_order_model_representative} with respect to $\Phi$ gives
\begin{equation}
    \frac{\partial z_0}{\partial \Phi}(\Phi, \Psi) X = -A \Psi^T  X A \Psi^T  x_0 = -A \Psi^T  X z_0(\Phi, \Psi) ,
\end{equation}
from which obtain
\begin{equation}
\begin{aligned}
    \left\langle v, \ \frac{\partial z_0}{\partial \Phi}(\Phi, \Psi) X \right\rangle 
    &= - v^T A \Psi^T  X z_0
    = -\Tr\left[ z_0 v^T A \Psi^T  X \right] \\
    &= \Tr\left[ \left( - \Psi A^T v z_0^T \right)^T  X \right].
\end{aligned}
\end{equation}
By definition of the adjoint and the inner product on the component $T_{\Phi}\mathbb{R}_*^{n,r}$, we verify that the first term of \cref{eqn:dz0_adj} is given by
\begin{equation}
    \left( \frac{\partial z_0}{\partial \Phi}(\Phi, \Psi) \right)^*v = - \Psi A^T v z_0^T.
\end{equation}
Differentiating the initial condition in \cref{eqn:reduced_order_model_representative} with respect to $\Psi$ gives
\begin{equation}
    \frac{\partial z_0}{\partial \Psi}(\Phi, \Psi) Y = A Y^T x_0 - A Y^T  \Phi \underbrace{A \Psi^T  x_0}_{z_0(\Phi, \Psi)} = A Y^T \left( x_0 - \Phi z_0 \right),
\end{equation}
from which we obtain
\begin{equation}
    \left\langle v, \ \frac{\partial z_0}{\partial \Psi}(\Phi, \Psi) Y \right\rangle
    = v^T A Y^T \left( x_0 - \Phi z_0 \right)
    = \Tr\left[ Y^T \left( x_0 - \Phi z_0 \right) v^T A \right].
\end{equation}
Therefore, by definition of the adjoint and the inner product on the component $T_{\Psi}\mathbb{R}_*^{n,r}$, we verify that the second term of \cref{eqn:dz0_adj} is given by
\begin{equation}
    \left( \frac{\partial z_0}{\partial \Psi}(\Phi, \Psi) \right)^*v = \left( x_0 - \Phi z_0 \right) v^T A.
\end{equation}

Finally, we compute the gradient of the regularization \cref{eqn:regularization} by considering a tangent vector $(X, Y) \in T_{(\Phi,\Psi)}\bar{\mathcal{M}}$ with the inner product \cref{eqn:structure_space_metric}.
Differentiating $\rho\circ\pi$ along $(X, Y)$ when $\det(\Psi^T \Phi) > 0$ gives
\begin{multline}
    \D(\rho\circ\pi)(\Phi, \Psi)(X,Y) = \Tr{\left\{(\Phi^T \Phi)^{-1}(\Phi^T X + X^T \Phi)\right\}} \\
    + \Tr{\left\{(\Psi^T \Psi)^{-1}(\Psi^T Y + Y^T \Psi)\right\}}
    - 2\Tr{\left\{(\Psi^T \Phi)^{-1}(\Psi^T X + Y^T \Phi)\right\}}.
\end{multline}
Here, we have differentiated the log-determinants of matrices with positive determinants using the formula $\D ( M \mapsto \log \det M ) V = \Tr(M^{-1}V)$ given by Theorem~2 in Section~8.4 of \cite{Magnus2007matrix}.
Applying permutation identities for the trace and collecting terms we have
\begin{multline}
    \D(\rho\circ\pi)(\Phi, \Psi)(X,Y) = 2\Tr{\left\{ \left[(\Phi^T \Phi)^{-1}\Phi^T  - (\Psi^T \Phi)^{-1}\Psi^T  \right] X \right\}} \\
    + 2\Tr{\left\{ Y^T  \left[ \Psi (\Psi^T \Psi)^{-1} - \Phi (\Psi^T \Phi)^{-1} \right] \right\}},
\end{multline}
yielding
\begin{multline}
    \D(\rho\circ\pi)(\Phi, \Psi)(X,Y) = 2\Tr{\left\{ (\Phi^T \Phi)^{-1} \left[\Phi - \Psi (\Phi^T \Psi)^{-1} (\Phi^T \Phi) \right]^T  X \right\}} \\
    + 2\Tr{\left\{ (\Psi^T \Psi)^{-1} \left[ \Psi - \Phi (\Psi^T \Phi)^{-1}(\Psi^T \Psi) \right]^T  Y \right\}} \\
    = \left\langle \left( 2\left[\Phi - \Psi (\Phi^T \Psi)^{-1} (\Phi^T \Phi) \right],\  2\left[ \Psi - \Phi (\Psi^T \Phi)^{-1}(\Psi^T \Psi) \right] \right), \ (X,Y) \right\rangle_{(\Phi, \Psi)}.
\end{multline}
Under the additional assumption that $\Phi^T \Phi = \Psi^T \Psi = I_r$, we obtain \cref{eqn:regularization_gradient}.
This completes the proof of \cref{prop:terms_for_adjoint}.
\end{proof}

\subsection{Optimization problem using integrated model error}
\label{subapp:integrated_model_error}
In the main part of this paper, we have considered a data-driven optimization objective \cref{eqn:ROM_cost} in which the error is computed by comparing the output of the reduced-order model to samples collected from the full-order model at discrete time instants.
As the temporal sampling becomes fine, we may use numerical quadrature to approximate an objective formed by integrating the error over a finite time interval $[t_0, t_f]$ according to
\begin{equation}
    J(V,W) = \int_{t_0}^{t_f} L_y(\hat{y}(t; (V,W)) - y(t)) \ \td t + \gamma \rho(V,W).
    \label{eqn:ROM_cost_integral}
\end{equation}
The optimization techniques presented in this paper may also be applied to this cost function instead of \cref{eqn:ROM_cost}, with identical theoretical properties thanks to the continuity of $(t, (V,W)) \mapsto \hat{x}(t; (V,W))$ and its partial derivatives with respect to $(V,W)$ proved in \cref{prop:domain_of_existence}.
The required horizontal lift of the gradient for \cref{eqn:ROM_cost_integral} is given below by \cref{thm:adjointOptimization_integral_objective}, which is analogous to \cref{thm:adjointOptimization}.
Likewise, we present \cref{alg:gradient_int_obj} for computing the gradient via numerical quadrature, as required for optimization using \cref{alg:conj_grad_alg}.

\begin{theorem}[Gradient of integral objective]
\label{thm:adjointOptimization_integral_objective}
Suppose we have an output signal $y(t)$, $t\in [t_0, t_f]$ generated by the full-order model \cref{eqn:full_order_model} with initial condition $x(t_0) = x_0$ and input signal $u(t)$.
Consider the reduced-order model representative \cref{eqn:reduced_order_model_representative} with parameters $\theta = (\Phi, \Psi)$ in the structure space $\bar{\mathcal{M}}$, which is a Riemannian manifold.
With $\pi(\theta) \in \mathcal{D}$, let $\hat{y}(t;\theta)$ be the observation at time $t$ generated by \cref{eqn:reduced_order_model_representative}.
Then the cost function
\begin{equation}
    \bar{J}(\theta) := \int_{t_0}^{t_f} L_y(\hat{y}(t; \theta) -  y(t)) \ \td t,
    \label{eqn:GlobalCostFunctional_int}
\end{equation}
measuring the error between the observations generated by the models, is differentiable at every $\theta \in \pi^{-1}(\mathcal{D})$.
Let $F(t)$, $S(t)$, $H(t)$, and $T(t)$ be as in \cref{thm:adjointOptimization}, let $g(t) = \grad L_y(\hat{y}(t;\theta) -  y(t))$,
and define an adjoint variable $\lambda (t)$ that satisfies
\begin{equation}
    -\ddt \lambda (t) = F(t)^*\lambda (t) + H(t)^* g(t),\qquad t\in [t_0,t_f], \qquad \lambda (t_f) = 0.
\label{eqn:adjoint_equations_int_obj}
\end{equation}
Then the gradient of the cost function \cref{eqn:GlobalCostFunctional_int} is given by
\begin{equation}
\boxed{
    \grad \bar{J}(\theta) =  \left( \frac{\partial z}{\partial \theta} (t_0;\theta) \right)^*\lambda (t_0) + \int_{t_0}^{t_{f}} \left[ S(t)^*\lambda (t) + T(t)^*g(t) \right] \ \td t.
    }
\label{eqn:parameter_gradient_int_obj}
\end{equation}
\end{theorem}
\begin{proof}[Proof of \cref{thm:adjointOptimization_integral_objective}]
Let $z_t(\theta) = z(t;\theta)$, $\xi \in T_{\theta}\bar{\mathcal{M}}$, and $v(t) = \D z_t(\theta)$ be as in the proof of \cref{thm:adjointOptimization} and let $\hat{y}_t(\theta) = \hat{y}(t;\theta) = \tilde{g}(z(t;\theta);\theta)$.
Then, as in the proof of \cref{thm:adjointOptimization}, we have
\begin{equation}
    \ddt v(t) - F(t) v(t) = S(t) \xi, \qquad
    v(t_0) = \left(\frac{\partial z}{\partial \theta} (t_0;\theta)\right) \xi, \qquad \mbox{and}
    \label{eqn:linearized_dynamics_1}
\end{equation}
\begin{equation}
    \D \hat{y}_t(\theta) \xi = H(t) v(t) + T(t) \xi.
\end{equation}
Thanks to \cref{prop:domain_of_existence}, the derivative of the integrand in \cref{eqn:GlobalCostFunctional_int} with respect to $\theta$ is continuous in both $t$ and $\theta$, so we may differentiate under the integral, yielding
\begin{equation}
    \D J(\theta) \xi 
    = \int_{t_0}^{t_f} \left\langle H(t)^*g(t), \ v(t) \right\rangle \ \td t
    + \left\langle \int_{t_0}^{t_f} T(t)^*g(t) \ \td t ,\ \xi \right\rangle_{\theta}
    \label{eqn:int_cost_fun_deriv}
\end{equation}
Substituting \cref{eqn:adjoint_equations_int_obj} into the first term above and integrating by parts gives
\begin{equation}
\begin{aligned}
    \int_{t_0}^{t_f} \left\langle H(t)^*g(t), \ v(t) \right\rangle \ \td t 
    &= -\int_{t_0}^{t_f} \left\langle \ddt \lambda(t), \ v(t) \right\rangle \ \td t - \int_{t_0}^{t_f} \left\langle F(t)^* \lambda(t), \ v(t) \right\rangle \ \td t \\
    &= \left\langle \lambda(t_0), \ v(t_0) \right\rangle + \int_{t_0}^{t_f} \left\langle \lambda(t), \ \ddt v(t) - F(t) v(t) \right\rangle \ \td t
\end{aligned}
\end{equation}
Substituting the linearized dynamics \cref{eqn:linearized_dynamics_1}, we find
\begin{equation}
\begin{aligned}
    \int_{t_0}^{t_f} \left\langle H(t)^*g(t), \ v(t) \right\rangle \ \td t 
    &= \left\langle \lambda(t_0), \ \left(\frac{\partial z}{\partial \theta} (t_0;\theta)\right) \xi \right\rangle + \int_{t_0}^{t_f} \left\langle \lambda(t), \ S(t)\xi \right\rangle \ \td t \\
    &= \left\langle \left(\frac{\partial z}{\partial \theta} (t_0;\theta)\right)^*\lambda(t_0),\ \xi \right\rangle_{\theta} + \left\langle\int_{t_0}^{t_f} S(t)^*\lambda(t) \ \td t,\ \xi \right\rangle_{\theta}.
\end{aligned}
\end{equation}
Substituting back into \cref{eqn:int_cost_fun_deriv}, yields \cref{eqn:parameter_gradient_int_obj} by definition of the gradient as the Riesz representative of the derivative in $T_{\theta}\bar{\mathcal{M}}$.
\end{proof}

\begin{algorithm}
\caption{Compute integral cost function gradient with respect to $(\Phi,\Psi)$}
\label{alg:gradient_int_obj}
\begin{algorithmic}[1]
\STATE{\textbf{input}: orthonormal representatives $(\Phi, \Psi)\in\pi^{-1}(V,W)$, initial condition $x_0$, FOM output signal $y(t)$ for $t\in [t_0, t_f]$, regularization weight $\gamma$, quadrature points $\{ s_l \}_{l=0}^{L-1}$ and weights $\{w_l\}_{l=0}^{L-1}$ for integration over $[t_0, t_f]$}
\STATE{Assemble and simulate the ROM representative \cref{eqn:reduced_order_model_representative} from initial condition $z_0 = \Psi^T x_0$, storing the trajectory $z(t)$ and predicted output $\hat{y}(t)$ via interpolation. \label{algstep:assemble_ROM_int_obj}}
\STATE{Solve the adjoint equation \cref{eqn:adjoint_equations_int_obj} backwards in time, storing the signal $\lambda(t)$.}
\STATE{Initialize with gradient due to initial condition: $\grad \bar{J} \leftarrow \left( \frac{\partial z}{\partial (\Phi, \Psi)} (t_0) \right)^*\lambda(t_0)$.}
\STATE{Add integral part of gradient using quadrature $$\grad \bar{J} \leftarrow \grad \bar{J} + \sum_{l=0}^{L-1} w_l \big[ S(s_l)^*\lambda(s_l) + T(s_l)^*\grad L_y(\hat{y}(s_l) - y(s_l)) \big]$$}
\STATE{Add regularization: $\grad \bar{J} \leftarrow \grad \bar{J} + \gamma \grad (\rho\circ\pi)(\Phi, \Psi)$.}
\RETURN{$\grad \bar{J}$}
\end{algorithmic}
\end{algorithm}

\section{Convergence Guarantees}
\label{app:CG_convergence}

\begin{proof}[Proof of \cref{thm:CG_convergence_exp}]
We begin by observing that the cost function \cref{eqn:ROM_cost} is twice continuously differentiable on $\mathcal{D}$ thanks to \cref{prop:domain_of_existence} and \cref{asmpn:FOM_is_C2}. 
Therefore, the constant $L_J$ defined by \cref{eqn:Lipschitz_const_for_exp}
is finite because $\{ (p,\xi)\in T\mathcal{M} \ : \ p\in\mathcal{D}_c, \ \Vert \xi\Vert_p = 1 \}$ is compact.

To show that \cref{eqn:Lipschitz_iterates_exp} is satisfied with $L_J$ given by \cref{eqn:Lipschitz_const_for_exp}, we consider an arbitrary fixed iterate $k$, and we drop the subscript $k$ to simplify notation.
Differentiating the cost along the search path $\gamma$, we obtain
\begin{equation}
    \ddt (J\circ\gamma)(t) = \D (J\circ \exp_p)(t\eta)\eta = \left\langle \grad J(\gamma(t)),\ \gamma'(t) \right\rangle_{\gamma(t)},
\end{equation}
for every $t\in\mathbb{R}$ such that $\gamma(t)\in\mathcal{D}_c$.
Differentiating a second time yields
\begin{equation}
    \ddtsq (J\circ\gamma)(t) = 
    \Vert \gamma'(t)\Vert_{\gamma(t)} \left\langle \nabla_{\gamma'(t)/\Vert \gamma'(t)\Vert_{\gamma(t)}}\grad J(\gamma(t)),\ \gamma'(t) \right\rangle_{\gamma(t)},
\end{equation}
thanks to linearity and compatibility of the Riemannian connection with the Riemannian metric and the fact that $\nabla_{\gamma'(t)}\gamma' = 0$ along the geodesic $\gamma$ by definition of the exponential map \cite{doCarmo1992Riemannian}.
By the Cauchy-Schwarz inequality, we obtain
\begin{equation}
    \left\vert \ddtsq (J\circ\gamma)(t) \right\vert \leq
    \left\Vert \nabla_{\gamma'(t)/\Vert \gamma'(t)\Vert_{\gamma(t)}}\grad J(\gamma(t)) \right\Vert_{\gamma(t)}
    \left\Vert \eta \right\Vert_{p}^2
    \leq L_J \left\Vert \eta \right\Vert_{p}^2.
\end{equation}
Therefore, as long as $\gamma(t) \in \mathcal{D}_c$ for every $t\in [0, \alpha_k]$, we may integrate the above inequality to produce the desired Lipschitz estimate \cref{eqn:Lipschitz_iterates_exp}.
Consequently the Riemannian version of Zoutendijk's theorem given by Theorem~2 in \cite{Ring2012optimization} (Theorem~4.1 in \cite{Sato2016dai}) holds and the algorithm converges in the sense of \cref{eqn:grad_convergence} thanks to Theorem~4.2 in \cite{Sato2016dai}.
\end{proof}

The assumption in \cref{thm:CG_convergence_exp} that the search paths remain in a compact set $\mathcal{D}_c$ is not in conflict with the Wolfe conditions.
In fact, the Wolfe conditions can always be satisfied along paths contained within the sub-level set of the cost $J$ corresponding to the cost of the initial iterate $(V_0, W_0)\in\mathcal{D}$.
This is easily verified using the same argument as Lemma~3.1 in J. Nocedal and S. J. Wright \cite{Nocedal2006numerical} together with the fact that the Riemannian Dai-Yuan search direction is always a direction of descent when the gradient is nonzero thanks to Proposition~4.1 in \cite{Sato2016dai}.
Consequently, we may take $\mathcal{D}_c$ to be a sub-level-set of the cost at some value greater than the cost of the initial iterate thanks to the compactness of sub-level-sets established by \cref{lem:sub_level_sets_are_compact}, below.
\begin{lemma}
\label{lem:sub_level_sets_are_compact}
Any non-empty sub-level-set $\mathcal{S}_C = \{ (V,W) \in \mathcal{D} \ : \ J(V,W) \leq C \}$ of the cost function $J$ defined by \cref{eqn:ROM_cost} is compact.
\end{lemma}
\begin{proof}[Proof of \cref{lem:sub_level_sets_are_compact}]
Since $\mathcal{M} = \mathcal{G}_{n,r}\times \mathcal{G}_{n,r}$ is compact, it suffices to prove that $\mathcal{S}_C$ is closed in $\mathcal{M}$
Suppose that $\{ (\tilde{V}_k, \tilde{W}_k) \}_{k=1}^{\infty} \subset \mathcal{S}_C$ is a sequence such that $(\tilde{V}_k, \tilde{W}_k)\to (\tilde{V}, \tilde{W})\in\mathcal{M}$.
Then it is clear that $(\tilde{V}, \tilde{W}) \in \mathcal{P}$, for if it were not then \cref{thm:regularization_properties} would give $\rho(\tilde{V}, \tilde{W}) \to \infty$ and so $J(\tilde{V}, \tilde{W}) \to \infty$.
Moreover, if $(\tilde{V},\tilde{W})\in\mathcal{P}\setminus\mathcal{D}$ then by \cref{prop:domain_of_existence}
we would have
\begin{equation}
    \max_{t\in [t_0, t_{L-1}]} \Vert \hat{x}(t;(\tilde{V}_k,\tilde{W}_k))\Vert \to \infty 
    \quad \mbox{as}\quad k \to \infty,
\end{equation}
which implies that $J(\tilde{V}_k, \tilde{W}_k) \to \infty$ by \cref{asmpn:blow_up}.
This contradicts the assumption that $J(\tilde{V}_k, \tilde{W}_k) \leq C$ for every $k$.
Therefore, the limit point $(\tilde{V}, \tilde{W}) \in \mathcal{D}$.
In \cref{prop:domain_of_existence} we showed that $J$ is continuous on $\mathcal{D}$, and so
\begin{equation}
    J(\tilde{V}, \tilde{W}) = \lim_{k\to\infty} J(\tilde{V}_k, \tilde{W}_k) \leq C.
\end{equation}
Therefore, $(\tilde{V}, \tilde{W}) \in \mathcal{S}_C$ and we conclude that $\mathcal{S}_C$ is closed in $\mathcal{M}$.
\end{proof}
If $\mathcal{D} = \mathcal{P}$, as in the case when $\frac{\partial f}{\partial x}$ is bounded (see \cref{prop:domain_of_existence}), then we may let $\mathcal{D}_c = \{ (V,W) \in \mathcal{P} \ : \ \rho(V,W) \leq C \}$ be a sub-level-set of the regularization function \cref{eqn:regularization} with $C \geq J(V_0,W_0)/\gamma$.
Here, $\mathcal{D}_c$ is compact thanks to \cref{thm:regularization_properties}, and $\mathcal{D}_c$ contains every $(V,W)\in\mathcal{P}$ with $J(V,W) \leq J(V_0, W_0)$ since
\begin{equation}
    \rho(V,W) \leq \frac{1}{\gamma}J(V,W) \leq \frac{1}{\gamma}J(V_0,W_0) \leq C.
\end{equation}

\subsection{Another choice of retraction and vector transport}
\label{subapp:another_retraction}
As described by P. A. Absil et al. \cite{Absil2004Riemannian}, the exponential map and parallel translation along geodesics may be replaced by more general maps called ``retractions'' and ``vector transports'' in the geometric conjugate gradient algorithm.
In particular, it is common (see Example~4.1.5 in \cite{Absil2009optimization}) to use a retraction given by
\begin{equation}
    R_V(\xi) = \Range{(\Phi + \bar{\xi}_{\Phi})}
    \label{eqn:retraction}
\end{equation}
on the Grassmann manifold in place of the exponential map $\exp_V(\xi)$.
Using this retraction, $\left(\Phi + \bar{\xi}_{\Phi}, \Psi + \bar{\zeta}_{\Psi}\right)$ are matrix representatives of $R_{(V,W)}(\xi, \zeta)$ on $\mathcal{G}_{n,r}\times\mathcal{G}_{n,r}$.
While the formula proved by Theorem~2.3 in \cite{Edelman1998geometry} for the exponential map retains orthogonality of the representatives, we must re-orthogonalize $\Phi + \bar{\xi}_{\Phi}$ and $\Psi + \bar{\zeta}_{\Psi}$ using, e.g., QR factorization.

Parallel translation along the geodesics generated by the exponential map may be replaced by a more general notion of vector transport \cite{Absil2004Riemannian}.
A common vector transport is found by differentiating the retraction
\begin{equation}
    \mathcal{T}_{V,\xi}(\eta) := \D R_V(\xi)\eta,
\end{equation}
whose horizontal lift, according to Example 8.1.10 in \cite{Absil2009optimization}, is given explicitly by orthogonal projection onto the horizontal space,
\begin{equation}
    \overline{\mathcal{T}_{V, \xi}(\eta)}_{\Phi + \bar{\xi}_{\Phi}} = P^{h}_{\Phi + \bar{\xi}_{\Phi}} (\bar{\eta}_{\Phi}),
    \qquad 
    P_{\Phi}^h(X) = X - \Phi (\Phi^T \Phi)^{-1}\Phi^T X,
    \label{eqn:vector_transport}
\end{equation}
for any $\Phi\in\mathbb{R}_*^{n,r}$ with $\Range{\Phi} = V$.

The same convergence guarantee holds for the Riemannian Dai-Yuan conjugate gradient algorithm \cite{Sato2016dai} when the exponential map and parallel translation are replaced by the above retraction \cref{eqn:retraction} and vector transport \cref{eqn:vector_transport}.
In fact the convergence guarantees proved in \cite{Ring2012optimization, Sato2015new, Sato2016dai} were formulated for general retractions and vector transports together with the Lipschitz condition
\begin{equation}
    \left\vert D(J\circ R_{p_k})(\alpha_k \eta_k)\eta_k - D(J\circ R_{p_k})(0)\eta_k \right\vert \leq L_J \alpha_k \Vert \eta_k\Vert_{p_k}^2,
    \label{eqn:Lipschitz_iterates}
\end{equation}
required for Zoutendijk's theorem (Theorem~2 in \cite{Ring2012optimization}, Theorem~4.1 in \cite{Sato2016dai}) to hold.
We verify \cref{eqn:Lipschitz_iterates} in the following \cref{thm:CG_convergence} by bounding the second derivative of each line search objective $J_k:t\mapsto J(R_{p_k}(t\eta_k))$ uniformly over $t$ and $k$.
\begin{theorem}
\label{thm:CG_convergence}
Suppose that there is a compact subset $\mathcal{D}_c$ of the domain $\mathcal{D}$ (defined in \cref{prop:domain_of_existence}) such that for every iteration $k=0,1,2,\ldots$, we have
\begin{equation}
    \gamma_k(t) = R_{p_k}(t\eta_k) \in \mathcal{D}_c \qquad \forall t\in [0,\alpha_k].
\end{equation}
Let $\nabla$ denote the Riemannian connection on $\mathcal{G}_{n,r}\times\mathcal{G}_{n,r}$ with metric given by \cref{eqn:Riemannian_metric}.
Then the Lipschitz condition\cref{eqn:Lipschitz_iterates} holds with
\begin{equation}
    L_J = \max_{\substack{(p,\xi)\in T\mathcal{M} : \\
    p \in \mathcal{D}_c, \ 
    \Vert \xi\Vert_p = 1}} \left[ \sqrt{r} \left\Vert \grad J(p) \right\Vert_p  + \left\Vert (\nabla_{\xi} \grad J)(p) \right\Vert_p \right] < \infty,
\end{equation}
and the geometric conjugate gradient algorithm with Dai-Yuan coefficient \cref{eqn:Riemannian_DY_coeff} and $\alpha_k$ satisfying the Wolfe conditions \cref{eqn:Wolfe_conditions} converges in the sense of \cref{eqn:grad_convergence}.
\end{theorem}
The key to establishing this result is the boundedness of velocity and acceleration of each path $\gamma_k:t\mapsto R_{p_k}(t\eta_k)$ given by \cref{lem:retraction_vel_and_acc}, below.
\begin{lemma}
\label{lem:retraction_vel_and_acc}
Let $V = \Range{(\Phi)}\in\mathcal{G}_{n,r}$ and $\xi \in T_V\mathcal{G}_{n,r}$ and consider the curve $\gamma : \mathbb{R} \to \mathcal{G}_{n,r}$ defined by the retraction \cref{eqn:retraction} according to
\begin{equation}
    \gamma(t) = R_V(t\xi) = \Range{\left(\Phi + t \bar{\xi}_{\Phi}\right)}.
\end{equation}
The velocity and acceleration of this curve are bounded by
\begin{equation}
    \left\Vert \frac{\td \gamma}{\td t}(t) \right\Vert_{\gamma(t)} \leq \Vert \xi \Vert_{V}, \qquad
    \left\Vert \frac{\mathrm{D}}{\td t}\frac{\td \gamma}{\td t}(t) \right\Vert_{\gamma(t)} \leq \sqrt{r} \Vert \xi \Vert_{V}^2,
\end{equation}
where $\frac{\mathrm{D}}{\td t}$ denotes the covariant derivative along $\gamma$ induced by the Riemannian connection on $\mathcal{G}_{n,r}$ with metric given by \cref{eqn:Riemannian_metric}.
Of course the same result holds for the element-wise retraction on any Cartesian product of such Grassmann manifolds.
\end{lemma}
The proof of \cref{lem:retraction_vel_and_acc} relies on the non-expansiveness of the vector transport:
\begin{lemma}
For every $\xi,\eta \in T_V\mathcal{G}_{n,r}$, the vector transport \cref{eqn:vector_transport} satisfies
\begin{equation}
    \left\Vert \mathcal{T}_{V, \xi}(\eta) \right\Vert_{R_V(\xi)} \leq \left\Vert \eta \right\Vert_V.
    \label{eqn:non_expanding_transport}
\end{equation}
\label{lem:non_expanding_transport}
\end{lemma}
\begin{proof}[Proof of \cref{lem:non_expanding_transport}]
Let $\Phi_0\in\mathbb{R}_*^{n,r}$ such that $V = \Range{\Phi_0}$ and for ease of notation, let $\Phi_1 = \Phi_0 + \bar{\xi}_{\Phi_0}$, $G_0 = \Phi_0^T \Phi_0$, and $G_1 = \Phi_1^T \Phi_1$.
We observe that $\Phi_0^T \bar{\xi}_{\Phi_0} = 0$ and so $G_1 = G_0 + \bar{\xi}_{\Phi_0}^T \bar{\xi}_{\Phi_0}$ since $\bar{\xi}_{\Phi_0}$ lies in the horizontal subspace at $\Phi_0$.
By definition of the transport \cref{eqn:vector_transport} and the Riemannian metric \cref{eqn:Riemannian_metric}, we have
\begin{equation}
\begin{aligned}
    \left\Vert \mathcal{T}_{V, \xi}(\eta) \right\Vert_{R_V(\xi)}^2 &= \Tr\left[G_1^{-1} \left(P^h_{\Phi_1}\bar{\eta}_{\Phi_0}\right)^T \left(P^h_{\Phi_1}\bar{\eta}_{\Phi_0}\right)  \right] \\
    &= \Tr\left[G_1^{-1} (\bar{\eta}_{\Phi_0})^T \bar{\eta}_{\Phi_0}  \right] - \Tr\left[G_1^{-1} (\bar{\eta}_{\Phi_0})^T \Phi_1 G_1^{-1}\Phi_1^T \bar{\eta}_{\Phi_0} \right] \\
    &\leq \Tr\left[G_1^{-1} (\bar{\eta}_{\Phi_0})^T \bar{\eta}_{\Phi_0}  \right] 
    \leq \Tr\left[G_0^{-1} (\bar{\eta}_{\Phi_0})^T \bar{\eta}_{\Phi_0}  \right] = \Vert \eta \Vert_V^2
\end{aligned}
\end{equation}
where the last inequality holds because $G_0^{-1} - G_1^{-1}$ is positive semi-definite \cite{Bhatia1997matrix}.
\end{proof}

\begin{proof}[Proof of \cref{lem:retraction_vel_and_acc}]
Since the velocity vector is given by transporting $\xi$ to the new point $\gamma(t)$ according to
\begin{equation}
    \gamma'(t) := \frac{\td \gamma}{\td t}(t) = \D R_V(t\xi) \xi = \mathcal{T}_{V,t \xi}\xi,
\end{equation}
the boundedness of velocity follows immediately from the non-expanding property of the vector transport stated in \cref{lem:non_expanding_transport}.

To show that the acceleration $\frac{\mathrm{D} \gamma'}{\td t}(t)$ is bounded, we recall that the covariant derivative can be written in terms of the Riemannian connection (see Proposition~2.2 in Chapter~2 of \cite{doCarmo1992Riemannian}), according to
\begin{equation}
    \frac{\mathrm{D} \gamma'}{\td t}(t) = (\nabla_{\gamma'(t)} \Xi)(\gamma(t)),
\end{equation}
where $\Xi$ is any smooth vector field on $\mathcal{G}_{n,r}$ extending $\gamma'$ in a neighborhood of $\gamma(t)$.
To simplify notation, let $\Phi(t) = \Phi + t \bar{\xi}_{\Phi}$ and $G(t) = \Phi(t)^T \Phi(t)$.
Using the expression for the horizontal lift of the Riemannian connection provided by Theorem~3.4 in \cite{Absil2004Riemannian}, we obtain
\begin{equation}
    \overline{\frac{\mathrm{D} \gamma'}{\td t}(t)}_{\Phi(t)} 
    = P^h_{\Phi(t)} \D \overline{\Xi}(\Phi(t)) \overline{\gamma'(t)}_{\Phi(t)},
\end{equation}
where $\overline{\Xi}:\mathbb{R}_*^{n,r}\to\mathbb{R}^{n,r}$ is the map defined by $\overline{\Xi}(\Psi) = \overline{\Xi(\Range{(\Psi)})}_{\Psi}$.
Consider the curve defined by
\begin{equation}
\begin{aligned}
    \Psi_t(\tau) &= \Phi(t) + \tau \overline{\gamma'(t)}_{\Phi(t)} - \tau^2 t \bar{\xi}_{\Phi}G(t)^{-1}(\bar{\xi}_{\Phi})^T \bar{\xi}_{\Phi} \\
    &= \underbrace{\left( \Phi + (t+\tau)\bar{\xi}_{\Phi} \right)}_{\Phi(t+\tau)}\underbrace{\left( I - \tau t G(t)^{-1} (\bar{\xi}_{\Phi})^T \bar{\xi}_{\Phi} \right)}_{A_t(\tau)},
\end{aligned}
\end{equation}
for $\tau$ in a small interval $(-\varepsilon, \varepsilon)$.
Since we have $\Psi_t(0) = \Phi(t)$ and $\frac{\td}{\td \tau} \Psi_t(0) = \overline{\gamma'(t)}_{\Phi(t)}$, we may use the curve to compute
\begin{equation}
    \D \overline{\Xi}(\Phi(t)) \overline{\gamma'(t)}_{\Phi(t)} 
    = \left.\frac{\td}{\td \tau} \overline{\Xi}(\Psi_t(\tau))\right\vert_{\tau = 0}
    = \left.\frac{\td}{\td \tau} \left( \overline{\Xi}(\Phi(t+\tau)) A_t(\tau) \right) \right\vert_{\tau = 0},
\end{equation}
where the second equality holds thanks to \cref{eqn:horizontal_lift_transformation}.
We observe that by definition of $\Xi$ and $\overline{\Xi}$, we have
\begin{equation}
    \overline{\Xi}(\Phi(t+\tau)) 
    = \overline{\gamma'(t+\tau)}_{\Phi(t+\tau)}
    = \overline{\left(\mathcal{T}_{V, (t+\tau) \xi}\xi\right)}_{\Phi(t+\tau)}
    = P^h_{\Phi(t+\tau)} \bar{\xi}_{\Phi},
\end{equation}
and so a simple calculation using \cref{eqn:vector_transport} and the fact that $P^h_{\Phi(t)}\Phi(t) = 0$ for all $t$ yields
\begin{equation}
    \overline{\frac{\mathrm{D} \gamma'}{\td t}(t)}_{\Phi(t)} 
    = P^h_{\Phi(t)} \left.\frac{\td}{\td \tau} \left( P^h_{\Phi(t+\tau)} \bar{\xi}_{\Phi} A_t(\tau) \right) \right\vert_{\tau = 0}
    = -2t P^h_{\Phi(t)} \bar{\xi}_{\Phi} G(t)^{-1} (\bar{\xi}_{\Phi})^T \bar{\xi}_{\Phi}.
\end{equation}
Using the definition of the Riemannian metric \cref{eqn:Riemannian_metric}, cyclic permutation properties of the trace, and existence of the matrix square root $G(t)^{-1/2}$, we obtain
\begin{multline}
\left\Vert \frac{\mathrm{D} \gamma'}{\td t}(t) \right\Vert_{\gamma(t)}^2 
= 4t^2\Tr\left[ G(t)^{-1} (\bar{\xi}_{\Phi})^T \bar{\xi}_{\Phi} G(t)^{-1} (P^h_{\Phi(t)} \bar{\xi}_{\Phi})^T  (P^h_{\Phi(t)} \bar{\xi}_{\Phi}) G(t)^{-1} (\bar{\xi}_{\Phi})^T \bar{\xi}_{\Phi} \right] \\
= 4t^2 \Tr\left[ \left( G(t)^{-1/2} (\bar{\xi}_{\Phi})^T \bar{\xi}_{\Phi} G(t)^{-1/2} \right)^2 G(t)^{-1/2} (P^h_{\Phi(t)} \bar{\xi}_{\Phi})^T  (P^h_{\Phi(t)} \bar{\xi}_{\Phi}) G(t)^{-1/2} \right].
\end{multline}
Since $\Tr(AB) \leq \Tr(A)\Tr(B)$ for any positive semi-definite matrices $A$, $B$, it follows that
\begin{equation}
\left\Vert \frac{\mathrm{D} \gamma'}{\td t}(t) \right\Vert_{\gamma(t)}^2 
\leq 4t^2 \Tr\left[ G(t)^{-1} (\bar{\xi}_{\Phi})^T \bar{\xi}_{\Phi} \right]^2 \underbrace{\Tr\left[ G(t)^{-1} (P^h_{\Phi(t)} \bar{\xi}_{\Phi})^T  (P^h_{\Phi(t)} \bar{\xi}_{\Phi}) \right]}_{\Vert \mathcal{T}_{V,t\xi}(\xi) \Vert_{R_V(t\xi)}^2}
\end{equation}
and so we obtain
\begin{equation}
    \left\Vert \frac{\mathrm{D} \gamma'}{\td t}(t) \right\Vert_{\gamma(t)} \leq 2 \vert t\vert \Tr\left[ G(t)^{-1} (\bar{\xi}_{\Phi})^T \bar{\xi}_{\Phi} \right]\Vert \xi\Vert_V,
\end{equation}
thanks to \cref{lem:non_expanding_transport}.
Letting 
\begin{equation}
    G(0)^{-1/2}(\bar{\xi}_{\Phi})^T \bar{\xi}_{\Phi}G(0)^{-1/2} = Q \Lambda Q^T, \qquad \Lambda = 
    \begin{bmatrix}
    \lambda_1^2 & & \\
    & \ddots & \\
    & & \lambda_r^2
    \end{bmatrix}
\end{equation}
be a symmetric eigen-decomposition, it is readily verified that
\begin{equation}
    2\vert t\vert \Tr\left[ G(t)^{-1} (\bar{\xi}_{\Phi})^T \bar{\xi}_{\Phi} \right] 
    = 2\vert t\vert \Tr\left[ (I + t^2 \Lambda)^{-1}\Lambda \right] = \sum_{i=1}^r \frac{2\vert t\vert \lambda_i^2}{1 + t^2\lambda_i^2}.
\end{equation}
Optimizing each term in the sum separately over $t$, we find that $\frac{2\vert t\vert \lambda_i^2}{1 + t^2\lambda_i^2} \leq \vert \lambda_i\vert$ for all $t\in\mathbb{R}$.
It follows from the Cauchy-Schwarz inequality that
\begin{multline}
    2\vert t\vert \Tr\left[ G(t)^{-1} (\bar{\xi}_{\Phi})^T \bar{\xi}_{\Phi} \right]
    \leq \sum_{i=1}^r \vert \lambda_i\vert \\
    \leq \sqrt{r} \sqrt{\sum_{i=1}^r\lambda_i^2} 
    = \sqrt{r}\sqrt{\Tr\left[ G(0)^{-1}(\bar{\xi}_{\Phi})^T \bar{\xi}_{\Phi} \right]} = \sqrt{r}\Vert \xi\Vert_V,
\end{multline}
which yields the desired boundedness result for acceleration,
\begin{equation}
    \left\Vert \frac{\mathrm{D} \gamma'}{\td t}(t) \right\Vert_{\gamma(t)} \leq \sqrt{r} \Vert \xi\Vert_V^2.
\end{equation}
\end{proof}

\begin{proof}[Proof of \cref{thm:CG_convergence}]
We begin by observing that the constant $L_J$ defined by 
\begin{equation}
    L_J = \max_{\substack{(p,\xi)\in T\mathcal{M} : \\
    p \in \mathcal{D}_c, \ 
    \Vert \xi\Vert_p = 1}} \left[ \sqrt{r} \left\Vert \grad J(p) \right\Vert_p  + \left\Vert (\nabla_{\xi} \grad J)(p) \right\Vert_p \right],
\end{equation}
is finite thanks to the compactness of $\{ (p,\xi)\in T\mathcal{D}_c \ : \ \Vert \xi\Vert_p = 1 \}$ and twice continuous differentiability of the cost function established by \cref{prop:domain_of_existence} and \cref{asmpn:FOM_is_C2}.

We consider an arbitrary, fixed iterate $k$, and drop the subscript $k$ to simplify notation.
Differentiating the cost along the search path $\gamma$, we obtain
\begin{equation}
    \ddt (J\circ\gamma)(t) = \D (J\circ R_p)(t\eta)\eta = \left\langle \grad J(\gamma(t)),\ \gamma'(t) \right\rangle_{\gamma(t)},
\end{equation}
for every $t\in\mathbb{R}$ such that $\gamma(t)\in\mathcal{D}_c$.
Differentiating a second time gives
\begin{multline}
    \ddtsq (J\circ\gamma)(t) = 
    \Vert \gamma'(t)\Vert_{\gamma(t)} \left\langle \nabla_{\gamma'(t)/\Vert \gamma'(t)\Vert_{\gamma(t)}}\grad J(\gamma(t)),\ \gamma'(t) \right\rangle_{\gamma(t)} \\
    + \left\langle \grad J(\gamma(t)),\ \frac{\mathrm{D} \gamma'}{\td t}(t) \right\rangle_{\gamma(t)},
\end{multline}
thanks to linearity and compatibility of the Riemannian connection with the Riemannian metric \cite{doCarmo1992Riemannian}.
By the Cauchy-Schwarz inequality and \cref{lem:retraction_vel_and_acc}, we obtain
\begin{equation}
\begin{aligned}
    \left\vert \ddtsq (J\circ\gamma)(t) \right\vert &\leq
    \left(\left\Vert \nabla_{\gamma'(t)/\Vert \gamma'(t)\Vert_{\gamma(t)}}\grad J(\gamma(t)) \right\Vert_{\gamma(t)}
    + \sqrt{r}\left\Vert \grad J(\gamma(t)) \right\Vert_{\gamma(t)}\right) \left\Vert \eta \right\Vert_{p}^2 \\
    &\leq L_J \left\Vert \eta \right\Vert_{p}^2.
\end{aligned}
\end{equation}
As long as $\gamma(t) \in \mathcal{D}_c$ for every $t\in [0, \alpha_k]$, we may integrate the above inequality to produce the desired Lipschitz estimate \cref{eqn:Lipschitz_iterates}.
Consequently the Riemannian version of Zoutendijk's theorem given by Theorem~2 in \cite{Ring2012optimization} (Theorem~4.1 in \cite{Sato2016dai}) holds and the algorithm converges in the sense of \cref{eqn:grad_convergence} thanks to Theorem~4.2 in \cite{Sato2016dai}.
\end{proof}

\section{The role of pressure in incompressible flows}
\label{app:pressure_app}
As mentioned in the body of the paper in \cref{subsec:flow_description}, the pressure may be removed entirely from the Navier-Stokes formulation \eqref{eq:ns_u}--\eqref{eq:cont}.

Let us write equations \eqref{eq:ns_u}--\eqref{eq:cont} in compact form as
\begin{equation}
\label{eq:block_eqtn}
    \begin{bmatrix}
    I & 0 \\ 0 & 0
    \end{bmatrix}\frac{\partial}{\partial t}\begin{bmatrix}
    q\\p
    \end{bmatrix} = \begin{bmatrix}
    F & -G \\ D & 0
    \end{bmatrix}\begin{bmatrix}
    q\\p
    \end{bmatrix} + 
    \begin{bmatrix}
    g(q) \\ 0
    \end{bmatrix},
\end{equation}
where $q = (u,v)$, $D$ is the divergence operator, $G$ is the gradient operator, $F$ contains the vector Laplacian and $g(q)$ contains the nonlinear terms in the momentum equations \eqref{eq:ns_u} and \eqref{eq:ns_v}.
Taking the divergence of the first row of \eqref{eq:block_eqtn}, and using $D q = 0$ (by the second row of \eqref{eq:block_eqtn}) along with $DFq = FDq = 0$, we obtain a Poisson equation
\begin{equation}
\label{eq:poisson}
    \underbrace{DG}_{\widetilde{L}} p = Dg(q),
\end{equation}
where $\widetilde{L}$ is the scalar Laplacian operator.
Often, instead of prescribing pressure boundary conditions at the physical
boundaries of the spatial domain, a unique solution to \eqref{eq:poisson} is
instead computed by fixing the pressure and the pressure gradient at some
location $(r_0,z_0)$ in physical space~\cite{perot}. That is, in cylindrical coordinates
\begin{equation}
    p = \frac{\partial p}{\partial z} = \frac{\partial p}{\partial \xi} = 0\quad \text{at} \ (\xi_0,z_0)\in\Omega,
\end{equation}
where $(\xi_0,z_0)$ may be chosen arbitrarily.
This approach is particularly convenient in numerical methods based on the finite volume or finite difference discretization of the spatial dimensions. 
The pressure may thus be written as
\begin{equation}
    p = \widetilde{L}^{-1} D g(q),
\end{equation}
and consequently \eqref{eq:block_eqtn} may be reduced to 
\begin{equation}
    \frac{\partial }{\partial t}q = Fq + g(q) - G\widetilde{L}^{-1}Dg(q) = f(q),
\end{equation}
which is in the form of \eqref{eqn:full_order_model}.
In practice, in order to compute the action of $f$ on some vector field $q$, we proceed as follows:
\begin{enumerate}
    \item compute $\varphi(q) = Fq + g(q)$,
    \item compute the pressure by solving $\widetilde{L} p = D \varphi(q)$, and finally
    \item compute $f(q) = \varphi(q) - Gp$.
\end{enumerate}

\section{Implementation of QB-balancing for the jet flow}
\label{app:qb_jet}

In this brief appendix we discuss the details behind our implementation of the QB-balancing algorithm for the jet flow. 
Henceforth, equation numbers and notation are as in \cite{Benner2017balanced}. 
Following \cite{Benner2017balanced}, the balancing transformation is obtain as follows:
\begin{enumerate}
    \item solve Lyapunov equation (3.21) for the $n\times n$ Grammian $\hat{P}_1$,
    \item solve Lyapunov equation (3.20a) for the $n\times n$ Grammian $P_\mathcal{T}$,
    \item solve Lyapunov equation (3.22) for for the $n\times n$ Grammian $\hat{Q}_1$, and finally, 
    \item solve Lyapunov equation (3.20b) for for the $n\times n$ Grammian $Q_\mathcal{T}$.
\end{enumerate}
Given the state dimension ($n = 10^5$) of the jet flow, solving the Lyapunov equations above using existing packages is computationally infeasible. 
Other techniques that leverage the low-rank nature of the solution of the Lyapunov equations are also prohibitive since they require solving several linear systems of size $n$. 
We therefore solve the equations via time-stepping (see, e.g., \cite{Rowley2005model}), which provides and approximation of the integral solution of the Lyapunov equation. 
Another issue is that solving formulas (3.20a) and (3.20b) requires computing $n\times n$ matrices $H(\hat{P}_1 \otimes\hat{P}_1)H^T$ and $H_2(\hat{P}_1\otimes\hat{Q}_1)H_2^T$, where $H$ and $H_2$ are matricizations of the second-order terms in the full-order model (see \cite{Benner2017balanced}).
This poses a storage problem due to the size of these matrices for our system. Moreover, computing the integral solution of (3.20a) and (3.20b) would require solving $n$ initial value problems.
We relieved these problems by employing SVD-truncated snapshot-based approximations in an analogous manner to BPOD with output projection \cite{Rowley2005model} in the context of balancing of high-dimensional linear systems. 
We proceed as follows:
\begin{enumerate}
    \item \label{step:step1} Solve equation (3.21) for $\hat{P}_1 = X X^T$, with $X$ being the data matrix containing the solution of the linear dynamics in response to impulses scaled by $\sqrt{\Delta t}$, where $\Delta t$ is the solution time-step. For instance, if $B \in \mathbb{R}^{n\times d}$, then $X$ will have size $n\times d\,n_t$, where $n_t$ is the number of snapshots along a trajectory. See \cite{Rowley2005model} for additional details. 
    \item \label{step:step2} Perform the (economy-size) singular value decomposition $X = U_X\Sigma_X V_X^T$ and select an appropriate truncation rank $r_X$. This results in a truncated approximation for the Grammian $\hat{P}_{1, r_X} = U_{X,r_X}\Sigma_{X,r_X}^2U_{X,r_X}^T$.
    \item \label{step:step3} Construct an $n\times r_X^2$ matrix $T = H((U_{X,r_X}\Sigma_{X,r_X}) \otimes (U_{X,r_X}\Sigma_{X,r_X}))$. Here we observe that $H(\hat{P}_{1, r_X} \otimes \hat{P}_{1, r_X})H^T = T T^T$, thanks to the mixed-product property of the Kronecker product.
    \item \label{step:step4} Compute the economy-sized SVD of $T=U_H \Sigma_H V_H^*$ and select an appropriate truncation rank $r_H$.
    \item Solve equation (3.20a) for $P_\mathcal{T} = X_\mathcal{T}X_\mathcal{T}^T$ with $H(\hat{P}_1 \otimes\hat{P}_1)H^T$ replaced by its truncation $U_{H,r_H} \Sigma_{H, r_H}^2 U_{H,r_H}^T$ via numerical integration of the linear dynamics $r_H$ times (as in step~\ref{step:step1}).
    \item Perform the (economy-size) singular value decomposition $X_\mathcal{T} = U_{X_\mathcal{T}}\Sigma_{X_\mathcal{T}} V_{X_\mathcal{T}}^T$ and select an appropriate truncation rank $r_{X_\mathcal{T}}$.
    \item Solve equation (3.22) for $\hat{Q}_1 = Y Y^T$ by integrating the adjoint of the linear dynamics with final condition $U_{X_\mathcal{T},i}$ with $i \in \{1,2,\ldots,r_{X_\mathcal{T}}\}$. This step is called ``output projection'' in \cite{Rowley2005model}.
    \item Perform the (economy-size) singular value decomposition $Y = U_Y\Sigma_Y V_Y^T$ and select an appropriate truncation rank $r_Y$. This results in a truncated approximation for the Grammian $\hat{Q}_{1, r_Y} = U_{Y,r_Y}\Sigma_{Y,r_Y}^2U_{Y,r_Y}^T$.
    \item Construct an $n\times (r_X r_Y)$ matrix $S = H_2((U_{X,r_X}\Sigma_{X,r_X}) \otimes (U_{Y,r_Y}\Sigma_{X,r_Y}))$. Here, we observe that $H_2(\hat{P}_{1, r_X} \otimes \hat{Q}_{1, r_Y})H_2^T = S S^T$ thanks to the mixed-product property of the Kronecker product.
    \item Compute the economy-sized SVD of $S=U_{H_2} \Sigma_{H_2} V_{H_2}^*$ and select an appropriate truncation rank $r_{H_2}$.
    \item Solve equation (3.20b) for $Q_{\mathcal{T}} = Y_{\mathcal{T}}Y_{\mathcal{T}}^T$ with $H_2(\hat{P}_1\otimes\hat{Q}_1)H_2^T$ replaced by its truncation $U_{H_2, r_{H_2}} \Sigma_{H_2, r_{H_2}}^2 U_{H_2, r_{H_2}}^T$ via numerical integration of the adjoint of the linear dynamics $r_{H_2}$ times.
    \item \label{step:QB_BPOD_step} Use the factors $X_{\mathcal{T}}$ and $Y_{\mathcal{T}}$ to construct a truncated balancing transformation for $P_{\mathcal{T}}$ and $Q_{\mathcal{T}}$. In particular, compute the SVD $Y_{\mathcal{T}}^T X_{\mathcal{T}} = U \Sigma V^T$. Use the rank-$r$ truncation of this SVD to construct bi-orthogonal representatives $\Phi = X_{\mathcal{T}} V_r \Sigma_r^{-1/2}$ and $\Psi = Y_{\mathcal{T}} U_r \Sigma_r^{-1/2}$ of the oblique projection $P_{V,W} = \Phi \Psi^T$ as in \cite{Rowley2005model}.
\end{enumerate}
When we implemented this for the jet flow in \cref{sec:jet_flow}, at all SVD steps, with the exception of the final SVD in step~\ref{step:QB_BPOD_step}, we retained at least $98.5\%$ of the variance.

\end{document}